\documentclass[reqno,11pt]{amsart}

\usepackage[top=2.0cm,bottom=2.0cm,left=3cm,right=3cm]{geometry}
\usepackage{amsthm,amsmath,amssymb}
\usepackage{mathrsfs,amsfonts,dsfont,functan,extarrows,mathtools}
\usepackage[numbers,sort&compress]{natbib}
\usepackage[colorlinks]{hyperref}

\usepackage{marginnote}
\usepackage{xcolor}
\usepackage{stmaryrd}
\usepackage{esint}
\usepackage{graphicx}

\usepackage[english]{babel}

\usepackage{concmath}

\newtheorem{theorem}{Theorem}[section]
\newtheorem{proposition}[theorem]{Proposition} 
\newtheorem{definition}[theorem]{Definition} 
\newtheorem{remark}[theorem]{Remark} 
\newtheorem{lemma}[theorem]{Lemma} 

\numberwithin{equation}{section}
\allowdisplaybreaks

\def\ls{\lesssim}

\def\p{\partial}
\def\ep{\epsilon}
\def\eps{\varepsilon}
\def\d{\mathrm{d}}
\def\no{\nonumber}
\def\R{\mathbb{R}}

\def\l{\left\langle}
\def\r{\right\rangle}
\def\J{\mathcal{J}}

\def\M{{\scriptscriptstyle M_0}}
\def\scrpt{\scriptscriptstyle\stackrel}

\newcounter{wronumber}

\newcommand{\abs}[1]{\left|#1\right|}
\newcommand{\nm}[1]{\left\|#1\right\|}

\newcommand{\hx}[1]{{H^ #1 _x}}

\newcommand{\skp}[2]{\left\langle #1,\, #2 \right\rangle}

\newcommand{\sskp}[2]{\left\langle\!\!\!\left\langle #1,\, #2 \right\rangle\!\!\!\right\rangle}

\newcommand{\sskm}[2]{\left\langle\!\!\!\left\langle #1,\, #2 \right\rangle\!\!\!\right\rangle_M}

\begin{document}
\title[SOH-NS Models and limit from SOK Models]
			{Coupled Self-Organized Hydrodynamics and Navier-Stokes models: local well-posedness and the limit from the Self-Organized Kinetic-fluid models}

\author[N. Jiang]{Ning Jiang}
\address[Ning Jiang]{\newline School of Mathematics and Statistics, Wuhan University, Wuhan, 430072, P. R. China}
\email{njiang@whu.edu.cn}

\author[Y.-L. Luo]{Yi-Long Luo}
\address[Yi-Long Luo]
		{\newline School of Mathematics and Statistics, Wuhan University, Wuhan, 430072, P. R. China}
\email{yl-luo@whu.edu.cn}

\author[T.-F. Zhang]{Teng-Fei Zhang}
\address[Teng-Fei Zhang]
        {\newline School of Mathematics and Physics, China University of Geosciences, Wuhan, 430074, P. R. China}
\email{zhtengf@mail2.sysu.edu.cn}
\thanks{$^\dag$ \today}

\maketitle

\begin{abstract}
  A coupled system of self-organized hydrodynamics and Navier-Stokes equations (SOH-NS), which models self-propelled particles in a viscous fluid, was recently derived by Degond et al. \cite{DMVY-2017-arXiv}, starting from a micro-macro particle system of Vicsek-Navier-Stokes model, through an intermediate step of a self-organized kinetic-kinetic model by multiple coarse-graining processes. We first transfer  SOH-NS into a non-singular system by stereographic projection, then  prove the local in time well-posedness of classical solutions by energy method. Furthermore, employing the Hilbert expansion approach, we justify the hydrodynamic limit from the self-organized kinetic-fluid model to macroscopic dynamics. This provides the first analytically rigorous justification of the modeling and asymptotic analysis in \cite{DMVY-2017-arXiv}.
\end{abstract}





\section{Introduction}
\label{sec:intro}

\subsection{Backgrounds}
Self-organized motion is ubiquitous in nature. It corresponds to the formation of large scale coherent structures that emerge from the many interactions between individuals without leaders. Well-known examples are bird flocks, fish schools or insect swarms. Furthermore, self-organization also takes place at the microscopic level, for example in bacterial suspensions and sperm dynamics. In these cases, the environment, typically a viscous fluid, plays a key role in the dynamics.

Recently, Degond \cite{DMVY-2017-arXiv} investigated the self-organized motion of self-propelled particles (which in the literature called ``swimmers") in a viscous Newtonian fluid. The main difficulty in studying these systems comes from the complex mechanical interplay between the swimmers and the fluid. Particularly, high nonlinear interactions occur between neighboring swimmers through the perturbations that their motions create in the surrounding fluid. While the density of the swimmers is high, these interactions are even more complicated. In \cite{DMVY-2017-arXiv}, by assuming the swimmers align their direction of motion, they adopt the Vicsek model \cite{VicsekModel-1995} for self-propelled particles undergoing local alignment to account for these swimmer-swimmer interactions in a phenomenological way. Then the Vicsek model is coupled with the Navier-Stokes  equations for the surrounding viscous fluid by taking into account the interactions between the swimmers and the fluid.

The main contribution of \cite{DMVY-2017-arXiv} is that they provide a coarse-grained description of the hybrid Vicsek-Navier-Stokes dynamics in the form of a fully macroscopic description in both the fluid and the swimmers, which is named ``Self-Organized Hydrodynamics Navier-Stokes" (SOH-NS) equations. In their previous works, the coarse-graining for the Vicsek model alone was the ``Self-Organized Hydrodynamics" (SOH) derived in \cite{DM-2008-M3AS}. The SOH model is a system of continuum equations for the density and mean velocity orientation of the swimmers. In \cite{BCC-2012-AML}, for the first time, a coupled model for the agents' continuum density and mean velocity orientation on the one hand and the fluid velocity and pressure on the other hand is derived. The derivations of both SOH and SOH-NS are based on the Generalized Collision Invariant concept introduced in \cite{DM-2008-M3AS}. This technique has already been successfully applied to a wide range of models inspired by the Vicsek model \cite{DFL-2015-ARMA,DFM-2017-M3AS,JXZ-2016-SIMA}.

The rigorous derivation of macroscopic dynamics establishes a clear link between the microscopic (particle system) and macroscopic (fluid-type equations) scales and, in particular, between the parameters of the two systems. Moreover, microscopic simulations tend to be very costly for large number of individuals while macroscopic simulations are much more cost-effective. However, usually it is hard to derive the macroscopic model directly from the particle system. In stead, the coarse-graining from particle dynamics to macroscopic dynamics is carried out with an intermediate step called the kinetic equation (or mean-field equation). We take the Vicsek model as the example to explain this limiting process. The first step is the coarse-graining from the Vicsek model (which is a particle system) to the ``Self-Organized Kinetic" (SOK) model (which is a kinetic equation). This process is called the mean-field limit. The second step is the coarse-graining from kinetic model to macroscopic model, i.e. from SOK to SOH. This step is in the same spirit of the so-called fluid limits from the Boltzmann equations, which has been a very active research field in the past three decades.

In some simple cases, it has been rigorously justified the mean field limit from the Vicsek model to SOK, see \cite{BCC-2012-AML}. The first convergence result from the SOK to SOH was provided by the authors of the current paper with Xiong in \cite{JXZ-2016-SIMA}, which provided a first rigorous justification of the formal analysis in \cite{DM-2008-M3AS}, based on the well-posedness results of SOH in \cite{DLMP-2013-MAA}.

In general, for the fluid limits from kinetic equations, there are two basic approaches. The first is the moment method, which starts from a sequence of global weak solutions of the scaled kinetic equations, then proves (usually weak) compactness of the several moments of the solutions of the kinetic equations, thus limits of the convergent subsequences are weak solutions to the macroscopic equations. The most famous example of this approach might be the so-called Bardos-Golse-Levermore's (BGL) program that justifies Leray solutions of the incompressible Navier-Stokes equations from renormalized solutions of the Boltzmann equations. This program is finally finished by Golse and Saint-Raymond \cite{G-SRM-2004-Invent}. We emphasize that in this approach, the well-posedness of the limiting macroscopic equations are not needed to be known a priori. It is an automatic consequence of the convergence. However, for the convergence from SOK to SOH, it is extremely difficult to employ this approach because at the current stage, the theory of global weak solutions of the SOK is far from well-understood, although there are some partial results on the homogeneous case \cite{FK-2017-ARMA}. For this reason, the convergence proof from SOK to SOH takes the following second approach, which was initialized by the classic work of Caflisch on the compressible Euler limit from the Boltzmann equation \cite{Cafli-1980-CPAM}.

Different with the moment method approach, the second, i.e. Caflisch's approach is looking for a class of special solutions of the scaled kinetic equations by the Hilbert expansion. It assumes the well-posedness of the (usually classical) solutions the limiting macroscopic equations is known, then construct the solutions of the kinetic equations around that of the limiting equations. Then the convergence is an automatic consequence. The key points of this approach are, first prove the well-posedness of the limiting macroscopic equations, then derive the uniform estimates for the remainder equations which usually are less singular and nonlinear than the original scaled kinetic equations. This approach was employed to justify incompressible Navier-Stokes limit from the Boltzmann equation, see \cite{Guo-2006-CPAM} and \cite{JX-2015-SIMA}. We also use it to justify the convergence from SOK to SOH \cite{JXZ-2016-SIMA}.

The goal of the current paper is to give a rigorous justification of the much harder convergence from the coupled SOK-NS system to the macroscopic SOH-NS system. For the reasons mentioned above, we first use energy method to prove the local-in-time well-posedness of the classical solutions of the SOH-NS equations, then employ Caflisch's Hilbert expansion approach to justify the convergence from SOK-NS to SOH-NS. We will introduce it in details in the following subsections.

\subsection{The SOH-NS model} 
\label{sub:the_soh_ns_model}


In \cite{DMVY-2017-arXiv}, it is formally derived the following coupled ``Self-Organized Hydrodynamics and Navier-Stokes'' (SOH-NS) model:
\begin{equation}\label{SOH-NS-General}
  \begin{cases}
  	\partial_t \rho + \nabla_x \cdot ( \rho U ) = 0 \,,
  \\[3pt]
    \rho \partial_t \Omega + \rho ( V \cdot \nabla_x ) \Omega + \tfrac{a}{\kappa} P_{\Omega^\bot} \nabla_x \rho = \gamma P_{\Omega^\bot} \Delta_x (\rho \Omega) + \rho P_{\Omega^\bot} \big{(} \widetilde{\lambda} S(v) + A(v) \big{)} \Omega \,,
  \\[3pt]
  	|\Omega| = 1 \,,
  \\[3pt]
    \partial_t \big{[} ( Re + \bar{c} \rho ) v + a c_1 \bar{c} \rho \Omega \big{]} + \nabla_x \cdot \big{[} ( Re + \bar{c} \rho ) v \otimes v + a c_1 \bar{c} \rho ( v \otimes \Omega + \Omega \otimes v ) \big{]}
  	\\[3pt]
  		\hspace*{4.2cm} + \nabla_x \cdot \big{[} ( a^2 \bar{c} + b ) \rho \mathcal{Q} (\Omega) \big{]} = - \nabla_x ( p + \tfrac{a^2 \bar{c} }{ 3 } \rho ) + \Delta_x v \,,
  \\[3pt]
      \nabla_x \cdot v = 0 \,,
  \end{cases}
\end{equation}
which governs the dynamics of density $\rho = \rho(t,x) : \mathbb{R}^+ \times \mathbb{R}^3 \rightarrow \R$, the mean motion direction $\Omega = \Omega(t,x) \in \mathbb{S}^{2}$, and the environmental fluid velocity $v = v(t,x) \in \mathbb{R}^3$. In addition, $p = p(t,x) \in \mathbb{R}$ denotes the pressure, and the symbols $\nabla_x$, $\nabla_x \cdot$ and $\Delta_x$ are gradient operator, divergence operator and Laplacian operator, respectively. Moreover, we denote that
	\begin{equation*}
	  U = a c_1 \Omega + v\,, \ V = a c_2 \Omega + v\,, \ \mathcal{Q} (\Omega) = c_4 \big{(} \Omega \otimes \Omega - \tfrac{1}{3} \mathrm{Id} \big{)}\,,
	\end{equation*}
and
	\begin{equation*}
	  (\nabla_x v)_{ij} = \partial_{x_i} v_j \equiv \partial_i v_j \,, \ S(v) = \tfrac{1}{2} ( \nabla_x v + \nabla_x v^\top ) \,, \ A(v) = \tfrac{1}{2} ( \nabla_x v - \nabla_x v^\top ) \,,
	\end{equation*}
and $P_{\Omega^\bot} = \mathrm{I} - \Omega \otimes \Omega$ denotes {\em the orthonormal projection operator} $P_{\Omega^\bot}$ onto the sphere $\mathbb{S}^{2}$ at $\Omega$.

For completeness, we mention here that the coefficients in this system satisfy
\begin{equation}\label{Coeffs-SOH_NS}
  \begin{aligned}
    \gamma =& k_0 \nu ( c_2 + \tfrac{2}{\kappa} ) \,, \ \widetilde{\lambda} = \lambda \lambda_0 \,, \ \lambda_0 = \tfrac{6}{\kappa} c_2 + c_3 - 1 \,, \ \kappa = \tfrac{\nu}{D}\,, \\
    k_0 =& \tfrac{R^2}{6} \int_{\R^n} K (|x|) |x|^2 \d x \bigg{(} \int_{\R^n} K (|x|) \d x \bigg{)}^{-1} \,, \\
     c_1 =& \tfrac{\int_0^\pi \cos \theta \exp (\kappa \cos \theta) \sin \theta \d \theta}{\int_0^\pi \exp (\kappa \cos \theta) \sin \theta \d \theta} \in [0,1]\,, \ c_2 = \tfrac{\int_0^\pi \cos \theta \sin^3 \theta \exp (\kappa \cos \theta) h (\cos \theta) \d \theta}{\int_0^\pi \sin^3 \theta \exp (\kappa \cos \theta) h (\cos \theta) \d \theta} \,, \\
     c_3 = & \tfrac{ 2 \int_0^\pi \cos^2 \theta \sin^3 \theta \exp (\kappa \cos \theta) h (\cos \theta) \d \theta}{\int_0^\pi \sin^3 \theta \exp (\kappa \cos \theta) h (\cos \theta) \d \theta} \,, \ c_4 = 1 - \tfrac{ 3 \int_0^\pi \sin^3 \theta \exp (\kappa \cos \theta) \d \theta}{ 2 \int_0^\pi \sin \theta \exp (\kappa \cos \theta) \d \theta } \,,
  \end{aligned}
\end{equation}
where $a$, $b$, $\nu$, $D$, $\lambda$ and $R$ are the known constants, $Re > 0$ is Reynolds number, $\bar{c}$ is an inertial constant, and $K = K(r) > 0$, $r \geq 0$, is a given sensing function, which weights the influence of the neighboring agents, and $h (\cdot)$ will be defined in Definition \ref{def:GCI} below.

As the inertial constant $\bar{c} \ll 1$ in physical experiment, we take the constant $\bar{c} = 0$ in this paper, and we concern a simple version of SOH-NS system shown in the following,
\begin{equation}\label{SOH-NS}
  \left\{
    \begin{array}{l}
      \partial_t \rho + \nabla_x \cdot ( \rho U ) = 0 \,, \\
      \rho \partial_t \Omega + \rho ( V \cdot \nabla_x ) \Omega + \tfrac{a}{\kappa} P_{\Omega^\bot} \nabla_x \rho = \gamma P_{\Omega^\bot} \Delta_x (\rho \Omega) + \rho P_{\Omega^\bot} \big{(} \widetilde{\lambda} S(v) + A(v) \big{)} \Omega \,, \\
      |\Omega| = 1 \,, \\
      Re \big{(} \partial_t v + v \cdot \nabla_x v \big{)} + \nabla_x p = \Delta_x v - b \nabla_x \cdot \big{(} \rho \mathcal{Q} (\Omega) \big{)} \,, \\
      \nabla_x \cdot v = 0 \,.
    \end{array}
  \right.
\end{equation}

For the SOH-NS system \eqref{SOH-NS}, the operator $P_{\Omega^\bot}$ on the right-hand side of the second equation ensures that the geometric constraint $|\Omega| = 1$ holds at all times (provided that $|\Omega|_{t=0} = 1$).  On the other hand, this operator makes the equation is not conservative, which means that the terms involving the spatial derivatives cannot be written as spatial divergence of a flux function.

To write the equation of $\Omega$ into coordinates, it seems that a natural choice is the spherical coordinates, i.e. $\Omega = ( \sin \theta \cos \varphi , \sin \theta \sin \varphi , \cos \theta )^\top$, which implies that the SOH-NS system \eqref{SOH-NS} can be written as
\begin{equation}\label{SOH-NS-3D-S-Simple}
    \left\{
      \begin{array}{c}
        \partial_t \widehat{\rho} + U \cdot \nabla_x \widehat{\rho} + a c_1 ( \Omega_\theta \cdot \nabla_x \theta + \Omega_\varphi \cdot \nabla_x \varphi ) = 0 \,, \\
        \partial_t \theta + V \cdot \nabla_x \theta + \mathcal{D}_\theta (\widehat{\rho}, \theta, \varphi, v) = 0 \,, \\
        \partial_t \varphi + V \cdot \nabla_x \varphi + \tfrac{1}{\sin^2 \theta} \mathcal{D}_\varphi (\widehat{\rho}, \theta, \varphi, v) = 0 \,, \\
        Re [ \partial_t  v   + v \cdot \nabla_x  v ] - \Delta_x v + \nabla_x p  + b  e^{\widehat{\rho}}  \mathcal{B} (\widehat{\rho}, \theta, \varphi) = 0 \,, \\
        \nabla_x \cdot v = 0 \,,
      \end{array}
    \right.
  \end{equation}
  where $\widehat{\rho} = \ln \rho$ and
   \begin{equation*}
     \begin{aligned}
       \mathcal{D}_\theta (\widehat{\rho}, \theta, \varphi, v) = & \tfrac{a}{\kappa} \Omega_\theta \cdot \nabla_x \widehat{\rho}  - \Big{(} \widetilde{\lambda} S(v) + A(v) \Big{)} : \Omega_\theta \otimes \Omega \\
       & - 2 \gamma \nabla_x \widehat{\rho} \cdot \nabla_x \theta - \gamma ( \Delta_x \theta - \sin \theta \cos \theta | \nabla_x \varphi |^2 ) \,, \\
       \mathcal{D}_\varphi (\widehat{\rho}, \theta, \varphi, v) = & \tfrac{a}{\kappa } \Omega_\varphi \cdot \nabla_x \widehat{\rho}  -  \Big{(} \widetilde{\lambda} S(v) + A(v) \Big{)} : \Omega_\varphi \otimes \Omega \\
       & - 2 \gamma \sin^2 \theta  \nabla_x \widehat{\rho} \cdot \nabla_x \varphi - \gamma ( \sin^2 \theta  \Delta_x \varphi + \sin \theta \cos \theta  \nabla_x \theta \cdot \nabla_x \varphi ) \,, \\
       \mathcal{B} (\widehat{\rho}, \theta, \varphi) = & \nabla_x \widehat{\rho} \cdot \mathcal{Q} (\Omega) + c_4 [ \Omega \cdot \nabla_x \theta \Omega_\theta + \Omega \cdot \nabla_x \varphi \Omega_\varphi + ( \Omega_\theta \cdot \nabla_x \theta + \Omega_\varphi \cdot \nabla_x \varphi ) \Omega ] \,,
     \end{aligned}
   \end{equation*}
  and the vectors $\Omega_\theta$ and $\Omega_\varphi$ are the partial derivative of $\Omega$ with respect to the variables $\theta$ and $\varphi$, which, explicitly, is
  $$ \Omega_\theta = ( \cos \theta \cos \varphi , \cos \theta \sin \varphi , - \sin \theta)^\top \,, \ \Omega_\varphi = ( - \sin \theta \sin \varphi , \sin \theta \cos \varphi , 0 )^\top \,. $$
 Here we omit the details of the derivation. Actually, one can refer to an analogous derivation in \cite{DLMP-2013-MAA} (see also in \cite{JZ-2017-NA}) for the SOH models in three dimensions. However, the system \eqref{SOH-NS-3D-S-Simple} has a coefficient $\frac{1}{\sin^2 \theta}$ which is singular near $\theta = 0$. This will bring some serious difficulties in analytic study. In fact, this degeneracy results from the orthonormal projection $P_{\Omega^\bot}$ under the spherical coordinates transform, specifically, $P_{\Omega^\bot} a = ( \Omega_\theta \cdot a ) \Omega_\theta + \frac{(\Omega_\varphi \cdot a)}{ \sin^2 \theta } \Omega_\varphi $ for all $a \in \R^3$.

In order to avoid this degeneracy, we adopt {\em stereographic projection transform} to deal with the geometric constraint $|\Omega| = 1$. Let $\widehat{\rho} = \ln \rho$, and $\Omega = \big{(} \tfrac{2 \phi}{W} , \tfrac{2 \psi}{W} , \tfrac{\phi^2 + \psi^2 - 1}{W} \big{)}$, where $W = 1 + \phi^2 + \psi^2$ and $(\phi, \psi) \in \R^2$, then the SOH-NS system \eqref{SOH-NS} reads
  \begin{equation}\label{SOH-NS-3D-SPT-Smp}
    \left\{
      \begin{array}{c}
        \partial_t \widehat{\rho} + U \cdot \nabla_x \widehat{\rho} + a c_1 ( \Omega_\phi \cdot \nabla_x \phi + \Omega_\psi \cdot \nabla_x \psi ) = 0 \,, \\
        \partial_t \phi + V \cdot \nabla_x \phi + \mathcal{H}_\phi ( \widehat{\rho}, \phi, \psi ) = 0 \,, \\
        \partial_t \psi + V \cdot \nabla_x \psi + \mathcal{H}_\psi ( \widehat{\rho}, \phi, \psi ) = 0 \,, \\
        Re ( \partial_t  v + v \cdot \nabla_x  v ) - \Delta_x v + \nabla_x  p   + b  e^{\widehat{\rho}} \mathcal{G} (\widehat{\rho}, \phi, \psi) = 0 \,, \\
        \nabla_x \cdot v = 0 \,.
      \end{array}
    \right.
  \end{equation}
  where the symbols $\mathcal{H}_\phi$, $\mathcal{H}_\psi$, $\mathcal{P}$ and $\mathcal{G}$ stand for
  \begin{equation}\label{Terms-Def}
    \begin{aligned}
      \mathcal{H}_\phi (\widehat{\rho}, \phi, \psi, v) = & \tfrac{a}{4 \kappa} W^2 \Omega_\phi \cdot \nabla_x \widehat{\rho} - \gamma \Delta_x \phi - 2 \gamma \nabla_x \widehat{\rho} \cdot \nabla_x \phi  + \tfrac{2 \gamma \phi}{W} |\nabla_x \phi|^2 \\
      & + \tfrac{ 4 \gamma \psi }{ W } \nabla_x \phi \cdot \nabla_x \psi - \tfrac{2 \gamma \phi}{W} |\nabla_x \psi|^2 - \tfrac{1}{4} W^2 (\widetilde{\lambda} S(v) + A(v) ) : \Omega_\phi \otimes \Omega \,, \\
      \mathcal{H}_\psi ( \widehat{\rho}, \phi, \psi, v ) = & \tfrac{a}{4 \kappa} W^2 \Omega_\psi \cdot \nabla_x \widehat{\rho} - \gamma \Delta_x \psi - 2 \gamma \nabla_x \widehat{\rho} \cdot \nabla_x \psi  - \tfrac{2 \gamma \psi}{W} |\nabla_x \phi|^2 \\
      & + \tfrac{ 4 \gamma \phi }{W} \nabla_x \phi \cdot \nabla_x \psi + \tfrac{2 \gamma \psi}{W} |\nabla_x \psi|^2 - \tfrac{1}{4} W^2 (\widetilde{\lambda} S(v) + A(v) ) : \Omega_\phi \otimes \Omega \,, \,, \\
       \mathcal{G} ( \widehat{\rho}, \phi, \psi ) = & \nabla_x \widehat{\rho} \cdot \mathcal{Q}(\Omega) + c_4 [ (\Omega \cdot \nabla_x \phi) \Omega_\phi + (\Omega \cdot \nabla_x \psi) \Omega_\psi ] \\
       & + c_4 ( \Omega_\phi \cdot \nabla_x \phi + \Omega_\psi \cdot \nabla_x \psi ) \Omega \,,
    \end{aligned}
  \end{equation}
and the vectors $\Omega_\phi$ and $\Omega_\psi$ are the partial derivative of $\Omega$ with respect to the variables $\phi$ and $\psi$, which, explicitly, is
  $$
    \Omega_\phi = \big{(} \tfrac{2 (1 - \phi^2 + \psi^2)}{W^2} , - \tfrac{4 \phi \psi}{W^2} , \tfrac{4 \phi}{W^2} \big{)} \ \textrm{and} \ \Omega_\psi = \big{(} - \tfrac{4 \phi \psi}{W^2} , \tfrac{2 (1 + \phi^2 - \psi^2)}{W^2} , \tfrac{4 \psi}{W^2} \big{)} \,,
  $$
respectively. The details of the derivations will be given in Appendix \ref{sec:SPT-3D}.

The main theorem of local well-posedness of SOH-NS system \eqref{SOH-NS} is stated as follows:
\begin{theorem}[Local well-posedness] \label{Thm-WP-SOHNS-3D}

Assume $ \gamma > 0$. If the initial conditions of the SOH-NS system \eqref{SOH-NS} are given by
  \begin{equation}\label{Initial-3D}
    \begin{aligned}
      ( \rho, \Omega, v ) \big{|}_{t=0} = \Big{(} \rho^{in}(x), \Omega^{in}(x) = (\tfrac{2 \phi^{in}(x)}{W^{in}(x)} , \tfrac{2 \psi^{in}(x)}{W^{in}(x)} , 1 - \tfrac{2}{W^{in}(x)} )^\top , v^{in}(x) \Big{)} \in \R \times \mathbb{S}^2 \times \R^3 \,,
    \end{aligned}
  \end{equation}
where $W^{in}(x) = 1 + ( \phi^{in}(x) )^2 + ( \psi^{in}(x) )^2$, $0 < \rho^{in}(x) \leq \bar{\rho} < \infty$, $(\phi^{in}(x), \psi^{in}(x) ) \in \R^2$ and $\ln \rho^{in}, \phi^{in}, \psi^{in}, v^{in} \in H^s(\R^3)$ for $s \geq 3$. Then there is a time $T > 0$, depending only on the initial data $\rho^{in}$, $\Omega^{in}$, $v^{in}$ and all of coefficients of the system \eqref{SOH-NS}, such that the SOH-NS system \eqref{SOH-NS} admits a unique solution $\Big{(} \rho, \Omega = ( \tfrac{2 \phi}{W}, \tfrac{2 \psi}{W} , 1 - \tfrac{2}{W} )^\top , v \Big{)}$ satisfying
  $$
    \ln \rho \in L^\infty ( 0,T; H^s(\R^3) ), \, \phi, \psi, v \in L^\infty(0,T;H^s(\R^3)) \cap L^2(0,T; H^{s+1}(\R^3))\,,
  $$
where $W = 1 + \phi^2 + \psi^2$. Moreover, the the following energy bound
  \begin{equation}\label{3D-Eng-Bnd}
    \begin{aligned}
      \sup_{t \in [0,T]} \Big{(} & \| \ln \rho \|^2_{H^s(\R^3)} + \| \phi \|^2_{H^s(\R^3)} + \| \psi \|^2_{H^s(\R^3)} + Re \| v \|^2_{H^s(\R^3)} \Big{)} \\
      & + \gamma \| \nabla_x \phi \|^2_{L^2(0,T; H^s(\R^3))} + \gamma \| \nabla_x \psi \|^2_{L^2(0,T; H^s(\R^3))} + \| \nabla_x v \|^2_{L^2(0,T; H^s(\R^3))} \leq C
    \end{aligned}
  \end{equation}
holds for some positive constant $C$, which depends only on the initial data and all of coefficients of the system \eqref{SOH-NS}.
\end{theorem}

\subsection{The hydrodynamic limit from the SOK-NS model} 
\label{sub:the_hydrodynamic_limit_from_the_sok_ns_model}


Based on the above local existence result of the SOH-NS system \eqref{SOH-NS}, we turn to describe the second part of this paper, i.e. the hydrodynamic limits from the ``self-organized kinetic equation and Navier-Stokes'' (SOK-NS) coupling system to the above SOH-NS system \eqref{SOH-NS} in $\R^3$. Let us first briefly introduce that SOK-NS can be formally derived from the following coupled Vicsek-Navier-Stokes (Vicsek-NS) model for particle system.

Let $X_i(t)\in\mathbb{R}^3\ (i \in \{ 1,2,\cdots, N \})$ and $\omega_i(t)\in \mathbb{S}^2$ be the position and its direction of motion of the $i\mbox{-}$th particle at time $t$, where the large number $N$ denotes the total number of particles contained in the system. Besides, denote by $v(t,x)\in \mathbb{R}^3$ the environmental fluid velocity. Thus the coupling Vicsek-Navier-Stokes dynamics can be expressed in the following system:
  \begin{align} \label{eq:particle-model}
    \begin{cases}
      \frac{\d X_i}{\d t} = u_i, \\[3pt]
      \frac{\d u_i}{\d t}= \eta (v(X_i,t) + a \omega_i -u_i),
    \\[3pt]
      \d\omega_i = P_{\omega_i^\perp} \circ [\nu \overline{\omega_i}\d t + ( \lambda S(v)+A(v) )\omega_i \d t]+ \sqrt{2D}\d B_t^i,
    \\[3pt] \displaystyle
      \bar\omega_i = \frac{J_i}{|J_i|} \text{ \quad with } J_i = \sum_{k=1}^N K \left( \tfrac{|X_i-X_{k}|}{R} \right) \omega_k,
    \\[3pt]
      \frac{\d u_i}{\d t} = F_i,
    \\[3pt] \displaystyle
      Re(\partial_{t}v + (v\cdot\nabla_{x}) v)  + \nabla_{x} p
      = \Delta_{x} v - c \sum_{i=1}^{N}F_i\delta_{X_i(t)} \\[5pt] \displaystyle
        \hspace*{5.5cm} - b  \frac{1}{N} \sum_{i=1}^N \left( \omega_i \otimes \omega_i - \tfrac{1}{3} {\rm Id} \right)  \nabla_x\delta_{X_i(t)},
    \\[2pt]
      \nabla_{x} \cdot v = 0,
    \end{cases}
  \end{align}
where $p(t,x)$ is the pressure, and $b$ is constant. $R$ stands for the non-dimensional variable of interaction range. It should be pointed out that the third equation in \eqref{eq:particle-model} is actually a stochastic differential equation, and the symbol $\circ$ implies that the SDE is taken in the Stratonovich sense.

By using of the large friction limit regime as stated in \cite{DMVY-2017-arXiv}, the mean-field equation at finite Reynolds and finite particle inertia can be obtained. The mean-field model governs the evolution of fluid velocity $v(t,x)$ and the one-particle distribution function $f(t,x,\omega)$, which represents the distribution of the microscopic molecules at the spatial position $x \in \R^3$ with the motion direction $\omega \in \mathbb{S}^2$ at time $t \geq 0$. This is referred in \cite{DMVY-2017-arXiv} as the ``self-organized kinetic equation and Navier-Stokes'' (SOK-NS) coupling system, and we omit here the derivation. We now turn to consider the scaled system for SOK-NS model with respect to a rescaling parameter $\eps$:
	\begin{equation}\label{eq:SOK-NS}
	\begin{cases}
	  \eps \big[ \partial_t f^\eps + \nabla_x \cdot (u^\eps f^\eps)  + \nabla_\omega \cdot ( \mathcal{F}_{(f^\eps, v^\eps)} f^\eps ) \big] = \mathcal{Q}(f^\eps) \,,
	\\[3pt]
	  u^\eps(t, x, \omega) = v^\eps(t,x) + a \omega \,,
	\\[3pt]
	  Re ( \partial_t v^\eps + v^\eps \cdot \nabla_x v^\eps ) + b \nabla_x \cdot Q_{f^\eps} = - \nabla_x p^\eps + \Delta_x v^\eps \,,
	\\[3pt]
	  \nabla_x \cdot v^\eps = 0 \,,
	\end{cases}
	\end{equation}
with
\begin{equation*}
  \begin{array}{l}
    \mathcal{Q}(f^\eps) = - \nabla_\omega \cdot [ \nu P_{\omega^\bot} \Omega_{f^\eps} f^\eps ] + D \Delta_\omega f^\eps \,,
  \\[3pt]
    Q_{f^\eps}(t,x) = \int_{\mathbb{S}^2} \big( \omega \otimes \omega - \tfrac{1}{3} \mathrm{Id} \big) f^\eps(t,x,\omega) \d \omega \,,
  \\[3pt]
    \mathcal{F}_{(f^\eps,v^\eps)} = P_{\omega^\bot} \left[ \nu \tfrac{k_0}{|j_{f^\eps}|} P_{\Omega_{f^\eps}^\bot} \Delta_x j_{f^\eps} + ( \lambda S(v^\eps) + A(v^\eps) ) \omega \right] \,,
  \end{array}
\end{equation*}
where $j_f(t,x) = \int_{\mathbb{S}^2} \omega f(t,x,\omega) \d \omega$ is the local current density and $\Omega_f (t,x) = \tfrac{j_f}{|j_f|} \in \mathbb{S}^2$ is the local average orientation. It should be pointed out that in the above system \eqref{eq:SOK-NS}, the operator $\mathcal{Q}(f)$ is actually a Fokker-Planck operator, in which the last diffusion term denotes the Brownian noise in particle directions, in fact, the operator $\Delta_\omega$ stands for the Laplace-Beltrami operator on the sphere since we have $\omega \in \mathbb{S}^2$.

In the above SOK-NS system \eqref{eq:SOK-NS}, the scaling parameter $\eps$ indicates the long-time scaling $t'= t/\eps$. As $\eps$ goes to zero, a formal coarse-graining process stated in \cite{DMVY-2017-arXiv} shows that the hydrodynamic limits of the SOK-NS \eqref{eq:SOK-NS} is exactly the SOH-NS system \eqref{SOH-NS} (see Theorem 6.7 in page 38 of \cite{DMVY-2017-arXiv}). The main challenge of deriving the macroscopic equations is the lack of conservation laws for the self-organized kinetic models. To overcome this difficulty, the \emph{Generalized Collision Invariants} (GCI) are employed in \cite{DM-2008-M3AS} to derive the macroscopic equations. For the sake of clarity and completeness, we list here some basic properties of the self-organized kinetic (SOK) model, for more details the readers are refereed to \cite{DFL-2015-ARMA,DFLMN-2013-Schwartz,Frouvl-2012-M3AS} and references therein.

Firstly we introduce the equilibrium of $\mathcal{Q}$, which are expressed by the von Mises-Fisher (VMF) distributions with respect to the local mean orientation $\Omega \in \mathbb{S}^{n-1}$ for $n$-dimension case, namely,
	\begin{align}\no
	  \mathbb{E} = & \{ f |\ \mathcal{Q}(f)=0 \}
	 = \{\omega \mapsto \rho M_\Omega (\omega),\ \forall \rho \in \mathbb{R}_+,\ \Omega \in \mathbb{S}^{n-1} \},
	\end{align}
where the VMF distribution is defined as
	\begin{align}
	  M_\Omega (\omega) = Z^{-1} \exp(\kappa \omega \cdot \Omega)
	\end{align}
with a constant $Z=\int_{\omega \in \mathbb{S}^{n-1}} \exp(\kappa \omega \cdot \Omega)\d \omega$ independent of $\Omega$. The VMF distribution enjoys the following properties:
	\begin{itemize}

	  \item[i)] $M_\Omega (\omega)$ is a probability density, i.e., $\int_{\omega \in \mathbb{S}^{n-1}} M_\Omega (\omega)\d v=1$;

	  \item[ii)] The first moment of $M_\Omega (\omega)$ satisfies
	  $\int_{\omega \in \mathbb{S}^{n-1}} \omega M_\Omega (\omega)\d \omega=c_1 \Omega, $
	  where the coefficient $c_1 \in [0,1]$ defined before denotes the order parameter in the study of phase transitions.

	\end{itemize}

Note that the formula $P_{\omega^\bot} \Omega =\nabla_\omega (\omega \cdot \Omega)$ ensures that the collision operator $\mathcal{Q}$ can be rewritten as
	\begin{align*}
	  \mathcal{Q}(f) = D \nabla_\omega \cdot \left( M_{\Omega_f} \nabla_\omega ( \tfrac{f}{M_{\Omega_f}} ) \right),
	\end{align*}
which results in a dissipation relation
	\begin{align}
	    \int_{\omega \in \mathbb{S}^{n-1}} \mathcal{Q}(f) \tfrac{f}{M_{\Omega_f}} \d \omega
	  = -D \int_{\omega \in \mathbb{S}^{n-1}} \abs{\nabla_\omega (\tfrac{f}{M_{\Omega_f}})}^2 M_{\Omega_f} \d \omega
	  \le 0.
	\end{align}
This implies that $\mathcal{Q}(f)=0$ is equivalent to $f \in \mathbb{E}$.

As mentioned above, one of the main difficulties to derive the macroscopic equations of the SOK model is that it obeys only the conservation law of mass. To recover the missing momentum conservation related to the quantity $\Omega(t,x)$, Degond-Motsch introduce the concept of the \emph{Generalized Collision Invariants} (GCI) in \cite{DM-2008-M3AS}.
	\begin{definition}[\cite{DM-2008-M3AS}] \label{def:GCI}
	  For any given $\Omega \in \mathbb{S}^{n-1}$, the linearized collision operator $\mathcal{L}_{\Omega}$ is defined as
		  \begin{align}
		      \mathcal{L}_{\Omega} f \triangleq \Delta_{\omega} f- \kappa \nabla_{\omega} \cdot (P_{\omega^\bot}\Omega f)
		    = \nabla_{\omega} \cdot \left(M_{\Omega} \nabla_{\omega} (\tfrac{f}{M_{\Omega}})\right).
		  \end{align}
	  The Generalized Collision Invariants (GCI) are the elements in the null space of $\mathcal{L}_{\Omega}$:
	  \begin{align}
	      \mathcal{N}(\mathcal{L}_{\Omega}) \triangleq & \left\{ \psi \Big| \int_{v \in \mathbb{S}^{n-1}} \mathcal{L}_{\Omega}f\, \psi \,\d \omega=0,\ \forall f \textrm{ such that } \Omega_f= \Omega \right\} \\\no
	      =& \left\{ \psi|\ \mathcal{L}_{\Omega}^* \psi(v) = A\cdot \epsilon \textrm{ with } A\cdot \Omega =0 \right\} \\\no
	      =& \left\{ \omega \mapsto h(\omega\cdot\Omega)A \cdot \omega +C \textrm{ with } C \in \mathbb{R},\ A\in \mathbb{R}^n, \textrm{ and } A\cdot \Omega =0 \right\},
	  \end{align}
	  where the operator $\mathcal{L}_{\Omega}^*$ is the adjoint of the linearized operator $\mathcal{L}_{\Omega}$, which takes the form,
	  \begin{align*}
	      \mathcal{L}_{\Omega}^* \psi = -\Delta_{\omega} \psi - \kappa \Omega \cdot \nabla_{\omega} \psi
	    = -\tfrac{1}{M_\Omega} \nabla_{\omega} \cdot (M_\Omega \nabla_{\omega} \psi).
	  \end{align*}
	  Here, $h(\omega \cdot \Omega)=h(\cos \theta)$ is some certain function satisfying that $h (\mu) = ( 1 - \mu^2 )^{- \frac{1}{2}} g(\mu) \geq 0$ for $g$ being the unique solution of the differential equation
\begin{equation*}
  - ( 1 - \mu^2 ) \partial_\mu \big{(} e^{\kappa \mu} (1-\mu^2) \partial_\mu g \big{)} + e^{\kappa \mu} g = - (1-\mu^2)^\frac{3}{2} e^{\kappa \mu}
\end{equation*}
in the weighted $H^1$ space $\mathcal{V}$ given by
$$ \mathcal{V} = \big{\{} g \ | \ ( 1 - \mu^2 )^{- \frac{1}{2}} g \in L^2 (-1, 1)\,, \ ( 1 - \mu^2 )^{\frac{1}{2}} \partial_\mu g \in L^2 (-1,1) \big{\}} \,.$$

	\end{definition}

To justify the SOH-NS from SOK-NS in the context of local-in-time classical solutions which we obtained for SOH-NS in Theorem \ref{Thm-WP-SOHNS-3D}, we take the following ansatz for the scaled SOK-NS system \eqref{eq:SOK-NS} with respect to the square root of scaling parameter $\eps$:
  \begin{align}\label{eq:expansion}
    f^\eps = f_0 + \sqrt\eps f_R^\eps = \rho_0 M_{\Omega_0} + \sqrt\eps f_R^\eps, \quad\! \mbox{and} \quad\!
    v^\eps = v_0 + \sqrt\eps v_R^\eps,
  \end{align}
where $(\rho_0, \Omega_0, v_0)$ be the solution to the Cauchy problem of the SOH-NS system. Moreover, we require that $\nabla_x \cdot v_0 =0$, and $\int_{\mathbb{S}^2} f^\eps \d \omega = \int_{\mathbb{S}^2} f_0 \d \omega$, $j_{f^\eps} = j_{f_0}$, which yield that $\nabla_x \cdot v_R^\eps = 0$, and
	\begin{align}
		f_R^\eps \in \Phi_0 = \left\{ \phi \in L^1(\mathbb{S}^2)| \int_{\mathbb{S}^2} \phi \d \omega = 0, \ \int_{\mathbb{S}^2} \omega \phi \d \omega = 0  \right\}.
	\end{align}

Then, the remainder $(f_R^{\eps},\ v_R^{\eps})$ satisfies
  \begin{align} \label{eq:remainder} 
    \begin{cases}
      \p_t f_R^{\eps} + \nabla_x\cdot(u_0 f_R^{\eps}) + \nabla_{\omega}\cdot(\mathcal{F}_0 f_R^{\eps})
      = \frac{1}{\eps}\mathcal{L}_{\Omega_0} f_R^{\eps} - \frac{1}{\sqrt\eps} h_0 - h_1 ,
    \\[3pt]
      Re (\p_t v_R^\eps + v_0 \cdot\! \nabla_x v_R^\eps + v_R^\eps \cdot \nabla_x v_0 + \sqrt\eps v_R^\eps \cdot \nabla_x v_R^\eps)
      + b \nabla_x\cdot\! Q_{f_R^\eps} \!=\! -\nabla_x p_R^\eps + \Delta_x v_R^\eps,
    \\[3pt]
      \nabla_x \cdot v_R^\eps =0,
    \end{cases}
  \end{align}
where $\mathcal{L}_{\Omega_0} f = \nabla_\omega \cdot (M_{\Omega_0} \nabla_\omega (\frac{f}{M_{\Omega_0}}))$, and
  \begin{align}
    h_0 =\ & \partial_t f_0 + u_0 \cdot \nabla_x f_0 + \nabla_\omega \cdot (\mathcal{F}_0 f_0), \\
    h_1 =\ & v_R^\eps \cdot \nabla_x (f_0 + \sqrt\eps f_R^\eps) + \nabla_\omega \cdot [\mathcal{P}_{\omega^\perp}B(v_R^\eps) \omega (f_0 + \sqrt\eps f_R^\eps)] .
  \end{align}

We now state our main result of the hydrodynamic limit.
\begin{theorem}[Hydrodynamic Limit] \label{thm:limit}
  Let $s\ge 2,\ m>s+4$. Assume $(\rho_0, \Omega_0, v_0)$ be the solution to the Cauchy problem of the SOH-NS system \eqref{SOH-NS} provided by Theorem \ref{Thm-WP-SOHNS-3D} with $H^m$-initial data $(\rho_0^{in}, \Omega_0^{in}, v_0^{in})$. Furthermore, assume
	  \begin{align}
	    f^{\eps,in}(x,\omega) =\ & \rho_0^{in}(x) M_{\Omega_0^{in}(x)}(\omega) + \sqrt\eps f_R^{\eps,in}(x,\omega), \\
	    v^{\eps,in}(x) =\ & v_0^{in}(x) + \sqrt\eps v_R^{\eps,in}(x),
	  \end{align}
  with the bound $\|f_R^{\eps,in}(x,\omega)\|_{H^s_{x,\omega}} \le C$ and $\|v_R^{\eps,in}(x)\|_{H^s_{x}} \le C$.

  Then there exists an $\eps_0 >0$ such that, for all $\eps \in (0,\eps_0)$, the coupling SOK-NS system \eqref{eq:SOK-NS} admits a unique solution $(f^\eps, v^\eps) \in L^\infty ([0,T];\ H^s_{x,\omega} \times H^s_x)$ in the class that $f^\eps \in \Phi_0$ and $\nabla_x \cdot v_R^\eps = 0$, being of the form
	  \begin{align}
	    f^\eps(t,x,\omega) =\ & \rho_0(t,x) M_{\Omega_0(t,x)}(\omega) + \sqrt\eps f_R^\eps(t,x,\omega), \\
	    v^\eps(t,x) =\ & v_0(t,x) + \sqrt\eps v_R^\eps(t,x),
	  \end{align}
  with $\|v_R^\eps\|_{H^s_x} + \|f_R^\eps\|_{H^s_{x,\omega}} \le C,$ where $C$ is independent of $\eps$ and $t\in[0,T]$.
\end{theorem}

\begin{remark}
	In fact, we can make a general expansion ansatz for the scaled SOK-NS system \eqref{eq:SOK-NS}
		\begin{align}
    f^\eps = f_0 + \eps^\alpha f_R^\eps = \rho_0 M_{\Omega_0} + \eps^\alpha f_R^\eps, \quad\! \mbox{and} \quad\!
    v^\eps = v_0 + \eps^\beta v_R^\eps, 	
		\end{align}	
	then the remainder equation can be expressed in a more general form:
	  \begin{align} \label{eq:remainder-general} 
    \begin{cases}
      \p_t f_R^{\eps} + \nabla_x\cdot(u_0 f_R^{\eps}) + \nabla_{\omega}\cdot(\mathcal{F}_0 f_R^{\eps})
      = \frac{1}{\eps}\mathcal{L}_{\Omega_0} f_R^{\eps} - \frac{1}{\eps^\alpha} h_0 - \eps^{\beta-\alpha} \bar h_1 ,
    \\[3pt]
      Re (\p_t v_R^\eps + v_0 \cdot \nabla_x v_R^\eps + v_R^\eps \cdot \nabla_x v_0 + \eps^\beta v_R^\eps \cdot \nabla_x v_R^\eps)
    \\[2pt]
      \hspace*{5cm} + \eps^{\alpha-\beta} b \nabla_x\cdot Q_{f_R^\eps} = - \nabla_x p_R^\eps + \Delta_x v_R^\eps,
    \\[3pt]
      \nabla_x \cdot v_R^\eps =0,
    \end{cases}
  \end{align}	
where $ \bar h_1 =v_R^\eps \cdot \nabla_x (f_0 + \eps^\alpha f_R^\eps) + \nabla_\omega \cdot [\mathcal{P}_{\omega^\perp}B(v_R^\eps) \omega (f_0 + \eps^\alpha f_R^\eps)]$. Through the proof below, we can obtain the relation $\alpha\le 1/2$ and $\alpha - \beta \in [0,1/2]$. The choose here $\alpha=\beta=1/2$ ensures actually the \emph{optimal} expansion index in the sense that the remainder functions $f^\eps_R$ and $v^\eps_R$ are bounded in some function spaces.
\end{remark}

\begin{remark}
  Note that we have improved the previous results \cite{JXZ-2016-SIMA}, which established the rigorous analysis of hydrodynamic limits from the self-organized kinetic equations to the self-organized hydrodynamic equations. To apply the weighted Poincar\'e inequalities, some higher-order moment assumptions are introduced in \cite{JXZ-2016-SIMA}. In the present paper, by choosing some better weight function, we work in some appropriate functional spaces such that there is no need to make the previous assumptions any more. In fact, our proof are performed in the weighted Sobolev spaces like $\| \frac{\cdot}{\sqrt{M_{\Omega}}} \|$ (see \S 4-5 below), due to the reason that the von Mises-Fisher function $M_{\Omega(t,x)}(\omega)$ is dependent on the spatial variables.

\end{remark}


The organization of this paper is as follows: The next section is devoted to the \emph{a priori} estimate for the SOH-NS system \eqref{SOH-NS}, under which we obtain the local well-posedness result of the SOH-NS system \eqref{SOH-NS}, and thus complete the proof of Theorem \ref{Thm-WP-SOHNS-3D}. To prove the hydrodynamic limits from SOK-NS to SOH-NS stated in Theorem \ref{thm:limit}, we establish in section 4 the uniform-in-$\eps$ \emph{a priori} estimate for the remainder equation \eqref{eq:remainder}, which stated in Lemma \ref{lemm:apriori-uniform}. Then based on Theorem \ref{Thm-WP-SOHNS-3D} and Lemma \ref{lemm:apriori-uniform}, we can finally prove the convergence result in the last section. For the convenience of readers, we write explicitly in appendix the transform from the original SOH-NS system \eqref{SOH-NS} to the system \eqref{SOH-NS-3D-SPT-Smp} by the stereographic projection transform.

\bigskip
\noindent{\bf Notations.} For notational simplicity, we denote by $\skp{\cdot}{\cdot}$ the usual $L^2\mbox{-}$inner product in variables $x$, by $\|\cdot\|_{L^2_x}$ its corresponding $L^2$-norm, and by $\|\cdot\|_\hx{s}$ the higher-order derivatives $H^s\mbox{-}$norm in variables $x$.

On the other hand, when we consider the spatial variables $x$ and microscopic variables $\omega$ at the same time, it is convenient for us to introduce the weighted Sobolev spaces. For that, we denote by $\langle\!\!\!\langle \cdot,\, \cdot \rangle\!\!\!\rangle$ the standard $L^2$ inner product in both variables $x$ and $\omega$, and by $\nm{\cdot}_{L^2_{x,\omega}}$ its corresponding norm, then we can define the weighted $L^2$ inner product that
	\begin{align*}
		\sskm{f}{g} \triangleq \left\langle\!\!\!\left\langle f,\, g M \right\rangle\!\!\!\right\rangle
			= \iint_{\Omega \times \mathbb{R}^3} fg \,M \d \omega\d x,
	\end{align*}
for any pairs $f(x,\omega), g(x,\omega) \in L^2_{x,\, \omega}$. We also use $\nm{\cdot}_M$ to denote the weighted $L^2$-norm with respect to the measure $M \d \omega\d x$.

Let $\alpha=(\alpha_1,\, \alpha_2,\, \alpha_3) \in \mathbb{N}^3 $ be a multi-index with its length defined as $\textstyle |\alpha| = \sum_{i=1}^3 \alpha_i $. We also use the notation $\nabla_x^k$ to denote the multi-derivative operator $\nabla_x^\alpha = \p_{x_1}^{\alpha_1} \p_{x_2}^{\alpha_2} \p_{x_3}^{\alpha_3}$ with $|\alpha|=k$.

At last, we mention that the notation $A \ls B$ will be used in the following texts to indicate that there exists some constant $C>0$ such that $A \le C B$. Furthermore, the notation $A \sim B$ means that the terms of both sides are equivalent up to a constant, namely, there exists some constant $C$ such that $ C^{-1} B \le A \le C B$.


\section{A Priori Estimates for SOH-NS System}
\label{sec:Apriori-Est-3D}

In this section, we derive \emph{a priori} estimates of the SOH-NS system \eqref{SOH-NS}. In order to overcome the inconvenience resulted from the geometric constraint $|\Omega| = 1$, we consider the system \eqref{SOH-NS-3D-SPT-Smp}, which is taken the stereographic projection transforms to deal with the condition $|\Omega| = 1$. We first define the energy functionals
\begin{equation*}
  \begin{aligned}
    \mathcal{E}(t) = & \| \widehat{\rho} \|^2_{H^s} + \| \phi \|^2_{H^s} + \| \psi \|^2_{H^s} + Re \| v \|^2_{H^s} \,, \\
    \mathcal{D} (t) = & \gamma \| \nabla_x \phi \|^2_{H^s} + \gamma \| \nabla_x \psi \|^2_{H^s} + \| \nabla_x v \|^2_{H^s} \,.
  \end{aligned}
\end{equation*}
Thus we can derive the following \emph{a priori} estimates for the SOH-NS system \eqref{SOH-NS-3D-SPT-Smp}.

\begin{proposition}\label{Prop-Apr-Est-3D}
  Let $s \geq 3$ and $Re , \gamma > 0$. Assume that $( \widehat{\rho}, \phi, \psi, v )$ is a sufficiently smooth solution of the SOH-NS system \eqref{SOH-NS-3D-SPT-Smp} in three dimension on the interval $[0,T]$. Then there is a constant $C > 0$, depending only upon $s$, all of coefficients of the system \eqref{SOH-NS-3D-SPT-Smp} and the initial data $\widehat{\rho}$, such that the inequality
  \begin{equation}\label{Aprori-Est-SOHNS-3D-SPT-Simp}
    \tfrac{\d}{\d t} \mathcal{E} (t) + \mathcal{D} (t) \leq C \Big{[} 1 + \exp \big{(} C \int_0^t \mathcal{E}^\frac{1}{2} (\tau) \d \tau \big{)} \Big{]} \mathcal{E} (t) [ 1 + \mathcal{E}^{3s} (t) ]
  \end{equation}
  holds for all $t \in [0,T]$.
\end{proposition}

Before proving Proposition \ref{Prop-Apr-Est-3D}, we establish the following lemmas, in which the inequalities will be frequently utilized in the proof of Proposition \ref{Prop-Apr-Est-3D}. Specifically, Lemma \ref{Lm-rho-L^infty} is utilized to control the quantity $\| e^{\widehat{\rho}} \|_{L^\infty}$ and some $L^p$-norms of the higher-order derivatives of $e^{\widehat{\rho}}$ occurred in the proof of Proposition \ref{Aprori-Est-SOHNS-3D-SPT-Simp}, and the inequalities justified in Lemma \ref{Lm-Auxil-Ineq} will be frequently used in deriving \emph{a priori} estimation, which can greatly simplify the proof.

\begin{lemma}\label{Lm-rho-L^infty}
\noindent
\begin{enumerate}
	\item Assume that the function $\widehat{\rho}$ solve the equation
  $$ \partial_t \widehat{\rho} + U \cdot \nabla_x \widehat{\rho} + f = 0 $$
  with the initial condition $ \widehat{\rho} |_{t=0} = \widehat{\rho}^{in}$, where $f(t,x) \in L^1(0,\infty;L^\infty(\R^3))$. If the initial data satisfies $ 0 <  e^{\widehat{\rho}^{in}(x) } \leq \bar{\rho} < +\infty $ for the positive constants $\bar{\rho}$ and the vector field $U (t,x)$ is Lipschitz in the spatial variable $x$, then for all $(t,x) \in \R^+ \times \R^3$
	  \begin{equation*}
	     0 < e^{\widehat{\rho} (t,x)} \leq \bar{\rho} \exp \Big{(}   \int_0^t \| f (s, \cdot) \|_{L^\infty(\R^3)} \d s \Big{)} \,.
	  \end{equation*}

  \item For all integer $k \geq 1$, there is $C(k) > 0$ such that
	  \begin{equation}\label{Bnds-rho-norms}
	    \begin{aligned}
	      \| \nabla_x^k e^{\widehat{\rho}} \|_{L^\infty(\R^3)} \leq & C(k) \| e^{\widehat{\rho}} \|_{L^\infty(\R^3)} \sum_{I = 1}^k \| \widehat{\rho} \|^I_{H^{k+2}(\R^3)} \,, \\
	      \| \nabla_x^k e^{\widehat{\rho}} \|_{L^4(\R^3)} \leq & C(k) \| e^{\widehat{\rho}} \|_{L^\infty(\R^3)} \sum_{I = 1}^k \| \widehat{\rho} \|^I_{H^{k+1}(\R^3)} \,, \\
	      \| \nabla_x^k e^{\widehat{\rho}} \|_{L^2(\R^3)} \leq & C(k) \| e^{\widehat{\rho}} \|_{L^\infty(\R^3)} \sum_{I = 1}^k \| \widehat{\rho} \|^I_{H^{k}(\R^3)} \,.
	    \end{aligned}
	  \end{equation}
\end{enumerate}
\end{lemma}

\begin{lemma}\label{Lm-Auxil-Ineq}
  Let $\Omega(x) = \big{(} \tfrac{2 \phi (x)}{W(x)} , \tfrac{2 \psi(x)}{W(x)} , 1 - \tfrac{2}{W(x)} \big{)}$, where $W (x) = 1 + \phi^2(x) + \psi^2(x)$ and $( \phi(x) , \psi(x) ) \in \R^2$. For any fixed integer $k \geq 1$, there is a positive constant $C = C(k) > 0$ such that
  \begin{equation}\label{Auxil-Ineq-1}
    \begin{aligned}
      \| \nabla_x^k \Omega \|_{L^4} \leq &  C(k) ( \| \nabla_x \phi \|_{H^{k }} + \| \nabla_x \psi \|_{H^{k }} ) \sum_{I=1}^{k } ( \| \phi \|_{H^{k }} + \| \psi \|_{H^{k }} )^{I-1}\,, \\
      \| \nabla_x^k \Omega \|_{L^2} \leq & C(k) ( \| \nabla_x \phi \|_{H^{k-1}} + \| \nabla_x \psi \|_{H^{k-1}} ) \sum_{I=1}^k ( \| \phi \|_{H^{k-1}} + \| \psi \|_{H^{k-1}} )^{I-1} \,, \\
      \| \nabla_x^k \Omega \|_{L^\infty} \leq & C(k) ( \| \nabla_x \phi \|_{H^{k+1}} + \| \nabla_x \psi \|_{H^{k+1}} ) \sum_{I=1}^{k } ( \| \phi \|_{H^{k+1}} + \| \psi \|_{H^{k+1}} )^{I-1}\,.
    \end{aligned}
  \end{equation}
  Moreover, the inequality
  \begin{equation}\label{Auxil-Ineq-2}
    \begin{aligned}
       \| \nabla_x^k W^2 \|_{L^4} \leq &  C(k) ( 1 + \| \phi \|^2_{H^k} + \| \psi \|^2_{H^k} + \| \phi \|^2_{H^3} + \| \psi \|^2_{H^3} ) \\
       & \quad \times ( \| \nabla_x \phi \|_{H^k}  \| \phi \|_{H^k} + \| \nabla_x \psi \|_{H^k} \| \psi \|_{H^k}  )
    \end{aligned}
  \end{equation}
  holds for any fixed $k \geq 2$, and the following inequality
  \begin{equation}\label{Auxil-Ineq-3}
    \begin{aligned}
      \| \nabla_x^k W^2 \|_{L^2} \leq & C(k) ( 1 + \| \phi \|^2_{H^2} + \| \psi \|^2_{H^2} + \| \phi \|^2_{H^{k-1}} + \| \psi \|^2_{H^{k-1}} ) \\
      & \ \ \times ( \| \phi \|_{H^2} + \| \psi \|_{H^2} + \| \phi \|_{H^{k-1}} + \| \psi \|_{H^{k-1}} ) ( \| \nabla_x \phi \|_{H^{k-1}} + \| \nabla_x \psi \|_{H^{k-1}} )
    \end{aligned}
  \end{equation}
  holds for any fixed $k \geq 1$.
\end{lemma}

\begin{proof}[Proof of Lemma \ref{Lm-rho-L^infty}.]
 (1). We consider the flow
  \begin{equation*}
    \left\{
      \begin{array}{l}
        \tfrac{\d}{\d t} X (t,x) = U ( t, X ( t , x ) ) \,, \\
        X(0,x) = x \,.
      \end{array}
    \right.
  \end{equation*}
  Then we have
  \begin{equation*}
    \begin{aligned}
      \tfrac{\d}{\d t} \widehat{\rho} ( t , X(t,x) ) = & \partial_t \widehat{\rho} ( t , X(t,x) ) + ( U \cdot \nabla_x \widehat{\rho} ) ( t , X(t,x) ) \\
      = & - f ( t , X(t,x) ) \,.
    \end{aligned}
  \end{equation*}
  Integrating the above equality on $[0,t]$, we gain
  \begin{equation*}
    \widehat{\rho} ( t , X(t,x) ) = \widehat{\rho}^{in} (x) -  \int_0^t f ( s , X(s,x) ) \d s \,,
  \end{equation*}
  and thus
  \begin{equation*}
    \begin{aligned}
      e^{\widehat{\rho} ( t , X(t,x) )} =  e^{\widehat{\rho}(x)} \exp \Big{(}  \int_0^t f ( s , X(s,x) ) \d s \Big{)} \leq  \bar{\rho} \exp \Big{(}   \int_0^t \| f(s, \cdot) \|_{L^\infty(\R^n)} \d s \Big{)} \,.
    \end{aligned}
  \end{equation*}

  (2). We observe that
  $$ \nabla_x^k e^{\widehat{\rho}} = e^{\widehat{\rho}} \sum_{I=1}^k \sum_{\sum\limits_{\alpha = 1}^I k_\alpha = k, k_\alpha \geq 1} \prod_{\alpha=1}^I \nabla_x^{k_\alpha} \widehat{\rho} \, \equiv \, e^{\widehat{\rho}} \sum_{I=1}^k \Theta^I (\widehat{\rho}) $$
  for any integer $k \geq 1$. Then the relation
  $$ \nabla_x^{k+1} e^{\widehat{\rho}} = e^{\widehat{\rho}} ( \nabla_x \widehat{\rho} \Theta^I(\widehat{\rho}) + \nabla_x \Theta^I(\widehat{\rho}) )  $$
  holds. So, by making use of H\"older inequality, Sobolev embedding $ H^2(\R^3) \hookrightarrow L^\infty(\R^3) $, $ H^1(\R^3) \hookrightarrow L^4(\R^3)$ and adopting the induction method, we can easily justify the inequalities \eqref{Bnds-rho-norms}. Consequently, the proof of Lemma \ref{Lm-rho-L^infty} is finished.

\end{proof}

\begin{proof}[Proof of Lemma \ref{Lm-Auxil-Ineq}.]
  Since each component of the unit vector field $\Omega$ is the form of a rational fraction with lower power polynomial of the factor than that of the denominator, we observe that for any integer $s \geq 1$, there is a constant $C = C(s) > 0$ such that
  $$ \sum_{s_1 + s_2 = s} \Big{|} \tfrac{\partial^s \Omega}{\partial \phi^{s_1} \partial \psi^{s_2}} ( \phi, \psi ) \Big{|} \leq \tfrac{C(s)}{W^\alpha} \leq C(s) $$
  for some $\alpha = \alpha(s) > 0$. By direct calculation, we know that
  \begin{equation}\label{Auxil-Ineq-4}
    \begin{aligned}
      \nabla_x^k \Omega = \sum_{I=1}^k \sum_{i+j=I} \frac{\partial^I \Omega}{\partial \phi^i \partial \psi^j} ( \phi, \psi ) \Upsilon_{ij} (\phi, \psi)\,,
    \end{aligned}
  \end{equation}
  where
  $$ \Upsilon_{ij} (\phi, \psi) = \sum_{ \substack{ \sum\limits_{\alpha = 1}^i k_{\alpha} + \sum\limits_{\beta = 1}^j l_{\beta} = k, \, k_{\alpha}, l_{\beta} \geq 1 }} \nabla_x^{k_1} \phi \cdots \nabla_x^{k_i} \phi \nabla_x^{l_1} \psi \cdots \nabla_x^{l_j} \psi \,.$$

  Now we justify the inequalities \eqref{Auxil-Ineq-1} by induction. For the case $k=1$, we have
  $$ |\nabla_x \Omega| \leq C(1) ( |\nabla_x \phi| + | \nabla_x \psi | ) \,. $$
  Then by Sobolev embedding $H^2(\R^3) \hookrightarrow L^\infty(\R^3)$ and $H^1 (\R^3) \hookrightarrow L^4(\R^3)$, we know that the inequalities \eqref{Auxil-Ineq-1} hold for $k=1$. Assume that the inequalities \eqref{Auxil-Ineq-1} hold for all $1 \leq i \leq k$. We check the inequalities \eqref{Auxil-Ineq-1} for the case $k+1$. By the relation \eqref{Auxil-Ineq-4}, we know that
  \begin{equation*}
    \begin{aligned}
      \nabla_x^{k+1} \Omega = &  \sum_{I=1}^k \sum_{i+j=I} \bigg{(} \frac{\partial^{I+1} \Omega}{\partial \phi^{i+1} \partial \psi^j} ( \phi, \psi ) \nabla_x \phi +  \frac{\partial^{I+1} \Omega}{\partial \phi^i \partial \psi^{j+1}} ( \phi, \psi ) \nabla_x \psi \bigg{)} \Upsilon_{ij} (\phi, \psi) \\
      & + \sum_{I=1}^k \sum_{i+j=I} \frac{\partial^I \Omega}{\partial \phi^i \partial \psi^j} ( \phi, \psi ) \nabla_x \Upsilon_{ij} (\phi, \psi) \,,
    \end{aligned}
  \end{equation*}
which immediately yields that
  \begin{align}\label{Auxil-Ineq-5}
    | \nabla_x^{k+1} \Omega | \leq C(k+1) \sum_{I=1}^k \sum_{i+j=I} \big{[} ( |\nabla_x \phi| + |\nabla_x \psi| ) |\Upsilon_{ij} (\phi, \psi)| + |\nabla_x \Upsilon_{ij} (\phi, \psi)| \big{]}\,.
  \end{align}
  As a consequence, Sobolev embedding $H^2(\R^3) \hookrightarrow L^\infty(\R^3)$, $H^1 (\R^3) \hookrightarrow L^4(\R^3)$ and the assumption of the case $k$ in the induction reduce to the conclusion of the case $k+1$, which means that the inequalities \eqref{Auxil-Ineq-1} holds for all integer $k \geq 1$.

  Next, our goal is to verify the inequality \eqref{Auxil-Ineq-2} for $k \geq 2$. We observe that
  \begin{equation}\label{Auxil-Relt-W-1}
    \begin{aligned}
      \| \nabla_x^k W^2 \|_{L^4} = & 2 \| W \nabla_x^k W \|_{L^4} + 2 \| \nabla_x W \nabla_x^{k-1} W \|_{L^4} + \sum_{\substack{a+b=k \\ a,b \geq 2}} \| \nabla_x^a W \nabla_x ^b W \|_{L^4} \\
      \equiv & \, A_1 + A_2 + A_3 \,.
    \end{aligned}
  \end{equation}
  We estimate the quantities $A_i (i=1,2,3)$ term by term by making use of H\"older inequality and Sobolev embedding theory. First, we have by the definition of $W$
  \begin{equation}\label{Auxil-Relt-W-2}
    \begin{aligned}
      A_1 \lesssim & ( 1 + \| \phi \|^2_{L^\infty} + \| \psi \|^2_{L^\infty}   ) \sum_{a+b=k} \| \nabla_x^a \phi \nabla_x^b \psi + \nabla_x^a \psi \nabla_x^b \psi \|_{L^4} \\
      \lesssim & ( 1 + \| \phi \|^2_{H^2} + \| \psi \|^2_{H^2} ) \Big{[} \| \phi \|_{L^\infty} \| \nabla_x^k \phi \|_{L^4} + \| \psi \|_{L^\infty} \| \nabla_x^k \psi \|_{L^4} \\
      & \quad + \sum_{\substack{a+b=k \\ a,b \geq 1}} ( \| \nabla_x^a \phi \|_{L^\infty} \| \nabla_x^b \phi \|_{L^4} + \| \nabla_x^a \psi \|_{L^\infty} \| \nabla_x^b \psi \|_{L^4} ) \Big{]} \\
      \lesssim & ( 1 + \| \phi \|^2_{H^k} + \| \psi \|^2_{H^k} ) \big{[} \| \nabla_x \phi \|_{H^k} \| \phi \|_{H^k}  + \| \nabla_x \psi \|_{H^k}  \| \psi \|_{H^k}  \big{]}\
    \end{aligned}
  \end{equation}
  for $k \geq 2$, and
  \begin{equation}\label{Auxil-Relt-W-3}
    \begin{aligned}
      A_2 \lesssim & \| \nabla_x W \|_{L^\infty} \| \nabla_x^{k-1} W \|_{L^4} \\
      \lesssim & ( \| \phi \|_{L^\infty} \| \nabla_x \phi \|_{L^\infty} + \| \psi \|_{L^\infty} \| \nabla_x \psi \|_{L^\infty} ) \sum_{a+b=k-1} \| \nabla_x^a \phi \nabla_x^b \phi + \nabla_x^a \psi \nabla_x^b \psi \|_{L^4} \\
      \lesssim & ( \| \phi \|^2_{H^3} + \| \psi \|^2_{H^3} ) \Big{[} \| \phi \|_{L^\infty} \| \nabla_x^{k-1} \phi \|_{L^4} + \| \psi \|_{L^\infty} \| \nabla_x^{k-1} \psi \|_{L^4} \\
      & \quad + \sum_{\substack{a+b=k-1 \\ a,b \geq 1}} ( \| \nabla_x^a \phi \|_{L^\infty} \| \nabla_x^b \phi \|_{L^4} + \| \nabla_x^a \psi \|_{L^\infty} \| \nabla_x^b \psi \|_{L^4} ) \Big{]} \\
      \lesssim & ( \| \phi \|^2_{H^3} + \| \psi \|^2_{H^3} ) ( \| \nabla_x \phi \|_{H^k} \| \phi \|_{H^k} + \| \nabla_x \phi \|_{H^k} \| \phi \|_{H^k} ) \,,
    \end{aligned}
  \end{equation}
  and
    \begin{align}\label{Auxil-Relt-W-4}
      A_3
    \lesssim & \sum_{\substack{a+b=k \\ a,b \geq 2}} \| \nabla_x^a W \|_{L^\infty} \| \nabla_x ^b W \|_{L^4}
    \no\\[3pt]
      \lesssim & \sum_{\substack{a+b=k \\ a,b \geq 2}} \sum_{a_1+a_2=a} ( \| \nabla_x^{a_1} \phi \|_{L^\infty} \| \nabla_x^{a_2} \phi \|_{L^\infty} + \| \nabla_x^{a_1} \psi \|_{L^\infty} \| \nabla_x^{a_2} \psi \|_{L^\infty} ) \| \nabla_x^b W \|_{L^4}
    \\[3pt] \no
      \lesssim & \sum_{\substack{a+b=k \\ 2 \leq b \leq k-2}} ( \| \phi \|^2_{H^{a+2}} + \| \psi \|^2_{H^{a+2}} ) \| \nabla_x^b W \|_{L^4}
    \\[3pt] \no
      \lesssim & ( \| \phi \|^2_{H^k} + \| \psi \|^2_{H^k} ) \sum_{2 \leq b \leq k-2 } \sum_{b_1+b_2=b} \| \nabla_x^{b_1} \phi \nabla_x^{b_2} \phi + \nabla_x^{b_1} \psi \nabla_x^{b_2} \psi \|_{L^4}
    \\[3pt] \no
      \lesssim & ( \| \phi \|^2_{H^k} + \| \psi \|^2_{H^k} ) \sum_{2 \leq b \leq k-2 } \sum_{b_1+b_2=b} ( \| \nabla_x^{b_1} \phi \|_{L^\infty} \| \nabla_x^{b_2} \phi \|_{L^4} + \| \nabla_x^{b_1} \psi \|_{L^\infty} \| \nabla_x^{b_2} \psi \|_{L^4} )
    \\[3pt] \no
      \lesssim & ( \| \phi \|^2_{H^k} + \| \psi \|^2_{H^k} ) ( \| \nabla_x \phi \|_{H^k} \| \phi \|_{H^k} + \| \nabla_x \psi \|_{H^k} \| \psi \|_{H^k} ) \,.
    \end{align}
  Consequently, the inequality \eqref{Auxil-Ineq-2} is derived from plugging the inequalities \eqref{Auxil-Relt-W-2}, \eqref{Auxil-Relt-W-3} and \eqref{Auxil-Relt-W-4} into the relation \eqref{Auxil-Relt-W-1}.

  In the end, we derive the last inequality \eqref{Auxil-Ineq-3} in this lemma. By H\"older inequality and Sobolev embedding $H^2(\R^3) \hookrightarrow L^\infty(\R^3)$ and $H^1 (\R^3) \hookrightarrow L^4(\R^3)$, we gain that for $k \geq 1$,
  \begin{equation}\label{Auxil-Relt-W-5}
    \begin{aligned}
      \| \nabla_x W^2 \|_{L^2} \lesssim & \| W \|_{L^\infty} \| \nabla_x^k W \|_{L^2} + \sum_{\substack{k_1+k_2=k \\ k_1, k_2 \geq 1}} \| \nabla_x^{k_1} W \nabla_x^{k_2} W \|_{L^2} \\
      \lesssim & ( 1 + \| \phi \|^2_{H^2} + \| \psi \|^2_{H^2} ) \sum_{k_1+k_2=k} ( \| \nabla_x^{k_1} \phi  \nabla_x^{k_2} \phi  + \nabla_x^{k_1} \psi  \nabla_x^{k_2} \psi \|_{L^2} ) \\
      + & \sum_{\substack{k_1+k_2=k \\ k_1, k_2 \geq 1}} \sum_{a+b=k_1} \| \nabla_x^a \phi \nabla_x^b \phi + \nabla_x^a \psi \nabla_x^b \psi \|_{L^4} \sum_{p+q=k_2} \| \nabla_x^p \phi \nabla_x^q \phi + \nabla_x^p \psi \nabla_x^q \psi \|_{L^4} \\
      \equiv & B_1 + B_2 \,.
    \end{aligned}
  \end{equation}
  For the term $B_1$, we calculate that
    \begin{align}\label{Auxil-Relt-W-6}
      \no B_1 \lesssim & ( 1 + \| \phi \|^2_{H^2} + \| \psi \|^2_{H^2} ) ( \| \phi \|_{L^\infty} \| \nabla_x^k \phi \|_{L^2} + \| \psi \|_{L^\infty} \| \nabla_x^k \psi \|_{L^2} ) \\
      \no & + ( 1 + \| \phi \|^2_{H^2} + \| \psi \|^2_{H^2} ) \sum_{\substack{k_1+k_2=k \\ k_1, k_2 \geq 1}} ( \| \nabla_x^{k_1} \phi \|_{L^4} \| \nabla_x^{k_2} \phi \|_{L^4} + \| \nabla_x^{k_1} \psi \|_{L^4} \| \nabla_x^{k_2} \psi \|_{L^4} ) \\
      \no \lesssim & ( 1 + \| \phi \|^2_{H^2} + \| \psi \|^2_{H^2} ) ( \| \phi \|_{H^2} \| \nabla_x \phi \|_{H^{k-1}} + \| \psi \|_{H^2} \| \nabla_x \psi \|_{H^{k-1}} ) \\
      \no & + ( 1 + \| \phi \|^2_{H^2} + \| \psi \|^2_{H^2} ) \sum_{\substack{k_1+k_2=k \\ k_1, k_2 \geq 1}} ( \| \nabla_x^{k_1} \phi \|_{H^1} \| \nabla_x^{k_2} \phi \|_{H^1} + \| \nabla_x^{k_1} \psi \|_{H^1} \| \nabla_x^{k_2} \psi \|_{H^1} ) \\
      \lesssim & ( 1 + \| \phi \|^2_{H^2} + \| \psi \|^2_{H^2} ) ( \| \phi \|_{H^2} + \| \psi \|_{H^2} + \| \phi \|_{H^{k-1}} + \| \psi \|_{H^{k-1}} ) \\
      \no & \qquad \qquad \times ( \| \nabla_x \phi \|_{H^{k-1}} + \| \nabla_x \psi \|_{H^{k-1}} ) \,.
    \end{align}
    We notice that
    \begin{equation*}
      \begin{aligned}
        & \sum_{a+b=k_1} \| \nabla_x^a \phi \nabla_x^b \phi \|_{L^4} \lesssim  \| \phi \|_{L^\infty} \| \nabla_x^{k_1} \phi \|_{L^4} + \sum_{\substack{a+b=k_1 \\ a,b \geq 1}} \| \nabla_x^a \phi \|_{L^\infty} \| \nabla_x^b \phi \|_{L^4} \\
        \lesssim & \| \phi \|_{H^2} \| \nabla_x^{k_1} \|_{H^1} + \sum_{\substack{a+b=k_1 \\ a,b \geq 1}} \| \nabla_x^a \phi \|_{H^2} \| \nabla_x^b \phi \|_{H^1} \lesssim  ( \| \phi \|_{H^2} + \| \phi \|_{H^{k_1}} ) \| \nabla_x \phi \|_{H^{k_1}} \,,
      \end{aligned}
    \end{equation*}
    which immediately deduces us that
    \begin{equation}\label{Auxil-Relt-W-7}
      \begin{aligned}
        B_2 \lesssim & \sum_{\substack{k_1+k_2=k \\ k_1 , k_2 \geq 1}} \big{[} ( \| \phi \|_{H^2} + \| \phi \|_{H^{k_1}} ) \| \nabla_x \phi \|_{H^{k_1}} + ( \| \psi \|_{H^2} + \| \psi \|_{H^{k_1}} ) \| \nabla_x \psi \|_{H^{k_1}} \big{]} \\
        & \quad \times \big{[} ( \| \phi \|_{H^2} + \| \phi \|_{H^{k_2}} ) \| \nabla_x \phi \|_{H^{k_2}} + ( \| \psi \|_{H^2} + \| \psi \|_{H^{k_2}} ) \| \nabla_x \psi \|_{H^{k_2}} \big{]} \\
        \lesssim & ( \| \phi \|_{H^2} + \| \phi \|_{H^{k-1}} + \| \psi \|_{H^2} + \| \psi \|_{H^{k-1}} )^2 \\
         & \qquad \times ( \| \phi \|_{H^{k-1}} + \| \psi \|_{H^{k-1}} ) ( \| \nabla_x \phi \|_{H^{k-1}} + \| \nabla_x \phi \|_{H^{k-1}} )\,.
      \end{aligned}
    \end{equation}
    Then, we substitute the inequalities \eqref{Auxil-Relt-W-6} and \eqref{Auxil-Relt-W-7} into the inequality \eqref{Auxil-Relt-W-5}, and consequently, we finish the proof of Lemma \ref{Lm-Auxil-Ineq}.
\end{proof}

\begin{remark}
  The inequalities \eqref{Auxil-Ineq-1} in Lemma \ref{Lm-Auxil-Ineq} still holds if the vector field $\Omega$ is replaced by the vector field $\Omega_\phi$ or $\Omega_\psi$.
\end{remark}

Based on the conclusions in Lemma \ref{Lm-rho-L^infty} and the inequalities in Lemma \ref{Lm-Auxil-Ineq}, we now can give the proof of Proposition \ref{Prop-Apr-Est-3D}.
\begin{proof}[Proof of Proposition \ref{Prop-Apr-Est-3D}.]
In the proof of this proposition, we mainly make use of the H\"older inequality, Sobolev embedding $H^2(\R^3) \hookrightarrow L^\infty(\R^3)$ and $H^1 (\R^3) \hookrightarrow L^4(\R^3)$. In order to finish the proof of this conclusion, we derive the energy inequality of each equation of the SOH-NS system \eqref{SOH-NS-3D-SPT-Smp}, respectively. In other word, we can divide this proof into five steps, where the last step is to close the energy estimate by summing up for all energy estimates obtained in the previous four steps.

{\smallskip\noindent \em\large Step 1. Energy estimate for $\widehat{\rho}$-equation.}

For all integer $0 \leq k \leq s$, we first act the derivative operator $\nabla_x^k$ on the first $\widehat{\rho}$-equation in the system \eqref{SOH-NS-3D-SPT-Smp}, and then multiply by $\nabla_x^k \widehat{\rho}$ and integrate by parts on $\R^3$, then we gain
\begin{equation}\label{Apriori-Est-3D-rho-1}
  \begin{aligned}
    \tfrac{1}{2} \tfrac{\d}{\d t} \| \nabla_x^k \widehat{\rho} \|^2_{L^2} = & - a c_1 \l \nabla_x^k ( \Omega \cdot \nabla_x \widehat{\rho} ), \nabla_x^k \widehat{\rho} \r - \l \nabla_x^k ( v \cdot \nabla_x \widehat{\rho} ) , \nabla_x^k \widehat{\rho} \r \\
    & - a c_1 \l \nabla_x^k ( \Omega_\phi \cdot \nabla_x \phi ), \nabla_x^k \widehat{\rho} \r - a c_1 \l \nabla_x^k ( \Omega_\psi \cdot \nabla_x \psi ), \nabla_x^k \widehat{\rho} \r \\
    \equiv & \, D_1 + D_2 + D_3 + D_4 \,,
  \end{aligned}
\end{equation}
where we utilize the relation $U = a c_1 \Omega + v$. Now we compute the terms $D_i \, (1 \leq i \leq 4)$ term by term. For the term $D_1$, we observe that
  \begin{align}\label{Apriori-Est-3D-rho-2}
    \no D_1 = & -a c_1 \l \Omega \cdot \nabla_x \nabla_x^k \widehat{\rho}, \nabla_x^k \widehat{\rho} \r - a c_1 \l \nabla_x \Omega \cdot \nabla_x^k \widehat{\rho}, \nabla_x^k \widehat{\rho} \r \\
    \no & \qquad \qquad- a c_1 \sum_{\substack{k_1+k_2=k \\ k_1 \geq 2}} \l \nabla_x^{k_1} \Omega \cdot \nabla_x^{k_2 + 1} \widehat{\rho}, \nabla_x^k \widehat{\rho} \r \\
     = & \tfrac{1}{2} a c_1 \l \nabla_x \cdot \Omega , |\nabla_x^k \widehat{\rho}|^2 \r  - a c_1 \l \nabla_x \Omega \cdot \nabla_x^k \widehat{\rho}, \nabla_x^k \widehat{\rho} \r \\
    \no & \qquad \qquad - a c_1 \sum_{\substack{k_1+k_2=k \\ k_1 \geq 2}} \l \nabla_x^{k_1} \Omega \cdot \nabla_x^{k_2 + 1} \widehat{\rho}, \nabla_x^k \widehat{\rho} \r \\
    \no \lesssim & |a c_1|  \l |\nabla_x \Omega|, |\nabla_x^k \widehat{\rho}|^2 \r + | a c_1 | \sum_{\substack{k_1+k_2=k \\ k_1 \geq 2}} \l | \nabla_x^{k_1} \Omega | | \nabla_x^{k_2 + 1} \widehat{\rho} | , | \nabla_x^k \widehat{\rho} |  \r \\
    \no \equiv & \, D_{11} + D_{12}\,.
  \end{align}
It is easily derived from the H\"older inequality, Sobolev embedding theory and the inequalities in Lemma \ref{Lm-Auxil-Ineq} that
\begin{equation}\label{Apriori-Est-3D-rho-3}
  \begin{aligned}
    D_{11} \lesssim & | a c_1 | \l |\Omega_\phi| |\nabla_x \phi| + |\Omega_\psi| |\nabla_x \psi |, |\nabla_x^k \widehat{\rho}|^2 \r \lesssim  |a c_1| ( \| \nabla_x \phi \|_{L^\infty} + \| \nabla_x \psi \|_{L^\infty} ) \| \nabla_x^k \widehat{\rho} \|^2_{L^2} \\
    \lesssim & | a c_1 | ( \| \nabla_x \phi \|_{H^s} + \| \nabla_x \psi \|_{H^s} ) \| \widehat{\rho} \|^2_{H^s}
  \end{aligned}
\end{equation}
if $s \geq 2$, and
\begin{equation}\label{Apriori-Est-3D-rho-4}
  \begin{aligned}
    D_{12} \lesssim & |a c_1| \sum_{\substack{k_1+k_2=k \\ k_1 \geq 2}} \| \nabla_x^{k_1} \Omega \|_{L^4} \| \nabla_x^{k_2 + 1} \|_{L^4} \| \nabla_x^k \widehat{\rho} \|_{L^2} \\
    \lesssim & |a c_1| \sum_{\substack{k_1+k_2=k \\ k_1 \geq 2}} ( \| \nabla_x \phi \|_{H^{k_1}} + \| \psi \|_{H^{k_1}} ) \sum_{I=1}^{k_1}( \| \phi \|_{H^{k_1}} + \| \psi \|_{H^{k_1}} )^{I-1} \| \nabla_x^{k_2 + 1} \widehat{\rho} \|_{H^1} \| \widehat{\rho} \|_{H^s} \\
    \lesssim & |a c_1| ( \| \nabla_x \phi \|_{H^s} + \| \nabla_x \psi \|_{H^s} ) \| \widehat{\rho} \|^2_{H^s} \big{[} 1 + ( \| \phi \|_{H^s} + \| \psi \|_{H^s} )^{s-1} \big{]}\,.
  \end{aligned}
\end{equation}
Thus, by plugging the inequalities \eqref{Apriori-Est-3D-rho-3} and \eqref{Apriori-Est-3D-rho-4} into the equality \eqref{Apriori-Est-3D-rho-2}, we immediately know that
\begin{equation}\label{Apriori-Est-3D-rho-5}
  D_1 \lesssim |a c_1| ( \|\nabla_x \phi \|_{H^s} + \| \nabla_x \psi \|_{H^s} ) \| \widehat{\rho} \|^2_{H^s} \big{[} 1 + ( \| \phi \|_{H^s} + \| \psi \|_{H^s} )^{s-1} \big{]} \,.
\end{equation}

For the term $D_2$, it is easy to be calculated that for $s \geq 2$
\begin{equation}\label{Apriori-Est-3D-rho-6}
  \begin{aligned}
    D_2 = & -  \l \nabla_x v \nabla_x^k \widehat{\rho}, \nabla_x^{k} \widehat{\rho} \r - \sum_{\substack{k_1+k_2=k \\ k_1 \geq 2}} \l \nabla_x^{k_1} v \nabla_x^{k_2+1} \widehat{\rho} , \nabla_x^k \widehat{\rho} \r \\
    \lesssim & \| \nabla_x v \|_{L^\infty} \| \nabla_x^k \widehat{\rho} \|^2_{L^2} + \sum_{\substack{k_1+k_2=k \\ k_1 \geq 2}} \| \nabla_x^{k_1} v \|_{L^4} \| \nabla_x^{k_2+1} \widehat{\rho} \|_{L^4} \| \nabla_x^k \widehat{\rho} \|_{L^2} \\
    \lesssim & \| \nabla_x v \|_{H^s} \| \widehat{\rho} \|^2_{H^s} \,.
  \end{aligned}
\end{equation}
Here we utilize the fact $\nabla_x \cdot v = 0$. Similar arguments in the estimation of the inequality \eqref{Apriori-Est-3D-rho-5} on the term $D_1$ deduce to that
\begin{equation}\label{Apriori-Est-3D-rho-7}
  \begin{aligned}
    D_3 + D_4 \lesssim |a c_1| ( \| \nabla_x \phi \|_{H^s} + \| \nabla_x \psi \|_{H^s} ) \| \widehat{\rho} \|_{H^s} \big{[} 1 + ( \| \phi \|_{H^s} + \| \psi \|_{H^s} )^s \big{]} \,.
  \end{aligned}
\end{equation}
Consequently, by substituting the estimations \eqref{Apriori-Est-3D-rho-5}, \eqref{Apriori-Est-3D-rho-6} and \eqref{Apriori-Est-3D-rho-7} into the equality \eqref{Apriori-Est-3D-rho-1} and summing up for all integers $0 \leq k \leq s$, we gain the energy estimate of the $\widehat{\rho}$-equation of the SOH-NS system \eqref{SOH-NS-3D-SPT-Smp}:
  \begin{align}\label{Apriori-Est-3D-rho}
    	\tfrac{\d}{\d t} \| \widehat{\rho} \|^2_{H^s}
  	\lesssim & \| \nabla_x v \|_{H^s} \| \widehat{\rho} \|^2_{H^s} + |a c_1| ( \| \nabla_x \phi \|_{H^s} + \| \nabla_x \psi \|_{H^s} ) \| \widehat{\rho} \|^2_{H^s} \big{[} 1 + ( \| \phi \|_{H^s} + \| \psi \|_{H^s} )^{s-1} \big{]}
  \no\\
    & + |a c_1| ( \| \nabla_x \phi \|_{H^s} + \| \nabla_x \psi \|_{H^s} ) \| \widehat{\rho} \|_{H^s} \big{[} 1 + ( \| \phi \|_{H^s} + \| \psi \|_{H^s} )^{s} \big{]} \,.
  \end{align}

{\smallskip \noindent \em \large Step 2. Energy estimate for $\phi$-equation.}

For all integers $0 \leq k \leq s$, taking $L^2$-inner product with $\nabla_x^k \phi$ in the second $\phi$-equation in the SOH-NS system \eqref{SOH-NS-3D-SPT-Smp} after acting the $k$-order derivative operator $\nabla_x^k$ on that, and integrating by parts on $\R^3$, one will derive that
  \begin{align}\label{Apriori-Est-3D-phi-1}
    \no \tfrac{1}{2} \tfrac{\d}{\d t} \| \nabla_x^k \phi \|^2_{L^2} = & - a c_2 \l \Omega \cdot \nabla_x^{k+1} \phi , \nabla_x^k \r - a c_2 \l \nabla_x \Omega \cdot \nabla_x^k \phi, \nabla_x^k \phi \r - \l \nabla_x v \nabla_x^k \phi , \nabla_x^k \phi \r \\
     & - a c_2 \sum_{\substack{k_1+k_2=k \\ k_1 \geq 2}} \l \nabla_x^{k_1} \Omega \nabla_x^{k_2+1} \phi , \nabla_x^k \phi \r - \sum_{\substack{k_1+k_2=k \\ k_1 \geq 2}} \l \nabla_x^{k_1} v \nabla_x^{k_2+1} \phi , \nabla_x^k \phi \r \\
    \no & \qquad - \l \nabla_x^k \mathcal{H}_\phi ( \widehat{\rho}, \phi, \psi, v ), \nabla_x^k \phi \r \\
    \no \equiv & \, E_1 + E_2 + E_3 + E_4 + E_5  + E_6\,,
  \end{align}
where we use the relation $V = a c_2 \Omega + v$ and the incompressibility $\nabla_x \cdot v = 0$. For the first three terms $E_1$, $E_2$ and $E_3$, which are of simple form, one can estimate that by the H\"older inequality and Sobolev embedding $H^2(\R^3) \hookrightarrow L^\infty(\R^3)$ and $H^1 (\R^3) \hookrightarrow L^4(\R^3)$
\begin{equation}\label{Apriori-Est-3D-phi-2}
  \begin{aligned}
    E_1 + E_2 & + E_3 \lesssim |a c_2| \| \nabla_x^{k+1} \phi \|_{L^2} \| \nabla_x^k \phi \|_{L^2} \\
    & + |a c_2| (\| \nabla_x \phi \|_{L^\infty} + \| \nabla_x \psi \|_{L^\infty} ) \| \nabla_x^k \phi \|^2_{L^2} + \| \nabla_x v \|_{L^\infty} \| \nabla_x^k \phi \|^2_{L^2} \\
    \lesssim & |a c_2| \| \nabla_x \phi \|_{H^s} \| \phi \|_{H^s} + |a c_2| ( \| \nabla_x \phi \|_{H^s} + \| \nabla_x \psi \|_{H^s} ) \| \phi \|^2_{H^s} + \| \nabla_x v\|_{H^s} \| \phi \|^2_{H^s} \,.
  \end{aligned}
\end{equation}
The forth and fifth terms $E_4$, $E_5$ are easily estimated as follows:
\begin{equation}\label{Apriori-Est-3D-phi-3}
  \begin{aligned}
    E_4 +&  E_5 \lesssim |a c_2| \sum_{\substack{k_1+k_2=k \\ k_1 \geq 2}} \| \nabla_x^{k_1} \Omega \|_{L^4} \| \nabla_x^{k_2+1} \phi \|_{L^4} \| \nabla_x^k \phi \|_{L^2} \\
    & + \sum_{\substack{k_1+k_2=k \\ k_1 \geq 2}} \| \nabla_x^{k_1} v \|_{L^4} \| \nabla_x^{k_2+1} \phi \|_{L^4} \| \nabla_x^k \phi \|_{L^2} \\
    \lesssim & |a c_2| ( \| \nabla_x \phi \|_{H^s} + \| \nabla_x \psi \|_{H^s} ) \| \phi \|^2_{H^s} \big{[} 1 + ( \| \phi \|_{H^s} + \| \psi \|_{H^s} )^{s-1} \big{]} + \| \nabla_x v \|_{H^s} \| \phi \|^2_{H^s} \,,
  \end{aligned}
\end{equation}
where we make use of the H\"older inequality, Sobolev embedding theory and the inequalities \eqref{Auxil-Ineq-1} in Lemma \ref{Lm-Auxil-Ineq}.

Now we estimate the term $E_6$, which is the hardest work in this proof. By the definition of the quantity $\mathcal{H}_\phi (\widehat{\rho}, \phi, \psi, v)$, we decompose the term $E_6$ as
\begin{equation}\label{Apriori-Est-3D-phi-4}
  \begin{aligned}
    E_6 = & - \gamma \| \nabla_x^{k+1} \phi \|^2_{L^2} - \tfrac{a}{4 \kappa} \l \nabla_x^k ( W^2 \Omega_\phi \cdot \nabla_x \widehat{\rho} ), \nabla_x^k \phi \r + 2 \gamma \l \nabla_x^k ( \nabla_x \widehat{\rho} \cdot \nabla_x \phi ) \r \\
    & - 2 \gamma \l \nabla_x^k ( \tfrac{\phi}{W} | \nabla_x \phi |^2 ) , \nabla_x^k \phi \r - 4 \gamma \l \nabla_x^k ( \tfrac{\psi}{W} \nabla_x \phi \cdot \nabla_x \psi ), \nabla_x^k \phi \r \\
    & + 2 \gamma \l \nabla_x^k( \tfrac{\phi}{W} | \nabla_x \psi |^2 ) , \nabla_x^k \phi \r+ \tfrac{1}{4} \l \nabla_x^k [ W^2 (\widetilde{\lambda} S(v) + A(v) ) : \Omega_\phi \otimes \Omega ] , \nabla_x^k \phi \r \\
    \equiv & \, - \gamma \| \nabla_x^{k+1} \phi \|^2_{L^2} + E_{61} + E_{62} + E_{63} + E_{64} + E_{65} + E_{66}\,.
  \end{aligned}
\end{equation}
Now we estimate the terms $E_{6i} \, (1 \leq i \leq 6)$. We first compute the term $E_{61}$ by integrating by parts over $\R^3$ as follows.
  \begin{align}\label{Apriori-Est-3D-phi-5}
    \no E_{61} = & \tfrac{a}{4 \kappa} \l \nabla_x \cdot ( W^2 \Omega_\phi ) \nabla_x^k \widehat{\rho} , \nabla_x^k \phi \r + \tfrac{a}{4 \kappa} \l W^2 \Omega_\phi \cdot \nabla_x^k \widehat{\rho} , \nabla_x^{k+1} \phi \r - \tfrac{a}{4 \kappa} \l \nabla_x ( W^2 \Omega_\phi ) \cdot \nabla_x^k \widehat{\rho} , \nabla_x^k \phi \r \\
    & - \tfrac{a}{4 \kappa} \sum_{\substack{k_1+k_2=k \\ k_1 \geq 2, k_2 \geq 1}} \l \nabla_x^{k_1} ( W^2 \Omega_\phi ) \nabla_x^{k_2+1} \widehat{\rho} , \nabla_x^k \phi \r - \tfrac{a}{4 \kappa} \l \nabla_x^k ( W^2 \Omega_\phi ) \nabla_x \widehat{\rho} , \nabla_x^k \phi \r \\
    \no \equiv & \, E_{611} + E_{612} + E_{613} + E_{614} + E_{615} \,.
  \end{align}
For $E_{611}$, one can derive from the H\"older inequality and Sobolev embedding $H^2(\R^3) \hookrightarrow L^\infty(\R^3)$ and $H^1 (\R^3) \hookrightarrow L^4(\R^3)$ that
\begin{equation}\label{Apriori-Est-3D-phi-6}
  \begin{aligned}
    E_{611} \lesssim & |\tfrac{a}{\kappa}| \| \nabla_x \cdot ( W^2 \Omega_\phi ) \|_{L^\infty} \| \nabla_x^k \widehat{\rho} \|_{L^2} \| \nabla_x^k \phi \|_{L^2} \\
    = & |\tfrac{a}{\kappa}| \| 2 W \nabla_x W \cdot \Omega_\phi + W^2 \Omega_{\phi \phi} \cdot \nabla_x \phi + W^2 \Omega_{\phi \psi} \cdot \nabla_x \psi \|_{L^\infty} \| \nabla_x^k \widehat{\rho} \|_{L^2} \| \nabla_x^k \phi \|_{L^2} \\
    \lesssim & |\tfrac{a}{\kappa}| ( \| W \|^{\frac{3}{2}}_{L^\infty} + \| W \|^2_{L^\infty} ) ( \| \nabla_x \phi \|_{L^\infty} + \| \nabla_x \psi \|_{L^\infty} ) \| \nabla_x^k \widehat{\rho} \|_{L^2} \| \nabla_x^k \phi \|_{L^2} \\
    \lesssim & |\tfrac{a}{\kappa}| ( 1 + \| \phi \|^4_{H^s} + \| \psi \|^4_{H^s} ) ( \| \nabla_x \phi \|_{H^s} + \| \nabla_x \psi \|_{H^s} ) \| \widehat{\rho} \|_{H^s} \| \phi \|_{H^s} \,.
  \end{aligned}
\end{equation}
Here we make use of the bounds $ |\Omega_\phi| + |\Omega_{\phi \phi }| + |\Omega_{\phi \psi}| \leq C $ and the inequality $\|W\|_{L^\infty} \leq C ( 1 + \| \phi \|^2_{H^2} + \| \psi \|^2_{H^2}  )$. For $E_{612}$, H\"older inequality, the bounds $|W \Omega_\phi| \leq C$ and Sobolev embedding $\|W\|_{L^\infty} \leq C ( 1 + \| \phi \|^2_{H^2} + \| \psi \|^2_{H^2}  )$ yield that
\begin{equation}\label{Apriori-Est-3D-phi-7}
  \begin{aligned}
    E_{612} \lesssim & |\tfrac{a}{\kappa}| \| W^2 \Omega_\phi \|_{L^\infty} \| \nabla_x^k \widehat{\rho} \|_{L^2} \| \nabla_x^{k+1} \phi \|_{L^2} \lesssim |\tfrac{a}{\kappa}| ( 1 +  \| \phi \|^2_{L^\infty} + \| \psi \|^2_{L^\infty} ) \| \widehat{\rho} \|_{H^s} \| \nabla_x \phi \|_{H^s} \\
    \lesssim & |\tfrac{a}{\kappa}| \| \nabla_x \phi \|_{H^s} \| \widehat{\rho} \|_{H^s} ( 1 +  \| \phi \|^2_{H^s} + \| \psi \|^2_{H^s} )\,.
  \end{aligned}
\end{equation}
By the similar estimation of $E_{611}$, one can easily gain the bound of the term $E_{613}$
\begin{equation}\label{Apriori-Est-3D-phi-8}
  \begin{aligned}
    E_{613} \lesssim |\tfrac{a}{\kappa}| ( \| \nabla_x \phi \|_{H^s} + \| \nabla_x \psi \|_{H^s} ) \| \widehat{\rho} \|_{H^s} \| \phi \|_{H^s} ( 1 + \| \phi \|^4_{H^s} + \| \psi \|^4_{H^s} )\,.
  \end{aligned}
\end{equation}

It is easily known that
  \begin{align*}
    E_{614} \lesssim & |\tfrac{a}{\kappa}| \sum_{\substack{k_1+k_2=k \\ k_1 \geq 2, k_2 \geq 1}} \| \nabla_x^{k_1} ( W^2 \Omega_\phi ) \|_{L^4} \| \nabla_x^{k_2+1} \widehat{\rho} \|_{L^4} \| \nabla_x^k \phi \|_{L^2} \\
    \lesssim & |\tfrac{a}{\kappa}|  \sum_{2 \leq k_1 \leq k-1} \| \nabla_x^{k_1} ( W^2 \Omega_\phi ) \|_{L^4} \| \widehat{\rho} \|_{H^s} \| \phi \|_{H^s} \,.
  \end{align*}
For all $2 \leq k_1 \leq k-1$, the quantity $ \| \nabla_x^{k_1} ( W^2 \Omega_\phi ) \|_{L^4} $ can be estimated by using the inequalities in Lemma \ref{Lm-Auxil-Ineq}
\begin{equation*}
  \begin{aligned}
    & \| \nabla_x^{k_1} ( W^2 \Omega_\phi ) \|_{L^4} \lesssim  \| W^2 \nabla_x^{k_1} \Omega_\phi \|_{L^4} + \| \nabla_x W^2 \nabla_x^{k_1 - 1} \Omega_\phi \|_{L^4} + \sum_{\substack{a+b=k_1 \\ a \geq 2}} \| \nabla_x^a W^2 \nabla_x^b \Omega_\phi \|_{L^4} \\
    \lesssim & \| W^2 \|_{L^\infty} \| \nabla_x^{k_1} \Omega_\phi \|_{L^4} + \| \nabla_x W^2 \|_{L^4} \| \nabla_x^{k_1 - 1} \Omega_\phi \|_{L^\infty} + \sum_{\substack{a+b=k_1 \\ a \geq 2}} \| \nabla_x^a W^2 \|_{L^4} \| \nabla_x^b \Omega_\phi \|_{L^\infty} \\
    \lesssim & ( 1 + \| \phi \|^4_{L^\infty} + \| \psi \|^4_{L^\infty} ) ( \| \nabla_x \phi \|_{H^{k_1}} + \| \nabla_x \psi \|_{H^{k_1}} )  \big{[} 1 + ( \| \phi \|_{H^{k_1}} + \| \psi \|_{H^{k_1}} )^{k_1 - 1} \big{]} \\
    & + ( 1 + \| \phi \|^2_{L^\infty} + \| \psi \|^2_{L^\infty} ) ( \| \phi \|_{L^\infty} \| \nabla_x \phi \|_{L^4} + \| \psi \|_{L^\infty} \| \nabla_x \psi \|_{L^4} ) \\
     & \quad \times ( \| \nabla_x \phi \|_{H^{k_1}} + \| \nabla_x \psi \|_{H^{k_1}} ) \big{[} 1 + ( \| \phi \|_{H^{k_1}} + \| \psi \|_{H^{k_1}} )^{k_1 - 2} \big{]} \\
    & + \sum_{\substack{a+b=k_1 \\ a \geq 2}} ( 1 + \| \phi \|^2_{H^s} + \| \psi \|^2_{H^s} ) ( \| \nabla_x \phi \|_{H^a} \| \phi \|_{H^a} + \| \nabla_x \psi \|_{H^a} \| \psi \|_{H^a}  ) \\
    & \qquad \times ( \| \nabla_x \phi \|_{H^{b+1}} + \| \nabla_x \psi \|_{H^{b+1}} ) \big{[} 1 + ( \| \phi \|_{H^{b+1}} + \| \psi \|_{H^{b+1}} )^{b-1} \big{]} \\
    \lesssim &  ( \| \nabla_x \phi \|_{H^s} + \| \nabla_x \psi \|_{H^s} ) \big{[} 1 + ( \| \phi \|_{H^s} + \| \psi \|_{H^s} )^{s+2} \big{]} \,,
  \end{aligned}
\end{equation*}
which immediately implies that
\begin{equation}\label{Apriori-Est-3D-phi-9}
  E_{614} \lesssim |\tfrac{a}{\kappa}| ( \| \nabla_x \phi \|_{H^s} + \| \nabla_x \psi \|_{H^s} ) \| \widehat{\rho} \|_{H^s} \| \phi \|_{H^s} \big{[} 1 + ( \| \phi \|_{H^s} + \| \psi \|_{H^s} )^{s+2} \big{]} \,.
\end{equation}

Now we estimate the term $E_{615}$. H\"older inequality and Sobolev embedding theory reduce to
\begin{equation*}
  \begin{aligned}
    E_{615} = & \tfrac{a}{4 \kappa} \l \nabla_x^{k+1} (W^2 \Omega_\phi) \widehat{\rho} , \nabla_x^k \phi \r + \tfrac{a}{4 \kappa} \l \nabla_x^{k} (W^2 \Omega_\phi) \widehat{\rho} , \nabla_x^{k+1} \phi \r \\
    \lesssim & | \tfrac{a}{\kappa} | \| \widehat{\rho} \|_{L^\infty} \| \nabla_x^k \phi \|_{L^2} \| \nabla_x^{k+1} (W^2 \Omega_\phi) \|_{L^2} + | \tfrac{a}{\kappa} | \| \widehat{\rho} \|_{L^\infty} \| \nabla_x^{k+1} \phi \|_{L^2} \| \nabla_x^{k} (W^2 \Omega_\phi) \|_{L^2} \\
    \lesssim & | \tfrac{a}{\kappa} | \| \widehat{\rho} \|_{H^s} \| \phi \|_{H^s} \| \nabla_x^{k+1} (W^2 \Omega_\phi) \|_{L^2} + | \tfrac{a}{\kappa} | \| \widehat{\rho} \|_{H^s} \| \nabla_x \phi \|_{H^s} \| \nabla_x^{k} (W^2 \Omega_\phi) \|_{L^2}.
  \end{aligned}
\end{equation*}
In order to dominate the term $E_{615}$, we require to estimate the $L^2$-norms $\| \nabla_x^{k+1} (W^2 \Omega_\phi) \|_{L^2}$ and $\| \nabla_x^{k} (W^2 \Omega_\phi) \|_{L^2}$ by utilizing the inequalities shown in Lemma \ref{Lm-Auxil-Ineq}, H\"older inequality and Sobolev embedding theory. More precisely, the calculation is displayed as follows:
  \begin{align*}
    \| \nabla_x^{k+1} (W^2 \Omega_\phi) \|_{L^2} \lesssim & \| \Omega_\phi \|_{L^\infty} \| \nabla_x^{k+1} W^2 \|_{L^2} + \| W^2 \|_{L^\infty} \| \nabla_x^{k+1} \Omega_\phi \|_{L^2} + \| \nabla_x W^2 \|_{L^\infty} \| \nabla_x^k \Omega_\phi \|_{L^2} \\
    & + \| \nabla_x \Omega_\phi \|_{L^\infty} \| \nabla_x^k W^2 \|_{L^2} + \sum_{\substack{k_1+k_2=k+1 \\ k_1, k_2 \geq 2}} \| \nabla_x^{k_1} W^2 \|_{L^4} \| \nabla_x^{k_2} \Omega_\phi \|_{L^4} \\
    \lesssim & ( 1 + \| \phi \|^2_{H^s} + \| \psi \|^2_{H^s} ) ( \| \phi \|_{H^s} + \| \psi \|_{H^s} ) ( \| \nabla_x \phi \|_{H^s} + \| \nabla \psi \|_{H^s} ) \\
    & + ( 1 + \| \phi \|^4_{H^s} + \| \psi \|^4_{H^s}  ) ( \| \nabla_x \phi \|_{H^s} + \| \nabla_x \psi \|_{H^s} ) \big{[} 1 + ( \| \phi \|_{H^s} + \| \psi \|_{H^s} )^s \big{]} \\
    & + ( \| \nabla_x \phi \|_{H^s} + \| \nabla_x \psi \|_{H^s} ) ( 1 + \| \phi \|^2_{H^s} + \| \psi \|^2_{H^s} ) ( \| \phi \|_{H^s} + \| \psi \|_{H^s} )^2 \\
    & + \sum_{\substack{k_1+k_2=k+1 \\ k_1, k_2 \geq 2}} ( 1 + \| \phi \|^2_{H^s} + \| \psi \|^2_{H^s} ) ( \| \nabla_x \phi \|_{H^{k_1}} \| \phi \|_{H^{k_1}} + \| \nabla_x \psi \|_{H^{k_1}} \| \psi \|_{H^{k_1}} ) \\
    & \qquad \times ( \| \nabla_x \phi \|_{H^{k_2}} + \| \nabla_x \psi \|_{H^{k_2}} ) \sum_{I=1}^{k_2} ( \| \phi \|_{H^{k_2}} + \| \psi \|_{H^{k_2}} )^{I-1} \\
    \lesssim & ( \| \nabla_x \phi \|_{H^s} + \| \nabla_x \psi \|_{H^s} ) ( \| \phi \|_{H^s} + \| \psi \|_{H^s} ) \big{[} 1 + ( \| \phi \|_{H^s} + \| \psi \|_{H^s} )^{s+4} \big{]}\,,
  \end{align*}
where we use the bound $|\Omega_\phi| + |\Omega_{\phi \phi }| + |\Omega_{\phi \psi}| \leq C$ in the second inequality. Furthermore, by making use of the same arguments in the above estimation of the $L^2$-norm $\| \nabla_x^{k+1} (W^2 \Omega_\phi) \|_{L^2}$, one can also derive the bound of the $L^2$-norm $\| \nabla_x^{k} (W^2 \Omega_\phi) \|_{L^2}$ as
\begin{equation*}
  \| \nabla_x^{k} (W^2 \Omega_\phi) \|_{L^2} \lesssim ( \| \phi \|_{H^s} + \| \psi \|_{H^s} )^2 \big{[} 1 + ( \| \phi \|_{H^s} + \| \psi \|_{H^s} )^{s+3} \big{]}\,.
\end{equation*}
Therefore, the estimation of the term $E_{615}$ is
\begin{equation}\label{Apriori-Est-3D-phi-10}
  E_{615} \lesssim |\tfrac{a}{\kappa}| ( \| \nabla_x \phi \|_{H^s} + \| \nabla_x \psi \|_{H^s} ) \| \widehat{\rho} \|_{H^s}  ( \| \phi \|_{H^s} + \| \psi \|_{H^s} ) \big{[} 1 + ( \| \phi \|_{H^s} + \| \psi \|_{H^s} )^{s+4} \big{]}\,.
\end{equation}

We plug the inequalities \eqref{Apriori-Est-3D-phi-6}, \eqref{Apriori-Est-3D-phi-7}, \eqref{Apriori-Est-3D-phi-8}, \eqref{Apriori-Est-3D-phi-9} and \eqref{Apriori-Est-3D-phi-10} into the relation \eqref{Apriori-Est-3D-phi-5}, and then we get the estimation of the term $E_{61}$:
\begin{equation}\label{Apriori-Est-3D-phi-11}
  \begin{aligned}
    E_{61} \lesssim & |\tfrac{a}{\kappa}| ( \| \nabla_x \phi \|_{H^s} + \| \nabla_x \psi \|_{H^s} ) \| \widehat{\rho} \|_{H^s}  ( \| \phi \|_{H^s} + \| \psi \|_{H^s} ) \big{[} 1 + ( \| \phi \|_{H^s} + \| \psi \|_{H^s} )^{s+5} \big{]} \,.
  \end{aligned}
\end{equation}

We now estimate the term $E_{62}$. It can be decomposed as
\begin{equation}\label{Apriori-Est-3D-phi-12}
  \begin{aligned}
    E_{62} = & 2 \gamma \l \nabla_x \nabla_x^k \widehat{\rho} \cdot \nabla_x \phi , \nabla_x \phi \r + 2 \gamma \l \nabla_x \widehat{\rho} \cdot \nabla_x^{k+1} \phi , \nabla_x^k \phi \r + 2 \gamma \l \nabla_x^k \widehat{\rho} \cdot \nabla_x^2 \phi , \nabla_x^k \phi \r \\
    & + 2 \gamma \l \nabla_x^2 \widehat{\rho} \cdot \nabla_x^k \phi ,\nabla_x^k \phi \r + 2 \gamma \sum_{\substack{k_1+k_2=k \\ k_1, k_2 \geq 2}} \l \nabla_x^{k_1 + 1} \widehat{\rho} \cdot \nabla_x^{k_2+1} \phi , \nabla_x^k \phi \r \\
    \equiv & \, E_{621} + E_{622} + E_{623} + E_{624} + E_{625} \,,
  \end{aligned}
\end{equation}
which can be estimated term by term. For the term $E_{621}$, we have
\begin{equation}\label{Apriori-Est-3D-phi-13}
  \begin{aligned}
    E_{621} = & - 2 \gamma \l \nabla_x^k \widehat{\rho} \cdot \nabla_x^2 \phi , \nabla_x^k \phi \r - 2 \gamma \l \nabla_x^k \widehat{\rho} \cdot \nabla_x \phi , \nabla_x^{k+1} \phi \r \\
    \lesssim & \gamma \| \nabla_x^k \widehat{\rho} \|_{L^2} \| \nabla_x^2 \phi \|_{L^4} \| \nabla_x^k \phi \|_{L^4} + \gamma  \| \nabla_x^k \widehat{\rho} \|_{L^2} \| \nabla_x \phi \|_{L^\infty} \| \nabla_x^{k+1} \phi \|_{L^2} \\
    \lesssim & \gamma \| \widehat{\rho} \|_{H^s} \| \phi \|_{H^3} \| \nabla_x \phi \|_{H^s} + \gamma \| \widehat{\rho} \|_{H^s} \| \phi \|_{H^3} \| \nabla_x \phi \|_{H^s} \\
    \lesssim & \gamma \| \nabla_x \phi \|_{H^s} \| \widehat{\rho} \|_{H^s} \| \phi \|_{H^s}\,.
  \end{aligned}
\end{equation}
For the remained terms $E_{622}, E_{623}, E_{624}$ and $E_{625}$, by the similar estimation of $E_{621},$ we have
\begin{equation}\label{Apriori-Est-3D-phi-14}
  E_{622} + E_{623} + E_{624} + E_{625} \lesssim \gamma \| \nabla_x \phi \|_{H^s} \| \widehat{\rho} \|_{H^s} \| \phi \|_{H^s}\,.
\end{equation}
Consequently, the inequalities \eqref{Apriori-Est-3D-phi-12}, \eqref{Apriori-Est-3D-phi-13} and \eqref{Apriori-Est-3D-phi-14} imply that
\begin{equation}\label{Apriori-Est-3D-phi-15}
  E_{62} \lesssim \gamma \| \nabla_x \phi \|_{H^s} \| \widehat{\rho} \|_{H^s} \| \phi \|_{H^s}\,.
\end{equation}
We remark that the condition $s \geq 3$ is {\em necessary} in the estimation of the term $E_{62}$.

Now we give the estimation of the term $E_{63}$, which can be decomposed as three parts:
\begin{equation}\label{Apriori-Est-3D-phi-16}
  \begin{aligned}
    E_{63} = & - 2 \gamma \l \tfrac{\phi}{W} \nabla_x^k | \nabla_x \phi |^2 , \nabla_x^k \phi \r - 2 \gamma \l \nabla_x^k ( \tfrac{\phi}{W} ) |\nabla_x^k \phi|^2 , \nabla_x^k \phi \r \\
    & \qquad - 2 \gamma \sum_{\substack{k_1+k_2=k \\ k_1, k_2 \geq 1}} \l \nabla_x^{k_1} (  \tfrac{\phi}{W} ) \nabla_x^{k_2} |\nabla_x \phi|^2 , \nabla_x^k \phi \r \\
    \equiv & \, E_{631} + E_{632} + E_{633} \,.
  \end{aligned}
\end{equation}
For the term $E_{631}$, by making use of the H\"older inequality and Sobolev embedding $H^2(\R^3) \hookrightarrow L^\infty(\R^3)$ and $H^1 (\R^3) \hookrightarrow L^4(\R^3)$, we have
\begin{equation}\label{Apriori-Est-3D-phi-17}
  \begin{aligned}
    E_{631} \lesssim & \gamma \| \nabla_x \phi \|_{L^\infty} \| \nabla_x^{k+1} \phi \|_{L^2} \| \nabla_x^k \phi \|_{L^2} + \gamma \sum_{\substack{k_1+k_2=k \\ k_1,k_2 \geq 1}} \| \nabla_x^{k_1 + 1} \phi \|_{L^4} \| \nabla_x^{k_2 + 1} \phi \|_{L^4} \| \nabla_x^k \phi \|_{L^2} \\
    \lesssim & \gamma \| \nabla_x \phi \|_{H^2} \| \nabla_x \phi \|_{H^s} \| \phi \|_{H^s} + \gamma \sum_{\substack{k_1+k_2=k \\ k_1,k_2 \geq 1}} \| \nabla_x^{k_1+1} \phi \|_{H^1} \| \nabla_x^{k_2 + 1} \phi \|_{H^1} \| \phi \|_{H^s} \\
    \lesssim & \gamma \| \phi \|^2_{H^s} \| \nabla_x \phi \|_{H^s}
  \end{aligned}
\end{equation}
for $s \geq 3$. For the term $E_{632}$, the H\"older inequality and Sobolev embedding theory imply that
\begin{equation*}
  \begin{aligned}
    E_{632} \lesssim \gamma \| \nabla_x \phi \|^2_{L^\infty} \| \nabla_x^k ( \tfrac{\phi}{W} ) \|_{L^2} \| \nabla_x^k \phi \|_{L^2} \lesssim \gamma \| \nabla_x \phi \|^2_{H^2} \| \phi \|_{H^s} \| \nabla_x^k ( \tfrac{\phi}{W} ) \|_{L^2} \,.
  \end{aligned}
\end{equation*}
Noticing that $ | \nabla_x^k ( \tfrac{\phi}{W} ) | \leq \tfrac{1}{2} |\nabla_x^k \Omega| $, we derive from Lemma \ref{Lm-Auxil-Ineq} that
\begin{equation*}
  \begin{aligned}
    \| \nabla_x^k ( \tfrac{\phi}{W} ) \|_{L^2} \lesssim & ( \| \nabla_x \phi \|_{H^{k-1}} + \| \nabla_x \psi \|_{H^{k-1}} ) \sum_{I=1}^k ( \| \phi \|_{H^{k-1}} + \| \psi \|_{H^{k-1}} )^I \\
    \lesssim & ( \| \phi \|_{H^s} + \| \psi \|_{H^s} ) \big{[} 1 + ( \| \phi \|_{H^s} + \| \psi \|_{H^s} )^{s-1} \big{]} \,,
  \end{aligned}
\end{equation*}
which immediately yields that
\begin{equation}\label{Apriori-Est-3D-phi-18}
  E_{632} \lesssim \gamma \| \phi \|^2_{H^s} \| \nabla_x \phi \|_{H^s} ( \| \phi \|_{H^s} + \| \psi \|_{H^s} ) \big{[} 1 + ( \| \phi \|_{H^s} + \| \psi \|_{H^s} )^{s-1} \big{]} \,.
\end{equation}
For the term $E_{633}$, it is implied that by Lemma \ref{Lm-Auxil-Ineq}
\begin{equation*}
  \begin{aligned}
    E_{633} \lesssim & \gamma \sum_{\substack{k_1+k_2=k \\ k_1 , k_2 \geq 1}} \| \nabla_x^{k_1} ( \tfrac{\phi}{W} ) \|_{L^4} \| \nabla_x^{k_2} |\nabla_x \phi|^2 \|_{L^4} \| \nabla_x^k \phi \|_{L^2} \\
    \lesssim & \gamma \sum_{1 \leq k_2 \leq k-1} \| \nabla_x^{k_2} |\nabla_x \phi|^2 \|_{L^4} \| \phi \|_{H^s} \sum_{I=1}^{s-1} ( \| \phi \|_{H^s} + \| \psi \|_{H^s} )^{I} \,.
  \end{aligned}
\end{equation*}
For $ 1 \leq k_2 \leq k-1 ( \leq s-1 ) $, if $s \geq 3$, then we have
\begin{equation*}
  \begin{aligned}
    & \| \nabla_x^{k_2} |\nabla_x \phi|^2 \|_{L^4} \lesssim  \| \nabla_x \phi \|_{L^\infty} \| \nabla_x^{k_2+1} \phi \|_{L^4} + \sum_{\substack{a+b=k_2 \\ a,b \geq 1}} \| \nabla_x^{a+1} \phi \|_{L^\infty} \| \nabla_x^{b+1} \phi \|_{L^4} \\
    \lesssim & \| \phi \|_{H^3} \| \nabla_x \phi \|_{H^s} + \sum_{\substack{a+b=k_2 \\ a,b \geq 1}} \| \nabla_x^{a+1} \phi \|_{H^2} \| \nabla_x^{b+1} \phi \|_{H^1} \lesssim \| \phi \|_{H^s} \| \nabla_x \phi \|_{H^s} \,,
  \end{aligned}
\end{equation*}
which means that
\begin{equation}\label{Apriori-Est-3D-phi-19}
  E_{633} \lesssim \gamma \| \nabla_x \phi \|_{H^s} \| \phi \|^2_{H^s} ( \| \phi \|_{H^s} + \| \psi \|_{H^s} ) \big{[} 1 + ( \| \phi \|_{H^s} + \| \psi \|_{H^s} )^{s-2} \big{]} \,.
\end{equation}
Thus, substituting the inequalities \eqref{Apriori-Est-3D-phi-17}, \eqref{Apriori-Est-3D-phi-18} and \eqref{Apriori-Est-3D-phi-19} into \eqref{Apriori-Est-3D-phi-16} implies that
\begin{equation}\label{Apriori-Est-3D-phi-20}
  E_{63} \lesssim \gamma \| \nabla_x \phi \|_{H^s} \| \phi \|^2_{H^s} \big{[} 1 + ( \| \phi \|_{H^s} + \| \psi \|_{H^s} )^{s} \big{]} \,.
\end{equation}
Furthermore, by the similar estimation of the term $E_{63}$, we have
\begin{equation}\label{Apriori-Est-3D-phi-21}
  E_{64} + E_{65} \lesssim \gamma ( \| \nabla_x \phi \|_{H^s} + \| \nabla_x \psi \|_{H^s} ) (\| \phi \|^2_{H^s} + \| \psi \|^2_{H^s} ) \big{[} 1 + ( \| \phi \|_{H^s} + \| \psi \|_{H^s} )^{s} \big{]} \,.
\end{equation}

Now we estimate the term $E_{66}$. For convenience of calculation, we decompose it as seven parts:
  \begin{align}\label{Apriori-Est-3D-phi-21-66}
    \no E_{66} = & \tfrac{1}{4} \l \nabla_x^k W^2 ( \widetilde{\lambda} S(v) + A(v) ) : \Omega_\phi \otimes \Omega, \nabla_x^k \phi \r + \tfrac{1}{4} \l W^2 \nabla_x^k ( \widetilde{\lambda} S(v) + A(v) ) : \Omega_\phi \otimes \Omega, \nabla_x^k \phi \r \\
    \no & + \tfrac{1}{4} \l W^2  ( \widetilde{\lambda} S(v) + A(v) ) : \nabla_x^k ( \Omega_\phi \otimes \Omega ) , \nabla_x^k \phi \r \\
    \no & + \tfrac{1}{4} \sum_{\substack{k_1+k_2=k \\ k_1, k_2 \geq 1}} \l \nabla_x^{k_1} W^2 \nabla_x^{k_2} ( \widetilde{\lambda} S(v) + A(v) ) : \Omega_\phi \otimes \Omega, \nabla_x^k \phi \r \\
     & + \tfrac{1}{4} \sum_{\substack{k_1+k_3=k \\ k_1, k_3 \geq 1}} \l \nabla_x^{k_1} W^2 ( \widetilde{\lambda} S(v) + A(v) ) : \nabla_x^{k_3} (  \Omega_\phi \otimes \Omega ) , \nabla_x^k \phi \r \\
    \no & + \tfrac{1}{4} \sum_{\substack{k_2+k_3=k \\ k_2, k_3 \geq 1}} \l  W^2 \nabla_x^{k_2} ( \widetilde{\lambda} S(v) + A(v) ) : \nabla_x^{k_3} (  \Omega_\phi \otimes \Omega ) , \nabla_x^k \phi \r \\
    \no & + \tfrac{1}{4} \sum_{\substack{k_1+k_2+k_3=k \\ k_2, k_3 \geq 1}} \l \nabla_x^{k_1} W^2 \nabla_x^{k_2} ( \widetilde{\lambda} S(v) + A(v) ) : \nabla_x^{k_3} (  \Omega_\phi \otimes \Omega ) , \nabla_x^k \phi \r \\
    \no \equiv & \, E_{661} + E_{662} + E_{663} + E_{664} + E_{665} + E_{666} + E_{667} \,.
  \end{align}
  For the first two terms $ E_{661} $ and $ E_{662} $, it is derived from Lemma \ref{Lm-Auxil-Ineq} that
    \begin{align}\label{Apriori-Est-3D-phi-22}
      \no E_{661} + E_{662} \lesssim & ( |\widetilde{\lambda}| + 1 ) \Big{[} \| \nabla_x^k \phi \|_{L^2} \| \nabla_x v \|_{L^\infty} \| \nabla_x^k W^2 \|_{L^2} +  \| W^2 \|_{L^\infty} \| \nabla_x^{k+1} v \|_{L^2} \| \nabla_x^k \phi \|_{L^2} \Big{]} \\
      \no \lesssim & ( |\widetilde{\lambda}| + 1 ) \| \phi \|_{H^s} \| \nabla_x v \|_{H^s} (  1 + \| \phi \|^2_{H^s} + \| \psi \|^2_{H^s} ) ( \| \phi \|_{H^s} + \| \psi \|_{H^s} )^2 \\
       & + ( |\widetilde{\lambda}| + 1 ) ( 1 + \| \phi \|^4_{H^s} + \| \psi \|^4_{H^s} ) \| \phi \|_{H^s} \| \nabla_x v \|_{H^s} \\
       \no \lesssim &  ( |\widetilde{\lambda}| + 1 ) \| \nabla_x v \|_{H^s} \| \phi \|_{H^s} ( 1 + \| \phi \|^4_{H^s} + \| \psi \|^4_{H^s} ) \,.
    \end{align}
  By Sobolev embedding theory and H\"older inequality, we estimate the term $E_{663}$ that
    \begin{align*}
      E_{663} \lesssim & ( |\widetilde{\lambda}| + 1 ) \| W^2 \|_{L^\infty} \| \nabla_x v \|_{L^\infty} \| \nabla_x^k \phi \|_{L^2} \| \nabla_x^k ( \Omega_\phi \otimes \Omega ) \|_{L^2} \\
      \lesssim & ( |\widetilde{\lambda}| + 1 ) ( 1 + \| \phi \|^4_{H^s} + \| \psi \|^4_{H^s} ) \| \nabla_x v \|_{H^s} \| \phi \|_{H^s} \| \nabla_x^k ( \Omega_\phi \otimes \Omega ) \|_{L^2} \,.
    \end{align*}
  The norm $ \| \nabla_x^k ( \Omega_\phi \otimes \Omega ) \|_{L^2} $ is easy to be derived from Lemma \ref{Lm-Auxil-Ineq} as follows:
    \begin{align*}
      & \| \nabla_x^k ( \Omega_\phi \otimes \Omega ) \|_{L^2} = \| \nabla_x^k \Omega_\phi \otimes \Omega \|_{L^2} + \| \Omega_\phi \otimes \nabla_x^k \Omega \|_{L^2} + \sum_{\substack{k_1+k_2=k \\ k_1, k_2 \geq 1}} \| \nabla_x^{k_1} \Omega_\phi \otimes \nabla_x^{k_2} \Omega \|_{L^2} \\
      \lesssim & \| \nabla_x^k \Omega_\phi \|_{L^2} + \| \nabla_x^k \Omega \|_{L^2} + \sum_{\substack{k_1+k_2=k \\ k_1, k_2 \geq 1}} \| \nabla_x^{k_1} \Omega_\phi \|_{L^4} \| \nabla_x^{k_2} \Omega \|_{L^4} \\
      \lesssim & \sum_{I=1}^k ( \| \phi \|_{H^s} + \| \psi \|_{H^s} )^I  + \sum_{\substack{k_1+k_2=k \\ k_1, k_2 \geq 1}} ( \| \nabla_x \phi \|_{H^{k_1}} + \| \nabla_x \psi \|_{H^{k_2}} ) \sum_{I=1}^{k_1} ( \| \phi \|_{H^{k_1}} + \| \psi \|_{H^{k_1}} )^{I-1} \\
      & \qquad \qquad \times ( \| \nabla_x \phi \|_{H^{k_2}} + \| \nabla_x \psi \|_{H^{k_2}} ) \sum_{I=1}^{k_2} ( \| \phi \|_{H^{k_2}} + \| \psi \|_{H^{k_2}} )^{I-1} \\
      \lesssim & ( \| \phi \|_{H^s} + \| \psi \|_{H^s} ) \big{[} 1 + ( \| \phi \|_{H^s} + \| \psi \|_{H^s} )^{s-1} \big{]} \,,
    \end{align*}
  where we utilize the bound $|\Omega| + |\Omega_\phi| \leq C$. As a consequence, we know that
  \begin{equation}\label{Apriori-Est-3D-phi-23}
    E_{663} \lesssim  ( |\widetilde{\lambda}| + 1 ) \| \nabla_x v \|_{H^s} ( \| \phi \|_{H^s} + \| \psi \|_{H^s} ) \big{[} 1 + ( \| \phi \|_{H^s} + \| \psi \|_{H^s} )^{s+1} \big{]} \,.
  \end{equation}
  For the term $E_{664}$, we calculate that by using Lemma \ref{Lm-Auxil-Ineq}
    \begin{align}\label{Apriori-Est-3D-phi-24}
      \no E_{664} \lesssim & ( |\widetilde{\lambda}| + 1 ) \sum_{\substack{k_1+k_2=k \\ k_1, k_2 \geq 1}} \| \nabla_x^{k_1} W^2 \|_{L^4} \| \nabla_x^{k_2+1} v \|_{L^4} \| \nabla_x^k \phi \|_{L^2} \\
      \no \lesssim &  ( |\widetilde{\lambda}| + 1 ) \| \phi \|_{H^s} \| \nabla_x v \|_{H^s} \sum_{1 \leq k_1 \leq k - 1} ( 1 + \| \phi \|^2_{H^{k_1}} + \| \psi \|^2_{H^{k_1}} + \| \phi \|^2_{H^3} + \| \psi \|^2_{H^3} ) \\
       & \qquad \times ( \| \nabla_x \phi \|_{H^{k_1}} \| \phi \|_{H^{k_1}} + \| \nabla_x \psi \|_{H^{k_1}} \| \psi \|_{H^{k_1}} ) \\
      \no \lesssim & ( |\widetilde{\lambda}| + 1 ) \| \nabla_x v \|_{H^s} \| \phi \|_{H^s} ( \| \phi \|_{H^s} + \| \psi \|_{H^s} )^2 ( 1 + \| \phi \|^2_{H^s} + \| \psi \|^2_{H^s}  ) \,.
    \end{align}
  For the term $ E_{665} $, we observe that by the inequalities shown in Lemma \ref{Lm-Auxil-Ineq}
  \begin{equation*}
    \begin{aligned}
      E_{665} \lesssim & ( |\widetilde{\lambda}| + 1 ) \| \nabla_x v \|_{L^\infty} \| \nabla_x^k \phi \|_{L^2} \sum_{\substack{k_1+k_3=k \\ k_1, k_3 \geq 1}} \| \nabla_x^{k_1} W^2 \|_{L^4} \| \nabla_x^{k_3} ( \Omega_\phi \otimes \Omega ) \|_{L^4} \\
      \lesssim & ( |\widetilde{\lambda}| + 1 ) \| \nabla_x v \|_{H^s} \| \phi \|_{H^s} ( \| \phi \|^2_{H^s} + \| \psi \|^2_{H^s} ) ( 1 + \| \phi \|^2_{H^s} + \| \psi \|^2_{H^s} ) \\
       & \qquad \times \sum_{1 \leq k_3 \leq k-1} \| \nabla_x^{k_3} ( \Omega_\phi \otimes \Omega ) \|_{L^4} \,,
    \end{aligned}
  \end{equation*}
  where the estimation of the quantity $ \sum\limits_{1 \leq k_3 \leq k-1} \| \nabla_x^{k_3} ( \Omega_\phi \otimes \Omega ) \|_{L^4} $
  \begin{equation*}
    \sum_{1 \leq k_3 \leq k-1} \| \nabla_x^{k_3} ( \Omega_\phi \otimes \Omega ) \|_{L^4} \lesssim ( \| \phi \|_{H^s} + \| \psi \|_{H^s} ) \big{[} 1 + ( \| \phi \|_{H^s} + \| \psi \|_{H^s} )^{s-2}
  \end{equation*}
  is implied by the inequalities shown in Lemma \ref{Lm-Auxil-Ineq}. Consequently, we gain the estimation of the term $E_{665}$
  \begin{equation}\label{Apriori-Est-3D-phi-25}
    E_{665} \lesssim ( |\widetilde{\lambda}| + 1 ) \| \nabla_x v \|_{H^s} \| \phi \|_{H^s} ( \| \phi \|_{H^s} + \| \psi \|_{H^s} )^3 \big{[} 1 + ( \| \phi \|_{H^s} + \| \psi \|_{H^s} )^{s} \big{]} \,.
  \end{equation}
  Following the analogous arguments in the estimation of the inequalities \eqref{Apriori-Est-3D-phi-24} and \eqref{Apriori-Est-3D-phi-25} reduces to
  \begin{equation}\label{Apriori-Est-3D-phi-26}
    E_{666} \lesssim ( |\widetilde{\lambda}| + 1 ) \| \nabla_x v \|_{H^s} \| \phi \|_{H^s} ( \| \phi \|_{H^s} + \| \psi \|_{H^s} ) \big{[} 1 + ( \| \phi \|_{H^s} + \| \psi \|_{H^s} )^{s+2} \big{]} \,.
  \end{equation}
  For the estimation of the term $E_{667}$, we can gain by using the H\"older inequality, Sobolev embedding theory and the inequalities shown in Lemma \ref{Lm-Auxil-Ineq}:
  \begin{equation}\label{Apriori-Est-3D-phi-27}
    \begin{aligned}
      E_{667} \lesssim & ( |\widetilde{\lambda}| + 1 ) \sum_{\substack{k_1+k_2+k_3=k \\ k_1, k_2, k_3 \geq 1}} \| \nabla_x^{k_2 + 1} v \|_{L^\infty} \| \nabla_x^{k_1} W^2 \|_{L^4} \| \nabla_x^{k_3} ( \Omega_\phi \otimes \Omega ) \|_{L^4} \| \nabla_x^k \phi \|_{L^2} \\
      \lesssim & ( |\widetilde{\lambda}| + 1 ) \| \nabla_x v \|_{H^s} \| \phi \|_{H^s} ( \| \phi \|_{H^s} + \| \psi \|_{H^s} ) \big{[} 1 + ( \| \phi \|_{H^s} + \| \psi \|_{H^s} )^{s-2} \big{]} \\
      & \quad \times \sum_{1 \leq k_1 \leq k-2} ( 1 + \| \phi \|^2_{H^{k_1}} + \| \psi \|^2_{H^{k_1}} + \| \phi \|^2_{H^3} + \| \psi \|^2_{H^3} ) \\
       & \qquad \qquad \qquad \times ( \| \nabla_x \phi \|_{H^{k_1}} \| \phi \|_{H^{k_1}} + \| \nabla_x \psi \|_{H^{k_1}} \| \psi \|_{H^{k_1}} ) \\
      \lesssim & ( |\widetilde{\lambda}| + 1 ) \| \nabla_x v \|_{H^s} ( \| \phi \|_{H^s} + \| \psi \|_{H^s} )^2 \big{[} 1 + ( \| \phi \|_{H^s} + \| \psi \|_{H^s} )^{s} \big{]} \,.
    \end{aligned}
  \end{equation}

  Thus, by plugging the inequalities \eqref{Apriori-Est-3D-phi-22}, \eqref{Apriori-Est-3D-phi-23}, \eqref{Apriori-Est-3D-phi-24},  \eqref{Apriori-Est-3D-phi-25}, \eqref{Apriori-Est-3D-phi-26} and \eqref{Apriori-Est-3D-phi-27} into \eqref{Apriori-Est-3D-phi-21-66}, we have
  \begin{equation}\label{Apriori-Est-3D-phi-28}
    E_{66} \lesssim ( |\widetilde{\lambda}| + 1 ) \| \nabla_x v \|_{H^s} ( \| \phi \|_{H^s} + \| \psi \|_{H^s} ) \big{[} 1 + ( \| \phi \|_{H^s} + \| \psi \|_{H^s} )^{s+3} \big{]} \,.
  \end{equation}

  Furthermore, it is derived from substituting the inequalities \eqref{Apriori-Est-3D-phi-11}, \eqref{Apriori-Est-3D-phi-15}, \eqref{Apriori-Est-3D-phi-20}, \eqref{Apriori-Est-3D-phi-21} and \eqref{Apriori-Est-3D-phi-28} into \eqref{Apriori-Est-3D-phi-4} that
    \begin{align}\label{Apriori-Est-3D-phi-29}
      \no & E_6 + \gamma \| \nabla_x^{k+1} \phi \|^2_{L^2} =  \gamma \| \nabla_x^{k+1} \phi \|^2_{L^2} - \l \nabla_x^k \mathcal{H}_\phi ( \widehat{\rho}, \phi , \psi, v ) , \nabla_x^k \phi \r \\
      \no \lesssim &  ( |\tfrac{a}{\kappa} | +  \gamma  ) ( \| \nabla_x \phi \|_{H^s} + \| \nabla_x \psi \|_{H^s} ) ( \| \phi \|_{H^s} + \| \psi \|_{H^s} + \| \widehat{\rho} \|_{H^s} ) \big{[} 1 + ( \| \phi \|_{H^s} + \| \psi \|_{H^s} )^{s+4} \big{]} \\
       & + ( |\widetilde{\lambda}| + 1 ) \| \nabla_x v \|_{H^s} ( \| \phi \|_{H^s} + \| \psi \|_{H^s} ) \big{[} 1 + ( \| \phi \|_{H^s} + \| \psi \|_{H^s} )^{s+3} \big{]} \,.
    \end{align}
  Therefore, we plug the inequalities \eqref{Apriori-Est-3D-phi-2}, \eqref{Apriori-Est-3D-phi-3} and \eqref{Apriori-Est-3D-phi-29} into the relation \eqref{Apriori-Est-3D-phi-1}, sum up for all integer $0 \leq k \leq s$ and then we gain the energy estimate of the $\phi$-equation of the SOH-NS system \eqref{SOH-NS-3D-SPT-Smp}
    \begin{align}\label{Apriori-Est-3D-phi}
      \no & \tfrac{1}{2} \tfrac{\d}{\d t} \| \phi \|^2_{H^s} + \gamma \| \nabla_x \phi \|^2_{H^s} \\
      \no \lesssim & ( |a c_2| + |\tfrac{a}{\kappa} | +  \gamma  ) ( \| \nabla_x \phi \|_{H^s} + \| \nabla_x \psi \|_{H^s} ) ( \| \phi \|_{H^s} + \| \psi \|_{H^s} + \| \widehat{\rho} \|_{H^s} ) \\
       & \qquad \qquad \times \big{[} 1 + ( \| \phi \|_{H^s} + \| \psi \|_{H^s} )^{s+4} \big{]} \\
       \no & + ( |\widetilde{\lambda}| + 1 ) \| \nabla_x v \|_{H^s} ( \| \phi \|_{H^s} + \| \psi \|_{H^s} ) \big{[} 1 + ( \| \phi \|_{H^s} + \| \psi \|_{H^s} )^{s+3} \big{]} \,.
    \end{align}

{\smallskip\noindent \em\large Step 3. Energy estimate for $\psi$-equation.}

From the geometric view, the two components $\phi$ and $\psi$ in the stereographic projection transform are symmetric in geometric structure, which results to that the $\psi$-equation in the SOH-NS system \eqref{SOH-NS-3D-SPT-Smp} is of the same form of the $\phi$-equation. Consequently, by following the similar arguments in the energy estimate of $\phi$-equation in {\em Step 2}, we can derive the energy estimate for $\psi$-equation of the SOH-NS system \eqref{SOH-NS-3D-SPT-Smp}
  \begin{equation}\label{Apriori-Est-3D-psi}
    \begin{aligned}
      & \tfrac{1}{2} \tfrac{\d}{\d t} \| \psi \|^2_{H^s} + \gamma \| \nabla_x \psi \|^2_{H^s} \\
      \lesssim & ( |a c_2| + |\tfrac{a}{\kappa} | +  \gamma  ) ( \| \nabla_x \phi \|_{H^s} + \| \nabla_x \psi \|_{H^s} ) ( \| \phi \|_{H^s} + \| \psi \|_{H^s} + \| \widehat{\rho} \|_{H^s} ) \\
       & \qquad \qquad \times \big{[} 1 + ( \| \phi \|_{H^s} + \| \psi \|_{H^s} )^{s+4} \big{]} \\
       & + ( |\widetilde{\lambda}| + 1 ) \| \nabla_x v \|_{H^s} ( \| \phi \|_{H^s} + \| \psi \|_{H^s} ) \big{[} 1 + ( \| \phi \|_{H^s} + \| \psi \|_{H^s} )^{s+3} \big{]} \,.
    \end{aligned}
  \end{equation}

{\smallskip\noindent \em\large Step 4. Energy estimate for $v$-equation.}

For any integer $0 \leq k \leq s$, we act the derivative operator $\nabla_x^k$ on the $v$-equation of the system \eqref{SOH-NS-3D-SPT-Smp}, take $L^2$-inner product by dot product with $\nabla_x^k v$, and integrate by parts over $\R^3$, then we gain
  \begin{equation}\label{Apriori-Est-3D-v-1}
    \begin{aligned}
       \tfrac{Re}{2} \tfrac{\d}{\d t} \| \nabla_x^k v \|^2_{L^2} + \| \nabla_x^{k+1} v \|^2_{L^2} = - Re \l \nabla_x^k ( v \cdot \nabla_x v ) , \nabla_x^k v \r - \l \nabla_x^k [ e^{\widehat{\rho}} \mathcal{G} ( \widehat{\rho}, \phi, \psi ) ] , \nabla_x^k v \r \,,
    \end{aligned}
  \end{equation}
  where the first term in the right-hand side of the above equality can be estimated
  \begin{equation}\label{Apriori-Est-3D-v-2}
    \begin{aligned}
      & - Re \l \nabla_x^k ( v \cdot \nabla_x v ) , \nabla_x^k v \r =  - Re \l \nabla_x v \cdot \nabla_x^k v \r - Re \sum_{\substack{a+b=k \\ a \geq 2}} \l \nabla_x^a v \nabla_x^{b+1} v , \nabla_x^k v \r  \\
       \lesssim & Re \| \nabla_x v \|_{L^\infty} \| \nabla_x^k v \|^2_{L^2} + Re \sum_{\substack{a+b=k \\ a \geq 2}} \| \nabla_x^a v \|_{L^4} \| \nabla_x^{b+1} v \|_{L^4} \| \nabla_x^k v \|_{L^2} \\
      \lesssim & Re \| \nabla_x v \|_{H^s} \| v \|^2_{H^s} \,,
    \end{aligned}
  \end{equation}
  and the second term in the right-hand side of the above equality can be decomposed as
    \begin{align}\label{Apriori-Est-3D-v-3}
      \no & - \l \nabla_x^k [ e^{\widehat{\rho}} \mathcal{G} ( \widehat{\rho}, \phi, \psi ) ] , \nabla_x^k v \r = - \l \nabla_x^k [ e^{\widehat{\rho}} \mathcal{Q} ( \Omega ) ] , \nabla_x^k v \r - c_4 \l \nabla_x^k [ e^{\widehat{\rho}} ( \Omega_\phi \cdot \nabla_x \phi ) \Omega ], \nabla_x^k v \r \\
       & - c_4 \l \nabla_x^k [ e^{\widehat{\rho}} ( \Omega_\psi \cdot \nabla_x \psi ) \Omega ], \nabla_x^k v \r - c_4 \l  \nabla_x^k [ e^{\widehat{\rho}} ( \Omega \cdot \nabla_x \phi ) \Omega_\phi ], \nabla_x^k v \r \\
     \no  & - c_4 \l  \nabla_x^k [ e^{\widehat{\rho}} ( \Omega \cdot \nabla_x \psi ) \Omega_\psi ], \nabla_x^k v \r \\
      \no & \equiv \, K_1 + K_2 + K_3 + K_4 + K_5 \,.
    \end{align}
  Here we make use of the definition of the quantity of $\mathcal{G} (\widehat{\rho}, \phi, \psi)$.

  It remains to estimate the terms $K_i \, (1 \leq i \leq 5)$ term by term. For the term $K_1$, we can divide it into seven parts for the convenience of calculation:
  \begin{equation}\label{Apriori-Est-3D-v-4}
    \begin{aligned}
      K_1 = & - \l \nabla_x^k e^{\widehat{\rho}} \nabla_x \widehat{\rho} \cdot \mathcal{Q}(\Omega) , \nabla_x^k v \r - \l e^{\widehat{\rho}} \nabla_x \nabla_x^k \widehat{\rho} \cdot \mathcal{Q}(\Omega) , \nabla_x^k v \r \\
      & - \l  e^{\widehat{\rho}} \nabla_x \widehat{\rho} \cdot \nabla_x^k \mathcal{Q}(\Omega) , \nabla_x^k v \r - \sum_{\substack{k_1+k_2=k \\ k_1, k_2 \geq 1}} \l \nabla_x^{k_1} e^{\widehat{\rho}} \nabla_x^{k_2+1} \widehat{\rho} \cdot \mathcal{Q}(\Omega) , \nabla_x^k v \r \\
      & - \sum_{\substack{k_1+k_3=k \\ k_1, k_3 \geq 1}} \l \nabla_x^{k_1} e^{\widehat{\rho}} \nabla_x \widehat{\rho} \cdot \nabla_x^{k_3} \mathcal{Q}(\Omega) , \nabla_x^k v \r - \sum_{\substack{k_2+k_3=k \\ k_2, k_3 \geq 1}} \l  e^{\widehat{\rho}} \nabla_x^{k_2+1} \widehat{\rho} \cdot \nabla_x^{k_3} \mathcal{Q}(\Omega) , \nabla_x^k v \r \\
      & - \sum_{\substack{k_1+k_2+k_3=k \\ k_1, k_2, k_3 \geq 1}} \l \nabla_x^{k_1} e^{\widehat{\rho}} \nabla_x^{k_2+1} \widehat{\rho} \cdot \nabla_x^{k_3} \mathcal{Q}(\Omega) , \nabla_x^k v \r \\
      \equiv & \, K_{11} + K_{12} + K_{13} + K_{14} + K_{15} + K_{16} + K_{17}\,.
    \end{aligned}
  \end{equation}

  For the term $K_{11}$, it is derived from the inequalities \eqref{Bnds-rho-norms} in Lemma \ref{Lm-rho-L^infty}, H\"older inequality, Sobolev embedding theory and the bound $ |\mathcal{Q}(\Omega)| \leq \tfrac{4}{3} |c_4| $ that
    \begin{align}\label{Apriori-Est-3D-v-5}
      \no K_{11} \lesssim & |c_4| \| \nabla_x \widehat{\rho} \|_{L^\infty} \| \nabla_x^k e^{\widehat{\rho}} \|_{L^2} \| \nabla_x^k v \|_{L^2} \lesssim  |c_4| \| \widehat{\rho} \|_{H^3} \sum_{I=1}^s \| e^{\widehat{\rho}} \|_{L^\infty} \| \widehat{\rho} \|^I_{H^s} \| v \|_{H^s} \\
      \lesssim & |c_4|  \| e^{\widehat{\rho}} \|_{L^\infty} \| v \|_{H^s} \| \widehat{\rho} \|^2_{H^s} ( 1 + \| \widehat{\rho} \|^{s-1}_{H^s} )
    \end{align}
  if $s \geq 3$. For the term $ E_{12} $, by making use of integrating by parts over $\R^3$ and the bound $ |\mathcal{Q}(\Omega)| \leq \tfrac{4}{3} |c_4| $, we have
  \begin{equation}\label{Apriori-Est-3D-v-6}
    \begin{aligned}
      E_{12} = & \l \nabla_x e^{\widehat{\rho}} \nabla_x^k \widehat{\rho} \cdot \mathcal{Q}(\Omega) , \nabla_x^k v \r + \l e^{\widehat{\rho}} \nabla_x^k \widehat{\rho} \cdot \nabla_x  \mathcal{Q}(\Omega) , \nabla_x^k v \r + \l  e^{\widehat{\rho}} \nabla_x^k \widehat{\rho} \cdot \mathcal{Q}(\Omega) , \nabla_x^{k+1} v \r \\
      \lesssim & |c_4| \| e^{\widehat{\rho}} \|_{L^\infty} \| \nabla_x^k \widehat{\rho} \|_{L^2} \Big{[} ( \| \nabla_x \widehat{\rho} \|_{L^\infty}  + \| \nabla_x \phi \|_{L^\infty} + \| \nabla_x \psi \|_{L^\infty} )  \| \nabla_x^k v \|_{L^2} +  \| \nabla_x^{k+1} v \|_{L^2} \Big{]} \\
      \lesssim & |c_4| \| e^{\widehat{\rho}} \|_{L^\infty} \| \widehat{\rho} \|_{H^s} \| v \|_{H^s} ( \| \widehat{\rho} \|_{H^s} +  \| \phi \|_{H^s} + \| \psi \|_{H^s} ) + |c_4| \| e^{\widehat{\rho}} \|_{L^\infty} \| \widehat{\rho} \|_{H^s} \| \nabla_x v \|_{H^s}
    \end{aligned}
  \end{equation}
  if $s \geq 3$. For the term $K_{13}$, the H\"older inequality, Sobolev embedding $ H^2(\R^3) \hookrightarrow  L^\infty(\R^3)$ and the inequalities shown in Lemma \ref{Lm-Auxil-Ineq} yield that
    \begin{align}\label{Apriori-Est-3D-v-7}
      \no K_{13} \lesssim & |c_4| \| e^{\widehat{\rho}} \|_{L^\infty} \| \nabla_x \widehat{\rho} \|_{L^\infty} \| \nabla_x^k ( \Omega \otimes \Omega ) \|_{L^2} \| \nabla_x^k v \|_{L^2} \\
      \no \lesssim & |c_4| \| e^{\widehat{\rho}} \|_{L^\infty} \| v \|_{H^s} \| \widehat{\rho} \|_{H^3} \Big{(} \| \nabla_x^k \Omega \|_{L^2} + \sum_{\substack{k_1+k_2=k \\ k_1, k_2 \geq 1}} \| \nabla_x^{k_1} \Omega \|_{L^4} \| \nabla_x^{k_2} \Omega \|_{L^4} \Big{]} \\
      \no \lesssim & |c_4| \| e^{\widehat{\rho}} \|_{L^\infty} \| v \|_{H^s} \| \widehat{\rho} \|_{H^s}  \Bigg{[} \sum_{I=1}^k ( \| \phi \|_{H^k} + \| \psi \|_{H^k} )^I + \sum_{\substack{k_1+k_2=k \\ k_1, k_2 \geq 1}} ( \| \nabla_x \phi \|_{H^{k_1}} + \| \nabla_x^{k_1} \psi \|_{H^{k_1}} ) \\
      \no & \qquad \times \sum_{I=1}^{k_1} ( \| \phi \|_{H^{k_1}} + \| \psi \|_{H^{k_1}} )^{I-1} ( \| \nabla_x \phi \|_{H^{k_2}} + \| \nabla_x^{k_2} \psi \|_{H^{k_2}} )  \sum_{I=1}^{k_2} ( \| \phi \|_{H^{k_2}} + \| \psi \|_{H^{k_2}} )^{I-1} \Bigg{]} \\
      \lesssim & |c_4| \| e^{\widehat{\rho}} \|_{L^\infty} \| v \|_{H^s} \| \widehat{\rho} \|_{H^s} ( \| \phi \|_{H^s} + \| \psi \|_{H^s} ) \big{[} 1 + ( \| \phi \|_{H^s} + \| \psi \|_{H^s} )^{s-1} \big{]}
    \end{align}
  for $s\geq 3$. We now estimate the term $K_{14}$ by utilizing the H\"older inequality, Sobolev embedding theory and the inequalities shown in Lemma \ref{Lm-rho-L^infty} as follows:
  \begin{equation}\label{Apriori-Est-3D-v-8}
    \begin{aligned}
      K_{14} \lesssim & |c_4| \| \nabla_x e^{\widehat{\rho}} \|_{L^\infty} \| \nabla_x^k v \|_{L^2} \| \nabla_x^k \widehat{\rho} \|_{L^2} + |c_4| \sum_{\substack{k_1+k_2=k \\ k_1 \geq 2, k_2 \geq 1}} \| \nabla_x^{k_1} e^{\widehat{\rho}} \|_{L^4} \| \nabla_x^{k_2+1} \widehat{\rho} \|_{L^4} \| \nabla_x^k v \|_{L^2} \\
      \lesssim & |c_4| \| e^{\widehat{\rho}} \|_{L^\infty} \| v \|_{H^s} \Bigg{[} \| \widehat{\rho} \|_{H^s} \| \widehat{\rho} \|_{H^3} + \sum_{\substack{k_1+k_2=k \\ k_1 \geq 2, k_2 \geq 1}} \sum_{I=1}^{k_1} \| \widehat{\rho} \|^I_{H^{k_1+1}} \| \nabla_x^{k_2 + 1} \widehat{\rho} \|_{H^1} \Bigg{]} \\
      \lesssim & |c_4| \| e^{\widehat{\rho}} \|_{L^\infty} \| v \|_{H^s} \| \widehat{\rho} \|_{H^s}^2 ( 1 + \| \widehat{\rho} \|^{s-2}_{H^s}  )
    \end{aligned}
  \end{equation}
  holds for $s \geq 3$. For the term $K_{15}$, it is derived from the inequalities in Lemma \ref{Lm-rho-L^infty} that
  \begin{equation*}
    \begin{aligned}
      K_{15} \lesssim & |c_4| \| \nabla_x \widehat{\rho} \|_{L^\infty} \sum_{\substack{k_1+k_3=k \\ k_1, k_3 \geq 1}} \| \nabla_x^{k_1} e^{\widehat{\rho}} \|_{L^4} \| \nabla_x^{k_3} ( \Omega \otimes \Omega ) \|_{L^4} \| \nabla_x^k v \|_{L^2} \\
      \lesssim & |c_4| \| e^{\widehat{\rho}} \|_{L^\infty} \| v \|_{H^s} \| \widehat{\rho} \|_{H^s}^2 ( 1 + \| \widehat{\rho} \|^{s-2}_{H^s}  ) \sum_{1 \leq k_3 \leq k-1} \| \nabla_x^{k_3} ( \Omega \otimes \Omega ) \|_{L^4} \,,
    \end{aligned}
  \end{equation*}
  where the norm $ \sum\limits_{1 \leq k_3 \leq k-1} \| \nabla_x^{k_3} ( \Omega \otimes \Omega ) \|_{L^4} $ can be computed by making use of the inequalities shown in Lemma \ref{Lm-Auxil-Ineq} as follows:
    \begin{align*}
      & \sum_{1 \leq k_3 \leq k-1} \| \nabla_x^{k_3} ( \Omega \otimes \Omega ) \|_{L^4} \lesssim \sum_{1 \leq k_3 \leq k-1} \Bigg{(} \| \nabla_x^{k_3} \Omega \|_{L^4} + \sum_{\substack{a+b=k_3 \\ a,b \geq 1}} \| \nabla_x^a \Omega \|_{L^\infty} \| \nabla_x^b \Omega \|_{L^4} \Bigg{)} \\
      \lesssim & \sum_{1 \leq k_3 \leq k-1} ( \| \nabla_x \phi \|_{H^{k_3}} + \| \nabla_x \psi \|_{H^{k_3}} ) \sum_{I=1}^{k_3} ( \|  \phi \|_{H^{k_3}} + \|  \psi \|_{H^{k_3}} )^{I-1} \\
      & + \sum_{1 \leq k_3 \leq k-1} \sum_{\substack{a+b=k_3 \\ a,b \geq 1}} ( \| \nabla_x \phi \|_{H^{a+1}} + \| \nabla_x \psi \|_{H^{a+1}} ) \sum_{I=1}^a ( \|  \phi \|_{H^{a+1}} + \| \psi \|_{H^{a+1}} )^{I-1} \\
      & \qquad \qquad \qquad \times ( \| \nabla_x \phi \|_{H^{b+1}} + \| \nabla_x \psi \|_{H^{b+1}} ) \sum_{J=1}^b ( \|  \phi \|_{H^{b+1}} + \| \psi \|_{H^{b+1}} )^{J-1} \\
      \lesssim & ( \| \phi \|_{H^s}  + \| \psi \|_{H^s} ) \big{[} 1 + ( \| \phi \|_{H^s}  + \| \psi \|_{H^s} )^{s-2} \big{]} \,.
    \end{align*}
  As a result, the estimation of the term $K_{15}$ is
  \begin{equation}\label{Apriori-Est-3D-v-9}
    K_{15} \lesssim |c_4| \| e^{\widehat{\rho}} \|_{L^\infty} \| v \|_{H^s} \| \widehat{\rho} \|_{H^s}^2 ( \| \phi \|_{H^s}  + \| \psi \|_{H^s} ) ( 1 + \| \widehat{\rho} \|^{s-2}_{H^s}  ) \big{[} 1 + ( \| \phi \|_{H^s}  + \| \psi \|_{H^s} )^{s-2} \big{]} \,.
  \end{equation}
  For the term $K_{16}$, it is easy to be derived from the inequalities \eqref{Bnds-rho-norms} in Lemma \ref{Lm-rho-L^infty} and the inequalities \eqref{Auxil-Ineq-1} in Lemma \ref{Lm-Auxil-Ineq} that
    \begin{align}\label{Apriori-Est-3D-v-10}
      	K_{16}
      \lesssim & |c_4| \| e^{\widehat{\rho}} \|_{L^\infty} \| \nabla_x^k v \|_{L^2} \Big\{ \| \nabla_x^k \widehat{\rho} \|_{L^2} \| \nabla_x ( \Omega \times \Omega ) \|_{L^\infty}  + \sum_{\substack{k_2+k_3=k \\ k_2 \geq 1, k_3 \geq 2}} \| \nabla_x^{k_2+1} \widehat{\rho} \|_{L^4} \| \nabla_x^{k_3} ( \Omega \otimes \Omega ) \|_{L^4} \Big\}
    \no\\
      \lesssim & |c_4| \| e^{\widehat{\rho}} \|_{L^\infty} \| v \|_{H^s} \| \widehat{\rho} \|_{H^s} ( \| \phi \|_{H^s}  + \| \psi \|_{H^s} ) \big{[} 1 + ( \| \phi \|_{H^s}  + \| \psi \|_{H^s} )^{s-2} \big{]}
    \end{align}
  for $s \geq 3$. By following the analogous arguments in the estimation of the term $K_{14}$, $K_{15}$ and $K_{16}$, one can derive that
  \begin{equation}\label{Apriori-Est-3D-v-11}
    K_{17} \lesssim |c_4| \| e^{\widehat{\rho}} \|_{L^\infty} \| v \|_{H^s} \| \widehat{\rho} \|_{H^s}^2 ( \| \phi \|_{H^s}  + \| \psi \|_{H^s} ) ( 1 + \| \widehat{\rho} \|^{s-3}_{H^s}  ) \big{[} 1 + ( \| \phi \|_{H^s}  + \| \psi \|_{H^s} )^{s-3} \big{]} \,.
  \end{equation}
  We plug the inequalities \eqref{Apriori-Est-3D-v-5}, \eqref{Apriori-Est-3D-v-6}, \eqref{Apriori-Est-3D-v-7}, \eqref{Apriori-Est-3D-v-8}, \eqref{Apriori-Est-3D-v-9}, \eqref{Apriori-Est-3D-v-10} and \eqref{Apriori-Est-3D-v-11} into the relation \eqref{Apriori-Est-3D-v-4} and then we gain
  \begin{equation}\label{Apriori-Est-3D-v-12}
    \begin{aligned}
      K_1 \lesssim & |c_4| \| e^{\widehat{\rho}} \|_{L^\infty} \| \nabla_x v \|_{H^s} \| \widehat{\rho} \|_{H^s} + \Big{\{}|c_4| \| e^{\widehat{\rho}} \|_{L^\infty} \| v \|_{H^s} \| \widehat{\rho} \|_{H^s} \\
       & \qquad \times ( \| \phi \|_{H^s}  + \| \psi \|_{H^s} + \| \widehat{\rho} \|_{H^s} ) ( 1 + \| \widehat{\rho} \|^{s-1}_{H^s}  ) \big{[} 1 + ( \| \phi \|_{H^s}  + \| \psi \|_{H^s} )^{s-1} \big{]} \Big{\}} \,.
    \end{aligned}
  \end{equation}

  Now we estimate the term $K_2$. It can be decomposed into seven parts:
    \begin{align}\label{Apriori-Est-3D-v-13}
      \no K_2 = & - c_4  \l \nabla_x^k [ e^{\widehat{\rho}} \nabla_x \phi \cdot \Omega_\phi \otimes \Omega ] , \nabla_x^k v \r \\
      \no = &  - c_4 \l \nabla_x^k e^{\widehat{\rho}} \nabla_x \phi \cdot \Omega_\phi \otimes \Omega , \nabla_x^k v \r - c_4 \l e^{\widehat{\rho}} \nabla_x^{k+1} \phi \cdot \Omega_\phi \otimes \Omega , \nabla_x^k v \r \\
      \no & - c_4 \l  e^{\widehat{\rho}} \nabla_x \phi \cdot \nabla_x^k ( \Omega_\phi \otimes \Omega ) , \nabla_x^k v \r - c_4 \sum_{\substack{k_1+k_2=k \\ k_1, k_2 \geq 1}} \l \nabla_x^{k_1} e^{\widehat{\rho}} \nabla_x^{k_2+1} \phi \cdot \Omega_\phi \otimes \Omega , \nabla_x^k v \r \\
       \no & - c_4 \sum_{\substack{k_1+k_3=k \\ k_1, k_3 \geq 1}} \l \nabla_x^{k_1} e^{\widehat{\rho}} \nabla_x \phi \cdot \nabla_x^{k_3} ( \Omega_\phi \otimes \Omega ) , \nabla_x^k v \r \\
       &- c_4 \sum_{\substack{k_2+k_3=k \\ k_2, k_3 \geq 1}} \l e^{\widehat{\rho}} \nabla_x^{k_2 + 1} \phi \cdot \nabla_x^{k_3} ( \Omega_\phi \otimes \Omega ) , \nabla_x^k v \r \\
      \no & - c_4 \sum_{\substack{k_1+k_2+k_3=k \\ k_1, k_2, k_3 \geq 1}} \l \nabla_x^{k_1} e^{\widehat{\rho}} \nabla_x^{k_2 + 1} \phi \cdot \nabla_x^{k_3} ( \Omega_\phi \otimes \Omega ) , \nabla_x^k v \r \\
      \no \equiv & \, K_{21} + K_{22} + K_{23} + K_{24} + K_{25} + K_{26} + K_{27} \,.
    \end{align}
    We need estimate the terms $K_{2i}\, (1 \leq i \leq 7)$ term by term. For the first two terms $K_{21}$ and $K_{22}$, one can easily derive from the H\"older inequality, Sobolev embedding theory and the inequalities \eqref{Bnds-rho-norms} in Lemma \ref{Lm-rho-L^infty} that
    \begin{equation}\label{Apriori-Est-3D-v-14}
      \begin{aligned}
        K_{21} + K_{22} \lesssim & |c_4| \| \nabla_x^k e^{\widehat{\rho}} \|_{L^2} \| \nabla_x \phi \|_{L^\infty} \| \nabla_x^k v \|_{L^2} + |c_4| \| e^{\widehat{\rho}} \|_{L^\infty} \| \nabla_x^{k+1} \phi \|_{L^2} \| \nabla_x^k v \|_{L^2} \\
        \lesssim & |c_4| \| e^{\widehat{\rho}} \|_{L^\infty} \| v \|_{H^s} \| \nabla_x \phi \|_{H^s} ( 1 + \| \widehat{\rho} \|^s_{H^s} ) \,.
      \end{aligned}
    \end{equation}
    For the term $K_{23}$, it is derived from the inequalities shown in Lemma \ref{Lm-Auxil-Ineq} that
      \begin{align}\label{Apriori-Est-3D-v-15}
        	K_{23}
        \lesssim & |c_4| \| e^{\widehat{\rho}} \|_{L^\infty} \| \nabla_x \phi \|_{L^\infty} \| \nabla_x^k ( \Omega_\phi \otimes \Omega ) \|_{L^2} \| \nabla_x^k v \|_{L^2}
      \no\\
        \lesssim & |c_4| \| e^{\widehat{\rho}} \|_{L^\infty} \| v \|_{H^s} \| \nabla_x \phi \|_{H^s} \Bigg{(} \| \nabla_x^k \Omega_\phi \|_{L^2} + \| \nabla_x^k \Omega \|_{L^2} + \sum_{\substack{a+b=k \\ a,b \geq 1}} \| \nabla_x^a \Omega_\phi \|_{L^4} \| \nabla_x^b \Omega \|_{L^4} \Bigg{)}
      \no\\
        \lesssim & |c_4| \| e^{\widehat{\rho}} \|_{L^\infty} \| v \|_{H^s} \| \nabla_x \phi \|_{H^s} \Bigg{[} \sum_{I=1}^k ( \| \phi \|_{H^k} + \| \psi \|_{H^k} )^I +
        \\\no
        & \qquad \sum_{\substack{a+b=k \\ a,b \geq 1}} \sum_{I=1}^a ( \| \phi \|_{H^{a+1}} + \| \psi \|_{H^{a+1}} )^I \sum_{J=1}^b ( \| \phi \|_{H^{b+1}} + \| \psi \|_{H^{b+1}} )^J \Bigg{]}
      \\\no
        \lesssim & |c_4| \| e^{\widehat{\rho}} \|_{L^\infty} \| v \|_{H^s} \| \nabla_x \phi \|_{H^s} ( \| \phi \|_{H^s}  + \| \psi \|_{H^s} ) \big{[} 1 + ( \| \phi \|_{H^s}  + \| \psi \|_{H^s} )^{s-1} \big{]} \Big{\}} \,.
      \end{align}
  For the term $K_{24}$, we have
  \begin{equation}\label{Apriori-Est-3D-v-16}
    \begin{aligned}
      K_{24} \lesssim & |c_4| \sum_{\substack{k_1+k_2=k \\ k_1, k_2 \geq 1}} \| \nabla_x^{k_1} e^{\widehat{\rho}} \|_{L^4} \| \nabla_x^{k_2+1} \phi \|_{L^4} \| \nabla_x^k v \|_{L^2} \\
      \lesssim & |c_4| \| e^{\widehat{\rho}} \|_{L^\infty} \| v \|_{H^s} \| \nabla_x \phi \|_{H^s} \| \widehat{\rho} \|_{H^s} ( 1 + \| \widehat{\rho} \|^{s-2}_{H^s} ) \,.
    \end{aligned}
  \end{equation}
  Here we make use of the inequalities \eqref{Bnds-rho-norms} in Lemma \ref{Lm-rho-L^infty}. For the terms $K_{25}$ and $K_{26}$, by the similar arguments in the inequalities \eqref{Apriori-Est-3D-v-9} and using the inequalities \eqref{Bnds-rho-norms} in Lemma \ref{Lm-rho-L^infty}, we can estimate that
  \begin{equation}\label{Apriori-Est-3D-v-17}
    K_{25} \lesssim |c_4| \| e^{\widehat{\rho}} \|_{L^\infty} \| v \|_{H^s} \| \nabla_x \phi \|_{H^s} \| \widehat{\rho} \|_{H^s} ( \| \phi \|_{H^s} + \| \psi \|_{H^s} ) ( 1 + \| \widehat{\rho} \|^{s-2}_{H^s} ) \big{[} 1 + ( \| \phi \|_{H^s}  + \| \psi \|_{H^s} )^{s-2} \big{]}\,,
  \end{equation}
  and
  \begin{equation}\label{Apriori-Est-3D-v-18}
    \begin{aligned}
      K_{26} \lesssim & |c_4| \| e^{\widehat{\rho}} \|_{L^\infty} \| \nabla_x^k v \|_{L^2} \sum_{\substack{k_2+k_3=k \\ k_2, k_3 \geq 1}} \| \nabla_x^{k_2+1} \phi \|_{L^4} \| \nabla_x^{k_2} ( \Omega_\phi \otimes \Omega ) \|_{L^4} \\
      \lesssim & |c_4| \| e^{\widehat{\rho}} \|_{L^\infty} \| v \|_{H^s} \| \nabla_x \phi \|_{H^s} ( \| \phi \|_{H^s} + \| \psi \|_{H^s} ) \big{[} 1 + ( \| \phi \|_{H^s}  + \| \psi \|_{H^s} )^{s-2} \big{]}\,.
    \end{aligned}
  \end{equation}
  Following the analogous calculation of the inequality \eqref{Apriori-Est-3D-v-11} and making use of the inequalities \eqref{Bnds-rho-norms} in Lemma \ref{Lm-rho-L^infty} yield that
  \begin{equation}\label{Apriori-Est-3D-v-19}
    K_{27} \lesssim |c_4| \| e^{\widehat{\rho}} \|_{L^\infty} \| v \|_{H^s} \| \nabla_x \phi \|_{H^s} \| \widehat{\rho} \|_{H^s} ( \| \phi \|_{H^s} + \| \psi \|_{H^s} ) ( 1 + \| \widehat{\rho} \|^{s-3}_{H^s} ) \big{[} 1 + ( \| \phi \|_{H^s}  + \| \psi \|_{H^s} )^{s-3} \big{]}\,.
  \end{equation}
  We plug the inequalities \eqref{Apriori-Est-3D-v-14}, \eqref{Apriori-Est-3D-v-15}, \eqref{Apriori-Est-3D-v-16}, \eqref{Apriori-Est-3D-v-17}, \eqref{Apriori-Est-3D-v-18}  and \eqref{Apriori-Est-3D-v-19} into the relation \eqref{Apriori-Est-3D-v-13}, and then we have
  \begin{equation}\label{Apriori-Est-3D-v-20}
    K_2 \lesssim |c_4| \| e^{\widehat{\rho}} \|_{L^\infty} \| v \|_{H^s} \| \nabla_x \phi \|_{H^s} ( 1 + \| \widehat{\rho} \|^{s }_{H^s} ) \big{[} 1 + ( \| \phi \|_{H^s}  + \| \psi \|_{H^s} )^{s } \big{]}\,.
  \end{equation}
  Since the terms $K_3$, $K_4$ and $K_5$ be of the same form of the term $K_2$, by the similar estimation of $K_2$, we can estimate that
  \begin{equation}\label{Apriori-Est-3D-v-21}
    K_4 \lesssim |c_4| \| e^{\widehat{\rho}} \|_{L^\infty} \| v \|_{H^s} \| \nabla_x \phi \|_{H^s} ( 1 + \| \widehat{\rho} \|^{s }_{H^s} ) \big{[} 1 + ( \| \phi \|_{H^s}  + \| \psi \|_{H^s} )^{s } \big{]}\,,
  \end{equation}
  and
  \begin{equation}\label{Apriori-Est-3D-v-22}
    K_3 + K_5 \lesssim |c_4| \| e^{\widehat{\rho}} \|_{L^\infty} \| v \|_{H^s} \| \nabla_x \psi \|_{H^s} ( 1 + \| \widehat{\rho} \|^{s }_{H^s} ) \big{[} 1 + ( \| \phi \|_{H^s}  + \| \psi \|_{H^s} )^{s } \big{]}\,.
  \end{equation}

  By substituting the inequalities \eqref{Apriori-Est-3D-v-12}, \eqref{Apriori-Est-3D-v-20}, \eqref{Apriori-Est-3D-v-21} and \eqref{Apriori-Est-3D-v-22} into the equality \eqref{Apriori-Est-3D-v-3}, one immediately obtain
  \begin{equation}\label{Apriori-Est-3D-v-23}
    \begin{aligned}
      & - \l \nabla_x^k [ e^{\widehat{\rho}} \mathcal{G} ( \widehat{\rho}, \phi, \psi ) ] , \nabla_x^k v \r \\
      \lesssim & |c_4| \| e^{\widehat{\rho}} \|_{L^\infty} \| v \|_{H^s} \| \widehat{\rho} \|_{H^s} ( \| \phi \|_{H^s}  + \| \psi \|_{H^s} ) ( 1 + \| \widehat{\rho} \|^{s - 1 }_{H^s} ) \big{[} 1 + ( \| \phi \|_{H^s}  + \| \psi \|_{H^s} )^{s-1} \big{]} \\
      + & |c_4| \| e^{\widehat{\rho}} \|_{L^\infty} ( | v \|_{H^s} + \| \widehat{\rho} \|_{H^s} ) ( \| \nabla_x v \|_{H^s} + \| \nabla_x \phi \|_{H^s} + \| \nabla_x \psi \|_{H^s} ) \\
      & \qquad \qquad \qquad \times ( 1 + \| \widehat{\rho} \|^{s-1}_{H^s} ) \big{[} 1 + ( \| \phi \|_{H^s}  + \| \psi \|_{H^s} )^{s-1} \big{]} \,.
    \end{aligned}
  \end{equation}

Consequently, by combining the inequalities \eqref{Apriori-Est-3D-v-1}, \eqref{Apriori-Est-3D-v-2} and \eqref{Apriori-Est-3D-v-23} together and summing up for all integer $0 \leq k \leq s$, one can derive the energy estimate of the $v$-equation of the SOH-NS system \eqref{SOH-NS-3D-SPT-Smp}
  \begin{equation}\label{Apriori-Est-3D-v}
    \begin{aligned}
      & \tfrac{1}{2} \tfrac{\d}{\d t} ( Re \| v \|^2_{H^s} ) + \| \nabla_x v \|^2_{H^s} \\
      \lesssim & Re \| \nabla_x v \|_{H^s} \| v \|^2_{H^s} + \Big{[} |c_4| \| e^{\widehat{\rho}} \|_{L^\infty} ( | v \|_{H^s} + \| \widehat{\rho} \|_{H^s} ) ( 1 + \| \widehat{\rho} \|^{s-1}_{H^s} )  \\
      & \qquad \qquad \times ( \| \nabla_x v \|_{H^s} + \| \nabla_x \phi \|_{H^s} + \| \nabla_x \psi \|_{H^s} )  \big{[} 1 + ( \| \phi \|_{H^s}  + \| \psi \|_{H^s} )^{s-1} \big{]} \Big{]} \\
      + & |c_4| \| e^{\widehat{\rho}} \|_{L^\infty} \| v \|_{H^s} \| \widehat{\rho} \|_{H^s} ( \| \phi \|_{H^s}  + \| \psi \|_{H^s} ) ( 1 + \| \widehat{\rho} \|^{s - 1 }_{H^s} ) \big{[} 1 + ( \| \phi \|_{H^s}  + \| \psi \|_{H^s} )^{s-1} \big{]} \,.
    \end{aligned}
  \end{equation}

{\smallskip\noindent \em\large Step 5. Closing the energy estimate.} We now close the energy estimate of the SOH-NS system \eqref{SOH-NS-3D-SPT-Smp}. It is implied by summing up for the inequalities \eqref{Apriori-Est-3D-rho}, \eqref{Apriori-Est-3D-phi}, \eqref{Apriori-Est-3D-psi} and \eqref{Apriori-Est-3D-v}  that
  \begin{equation}\label{Apriori-Est-3D-1}
    \begin{aligned}
      & \tfrac{1}{2} \tfrac{\d}{\d t} ( \| \widehat{\rho} \|^2_{H^s} + \| \phi \|^2_{H^s} + \| \psi \|^2_{H^s} + Re \| v \|^2_{H^s} ) + \gamma \| \nabla_x \phi \|^2_{H^s} + \gamma \| \nabla_x \psi \|^2_{H^s} + \| \nabla_x v \|^2_{H^s} \\
      \lesssim & \| \nabla_x v \|_{H^s} ( \| \widehat{\rho} \|^2_{H^s} + Re \| v \|^2_{H^s} ) + ( |a c_2| + \gamma + |\tfrac{a}{\kappa}| ) ( \| \nabla_x \phi \|_{H^s} + \| \nabla_x \psi \|_{H^s} ) \\
      & \qquad \qquad \times ( \| \phi \|_{H^s}  + \| \psi \|_{H^s} + \| \widehat{\rho} \|_{H^s} ) \big{[} 1 + ( \| \phi \|_{H^s}  + \| \psi \|_{H^s} )^{s+4} \big{]} \\
      & + |a c_1| ( \| \nabla_x \phi \|_{H^s} + \| \nabla_x \psi \|_{H^s} ) \| \widehat{\rho} \|_{H^s} ( 1 + \| \widehat{\rho} \|_{H^s} ) \big{[} 1 + ( \| \phi \|_{H^s}  + \| \psi \|_{H^s} )^{s} \big{]} \\
      & + ( |\widetilde{\lambda}| + 1 ) \| \nabla_x v \|_{H^s} ( \| \phi \|_{H^s}  + \| \psi \|_{H^s} ) \big{[} 1 + ( \| \phi \|_{H^s}  + \| \psi \|_{H^s} )^{s+3} \big{]} \\
      & + |c_4| \| e^{\widehat{\rho}} \|_{L^\infty} \| v \|_{H^s} \| \widehat{\rho} \|_{H^s} ( \| \phi \|_{H^s}  + \| \psi \|_{H^s} ) ( 1 + \| \widehat{\rho} \|^{s - 1 }_{H^s} ) \big{[} 1 + ( \| \phi \|_{H^s}  + \| \psi \|_{H^s} )^{s-1} \big{]} \\
      + & |c_4| \| e^{\widehat{\rho}} \|_{L^\infty} ( | v \|_{H^s} + \| \widehat{\rho} \|_{H^s} ) ( \| \nabla_x v \|_{H^s} + \| \nabla_x \phi \|_{H^s} + \| \nabla_x \psi \|_{H^s} ) \\
      & \qquad \qquad \qquad \times ( 1 + \| \widehat{\rho} \|^{s-1}_{H^s} ) \big{[} 1 + ( \| \phi \|_{H^s}  + \| \psi \|_{H^s} )^{s-1} \big{]} \,.
    \end{aligned}
  \end{equation}
  Recalling the definition of the energy functionals $\mathcal{E}(t)$ and $\mathcal{D} (t)$, the energy estimate \eqref{Apriori-Est-3D-1} reduce to
  \begin{equation}\label{Apriori-Est-3D-2}
    \tfrac{\d}{\d t} \mathcal{E} (t) + \mathcal{D} (t) \leq C ( 1 + \| e^{\widehat{\rho}} \|^2_{L^\infty} ) \mathcal{E}(t) [ 1 + \mathcal{E}^{3s}(t) ] \,,
  \end{equation}
  where we utilize the Young inequality, and the positive constant
  $$ C = C'(s) \big{[} ( 1 + |\widetilde{\lambda}| + \tfrac{1}{\sqrt{\gamma}} ( |a c_1| + |a c_2| + \gamma + |\tfrac{a}{\kappa}| ) + |c_4| ( 1 + \tfrac{1}{\sqrt{\gamma}} ) ( 1 + \tfrac{1}{\sqrt{Re}} ) )^2 + \tfrac{Re}{\sqrt{\gamma}} \big{]} > 0 $$
  for some constant $C'(s) > 0$.

  It remains to control the $L^\infty$-norm $ \| e^{\widehat{\rho}} \|^2_{L^\infty} $. As shown in Lemma \ref{Lm-rho-L^infty}, we let $f = - a c_1 ( \Omega_\phi \cdot \nabla_x \phi + \Omega_\psi \cdot \nabla_x \psi ) $. As a result, we have
  \begin{equation}\label{Apriori-Est-3D-3}
    \begin{aligned}
      \| e^{\widehat{\rho}} \|^2_{L^\infty} \leq &  \bar{\rho}^2 \exp \Big{(} 2 |a c_1| \int_0^t ( \| \nabla_x \phi \|_{L^\infty} + \| \nabla_x \psi \|_{L^\infty} ) \d \tau \Big{)} \\
      \leq &   \bar{\rho}^2 \exp \Big{(} 2 |a c_1| \int_0^t ( \| \phi \|_{H^s} + \|  \psi \|_{H^s} ) \d \tau \Big{)}
    \end{aligned}
  \end{equation}
  for $s \geq 3$. Consequently, the inequalities \eqref{Apriori-Est-3D-2} and \eqref{Apriori-Est-3D-3} imply the inequality \eqref{Aprori-Est-SOHNS-3D-SPT-Simp}. Then the proof of Proposition \ref{Prop-Apr-Est-3D} is finished.
\end{proof}

\section{Local Well-posedness of SOH-NS System \eqref{SOH-NS}}\label{sec:Local-WP}

In this section, based on the \emph{a priori} estimates in Section \ref{sec:Apriori-Est-3D}, we mainly justify the local existence of the SOH-NS system \eqref{SOH-NS}, i.e. to prove Theorem \ref{Thm-WP-SOHNS-3D}.

\begin{proof}[Proof of Theorem \ref{Thm-WP-SOHNS-3D}]

To complete the justification of Theorem \ref{Thm-WP-SOHNS-3D}, we only need to take the system \eqref{SOH-NS-3D-SPT-Smp} into consideration. The approximate system of the equations \eqref{SOH-NS-3D-SPT-Smp} with the initial conditions
  $$
  	( \widehat{\rho}, \phi, \psi, v )|_{t=0} = ( \ln \rho^{in}, \phi^{in}, \psi^{in}, v^{in} )
  $$
can be constructed by the following forms:
  \begin{equation}\label{Appr-Syst-SOHNS}
    \left\{
      \begin{array}{l}
        \partial_t \widehat{\rho}^\ep + \J_\ep [ ( a c_1 \Omega(\J_\ep \phi^\ep, \J_\ep \psi^\ep) + \J_\ep v^\ep ) \cdot \nabla_x \J_\ep \widehat{\rho}^\ep ]
        	\\[3pt]
        	\qquad \qquad + a c_1 \J_\ep [ \Omega_\phi (\J_\ep \phi^\ep, \J_\ep \psi^\ep) \cdot \nabla_x \J_\ep \phi^\ep + \Omega_\psi (\J_\ep \phi^\ep, \J_\ep \psi^\ep) \cdot \nabla_x \J_\ep \psi^\ep  ] = 0\,,
      \\[3pt]
        \partial_t \phi^\ep + \J_\ep [ ( a c_2 \Omega(\J_\ep \phi^\ep, \J_\ep \psi^\ep) + \J_\ep v^\ep ) \cdot \nabla_x \J_\ep \phi^\ep ] + \J_\ep \mathcal{H}_\phi (\J_\ep \widehat{\rho}^\ep, \J_\ep \phi^\ep, \J_\ep \psi^\ep) = 0 \,,
      \\[3pt]
        \partial_t \psi^\ep + \J_\ep [ ( a c_2 \Omega(\J_\ep \phi^\ep, \J_\ep \psi^\ep) + \J_\ep v^\ep ) \cdot \nabla_x \J_\ep \psi^\ep ] + \J_\ep \mathcal{H}_\psi (\J_\ep \widehat{\rho}^\ep, \J_\ep \phi^\ep, \J_\ep \psi^\ep) = 0 \,,
      \\[3pt]
        Re \partial_t v^\ep + Re \J_\ep ( \J_\ep v^\ep \cdot \nabla_x \J_\ep v^\ep ) - \Delta_x \J_\ep v^\ep + \nabla_x p^\ep
        \\[3pt]
        	\hspace*{7cm} + b \J_\ep [ e^{\J_\ep \widehat{\rho}^\ep} \mathcal{G} ( \J_\ep \widehat{\rho}^\ep , \J_\ep \phi^\ep, \J_\ep \psi^\ep ) ] = 0 \,,
      \\[3pt]
        \qquad \nabla_x \cdot v^\ep = 0 \,,
      \\[3pt]
        ( \widehat{\rho}^\ep, \phi^\ep, \psi^\ep, v^\ep )|_{t=0} = ( \J_\ep \ln \rho^{in}, \J_\ep \phi^{in}, \J_\ep \psi^{in}, \J_\ep v^{in} )\,.
      \end{array}
    \right.
  \end{equation}
where the mollifier $\J_\ep$ is defined as $ \J_\ep f = \mathcal{F}^{-1} \big{(} \mathbf{1}_{|\xi| \leq \frac{1}{\ep}} \mathcal{F} (f) (\xi) \big{)}$, in which the symbol $\mathcal{F}$ is the standard Fourier transform and $\mathcal{F}^{-1}$ represents its inverse transform, and the quantities $\mathcal{H}_\phi(\cdot, \cdot, \cdot, \cdot)$, $\mathcal{H}_\psi(\cdot, \cdot, \cdot, \cdot)$ and $\mathcal{G}(\cdot, \cdot, \cdot)$ are defined in \eqref{Terms-Def}.

By ODE theory, we know that there is a maximal $T_\ep > 0$ such that the approximate system \eqref{Appr-Syst-SOHNS} has a unique solution $(\widehat{\rho}^\ep , \phi^\ep, \psi^\ep, v^\ep) \in C([0,T_\ep); H^s(\R^2))$. Since the mollifier $\J_\ep$ satisfies $\J_\ep^2 = \J_\ep$,  we observe that $(\J_\ep \widehat{\rho}^\ep , \J_\ep \phi^\ep , \J_\ep \psi^\ep , \J_\ep v^\ep)$ is also a solution to the system \eqref{Appr-Syst-SOHNS}. Then the uniqueness implies that $(\J_\ep \widehat{\rho}^\ep , \J_\ep \phi^\ep , \J_\ep \psi^\ep , \J_\ep v^\ep) = (\widehat{\rho}^\ep , \phi^\ep , \psi^\ep , v^\ep)$. As a consequence, the solution $(\widehat{\rho}^\ep , \phi^\ep , \psi^\ep, v^\ep)$ to the approximate system \eqref{Appr-Syst-SOHNS} also solves the system
  \begin{equation}\label{Appr-Syst-SOHNS-Smp}
    \left\{
      \begin{array}{l}
        \partial_t \widehat{\rho}^\ep + \J_\ep [ ( a c_1 \Omega( \phi^\ep, \psi^\ep) + v^\ep ) \cdot \nabla_x \widehat{\rho}^\ep ]
        \\[3pt]
	        \qquad + a c_1 \J_\ep [ \Omega_\phi ( \phi^\ep, \psi^\ep) \cdot \nabla_x \phi^\ep + \Omega_\psi ( \phi^\ep, \psi^\ep) \cdot \nabla_x \psi^\ep  ] = 0\,,
	    \\[3pt]
        \partial_t \phi^\ep + \J_\ep [ ( a c_2 \Omega( \phi^\ep, \psi^\ep) + v^\ep ) \cdot \nabla_x \phi^\ep ] + \J_\ep \mathcal{H}_\phi ( \widehat{\rho}^\ep, \phi^\ep, \psi^\ep) = 0 \,,
      \\[3pt]
        \partial_t \psi^\ep + \J_\ep [ ( a c_2 \Omega( \phi^\ep, \psi^\ep) + v^\ep ) \cdot \nabla_x \psi^\ep ] + \J_\ep \mathcal{H}_\psi ( \widehat{\rho}^\ep, \phi^\ep, \psi^\ep) = 0 \,,
      \\[3pt]
        Re \partial_t v^\ep + Re \J_\ep ( v^\ep \cdot \nabla_x v^\ep ) - \Delta_x v^\ep + \nabla_x p^\ep + b \J_\ep [ e^{ \widehat{\rho}^\ep} \mathcal{G} ( \widehat{\rho}^\ep , \phi^\ep, \psi^\ep ) ] = 0 \,,
      \\[3pt]
        \qquad \nabla_x \cdot v^\ep = 0 \,,
      \\[3pt]
        ( \widehat{\rho}^\ep, \phi^\ep, \psi^\ep, v^\ep )|_{t=0} = ( \J_\ep \ln \rho^{in}, \J_\ep \phi^{in}, \J_\ep \psi^{in}, \J_\ep v^{in} )\,.
      \end{array}
    \right.
  \end{equation}

  As shown in the proceeding of the derivation of \emph{a priori} estimate in Proposition \ref{Prop-Apr-Est-3D} in Section \ref{sec:Apriori-Est-3D}, we can derive the energy estimate of the approximate system \eqref{Appr-Syst-SOHNS-Smp}
  \begin{equation}\label{Unf-Bnd-1}
    \tfrac{\d}{\d t} \mathcal{E}_\ep (t) + \mathcal{D}_\ep (t) \leq C \Big{[} 1 + \exp \big{(} C \int_0^t \mathcal{E}_\ep^\frac{1}{2} (\tau) \d \tau \big{)} \Big{]} \mathcal{E}_\ep (t) [ 1 + \mathcal{E}_\ep^{3s} (t) ]
  \end{equation}
  for all $t \in [ 0 , T_\ep )$, where the energy functionals $\mathcal{E}_\ep(t)$ and $\mathcal{D}_\ep(t)$ are
  \begin{equation*}
   \begin{aligned}
      \mathcal{E}_\ep(t) = &  \| \widehat{\rho}^\ep \|^2_{H^s} + \| \phi^\ep \|^2_{H^s} + \| \psi^\ep \|^2_{H^s} + Re \| v^\ep \|^2_{H^s}\,, \\
       \mathcal{D}_\ep (t) = & \gamma \| \nabla_x \phi^\ep \|^2_{H^s} + \gamma \| \nabla_x \psi^\ep \|^2_{H^s} + \| \nabla_x v^\ep \|^2_{H^s} \,,
   \end{aligned}
  \end{equation*}
  and $\mathcal{E}_\ep(0) \leq \mathcal{E}^{in} \equiv \| \ln \rho^{in} \|^2_{H^s} + \| \phi^{in} \|^2_{H^s} + \| \psi^{in} \|^2_{H^s} + Re \| v^{in} \|^2_{H^s} < \infty$.

  For any fixed $M > 2 \mathcal{E}^{in}$, we define $ T^\ep_M = \Big{ \{ } \tau \in [ 0, T_\ep ) ; \sup_{t \in [ 0, \tau ]} \mathcal{E}_\ep (t) \leq M \Big{ \} } $. By the continuity of $\mathcal{E}_\ep(t)$ in $[0,T_\ep )$, we immediately know that $T^\ep_M > 0$. Then the inequality \eqref{Unf-Bnd-1} implies that for all $t \in [ 0, T^\ep_M ]$,
  \begin{equation}\label{Unf-Bnd-2}
    \tfrac{\d}{\d t} \mathcal{E}_\ep(t) + \mathcal{D}_\ep(t) \leq C [ 1 + \exp ( C M^\frac{1}{2} t ) ] ( 1+M^{3s} ) \mathcal{E}_\ep(t) \,,
  \end{equation}
  which can be solved (thanks to the ODE theory) that
  \begin{equation*}
    \begin{aligned}
      \mathcal{E}_\ep (t) \leq & \mathcal{E}_\ep(0) \exp \Big{[} C ( 1 + M^{3s} ) t + \tfrac{1 + M^{3s}}{M^\frac{1}{2}} \big{(} \exp(C M^\frac{1}{2} t) - 1 \big{)} \Big{]} \\
      \leq & \mathcal{E}^{in} \exp \Big{[} C ( 1 + M^{3s} ) t + \tfrac{1 + M^{3s}}{M^\frac{1}{2}} \big{(} \exp(C M^\frac{1}{2} t) - 1 \big{)} \Big{]} \,.
    \end{aligned}
  \end{equation*}

We claim that there exists some $T = T(M, \mathcal{E}^{in}, C) > 0$ such that $\mathcal{E}_\ep (t) \leq \tfrac{M}{2}$ holds uniformly for all $t \in [0,T]$. Indeed, it requires to  have
  $$
  	\mathcal{E}^{in} \exp \Big{[} C ( 1 + M^{3s} ) t + \tfrac{1 + M^{3s}}{M^\frac{1}{2}} \big{(} \exp(C M^\frac{1}{2} t) - 1 \big{)} \Big{]} \leq \tfrac{M}{2} \,,
  $$
namely,
  $$
  	C ( 1 + M^{3s} ) t + \tfrac{1 + M^{3s}}{M^\frac{1}{2}} \big{(} \exp(C M^\frac{1}{2} t) - 1 \big{)} \leq \ln \tfrac{M}{2 \mathcal{E}^{in}} \,.
  $$
Then we assume that
  \begin{equation*}
    \left\{
      \begin{array}{l}
        C ( 1 + M^{3s} ) t \leq \tfrac{1}{2} \ln \tfrac{M}{2 \mathcal{E}^{in}} \,,
      \\[5pt]
        \tfrac{1 + M^{3s}}{M^\frac{1}{2}} \big{(} \exp(C M^\frac{1}{2} t) - 1 \big{)} \leq \tfrac{1}{2} \ln \tfrac{M}{2 \mathcal{E}^{in}} \,,
      \end{array}
    \right.
  \end{equation*}
from which we solve that $t \leq T  \equiv \min \bigg{\{} \tfrac{\ln \tfrac{M}{2 \mathcal{E}^{in}}}{ 2 C ( 1 + M^{3s} ) } , \tfrac{1}{ C M^\frac{1}{2} } \ln \Big{[} 1 + \tfrac{M^\frac{1}{2}}{ 2 ( 1 + M^{3s} ) \ln \tfrac{M}{2 \mathcal{E}^{in}} } \Big{]} \bigg{\}} $ and this number $T > 0$ is exact what we want to find. So the claim holds.

By the definition of $T^\ep_M$, it follows that $T^\ep_M \geq T > 0$. Therefore, for all $t \in [0,T]$, $\mathcal{E}_\ep(t) \leq \tfrac{1}{2} M < M$ and the inequality \eqref{Unf-Bnd-2} reduces to
  \begin{equation*}
    \mathcal{E}_\ep(t) + \int_0^t \mathcal{D}_\ep(s) \d s \leq C(M,T) \equiv \tfrac{C}{16} \Big{[} 1 + \exp ( \tfrac{\sqrt{2}}{2} C M^\frac{1}{2} T ) \Big{]} M ( 8 + M^{3s} )
  \end{equation*}
  for all $t \in [0,T]$, hence we have
  \begin{equation}\label{Unf-Bnd}
    \begin{aligned}
      \sup_{0 \leq t \leq T} \Big{(} \| \widehat{\rho}^\ep \|^2_{H^s} & + \| \phi^\ep \|^2_{H^s} + \| \psi^\ep \|^2_{H^s} + Re \| v^\ep \|^2_{H^s} \Big{)} \\
      & + \gamma \| \nabla_x \phi^\ep \|^2_{L^2(0,T;H^s)} + \gamma \| \nabla_x \psi^\ep \|^2_{L^2(0,T;H^s)} + \| \nabla_x v^\ep \|^2_{H^s} \leq C(M,T) \,.
    \end{aligned}
  \end{equation}
  Thus, based on the uniform bound \eqref{Unf-Bnd}, we finish the proof of Theorem \ref{Thm-WP-SOHNS-3D} by compactness arguments.
  \end{proof}

  We remark that the local well-posedness of the SOH-NS system in $\R^2$ can also easily be proved by the analogous arguments of the three dimension case. More specifically, one just make use of the polar coordinates transform to deal with the geometric constraint $|\Omega| = 1$, and then employing the energy method in proving the local solution to the 3D system \eqref{SOH-NS} can gain the goal.







\section{A Priori Estimates Uniformly in $\varepsilon$} 
\label{sec:2}

Theorem \ref{thm:limit} is based on the following key lemma, which represents \emph{a priori} estimates for the remainder system \eqref{eq:remainder} uniformly in $\eps$.
\begin{lemma}[Uniform-in-$\eps$ Estimate] \label{lemm:apriori-uniform}
	Define the energy functionals as follows,
		\begin{align}
		  E = & \nm{v_R^\eps}_{H^s_x}^2 + \sum_{0\le k+l =m \le s} \nm{\tfrac{\nabla_x^k \nabla_{\omega}^l f_R^\eps}{M_0}}_\M^2, \\\no
		  D = & \nm{\nabla_x v_R^\eps}_{H^s_x}^2 + \frac{1}{\eps} \sum_{0\le k+l =m \le s} \nm{\nabla_{\omega} (\tfrac{\nabla_x^k \nabla_{\omega}^l f_R^\eps)}{M_0}}_\M^2.
		\end{align}
	Then there exist some constants $\eps_0>0$ and $c_0 > 0$ such that, for any $\eps \in (0,\ \eps_0)$ and $t \in [0,\ T]$, we have
		\begin{align}\label{esm:apriori-uniform}
			\tfrac{1}{2} \tfrac{\d}{\d t} E + c_0 D \le C_0 (E + 1) + C_0 (\eps + \eps E) D,
		\end{align}
	where $C_0$ depends only on the initial data $\|(\rho_0^{in}, \Omega_0^{in}, v_0^{in})\|_{H^m_x}$\ $(\text{with } m>s+4)$.
\end{lemma}

We give here the proof of Theorem \ref{thm:limit} by virtue of Lemma \ref{lemm:apriori-uniform}. Given initial data $(f_R^{\eps,in},\ v_R^{\eps,in})$, it follows from the local existence result of the remainder equations \eqref{eq:remainder} (see Lemma \ref{lemm:local-Re} below) that, there exists a unique solution $(f_R^\eps(t,x,\omega), v_R^\eps(t,x))$ on $[0,\ T_\eps]$ with $T_\eps > 0$ be the maximal lifespan of this local solution.

Denote $E_T = 2 e^{2 C_0 T} (E(0) + 2 C_0 T)$. By Lemma \ref{lemm:apriori-uniform}, it holds for any $t \in [0,\ T_\eps]$, that
	\begin{align}\label{inequ-1}
		\tfrac{1}{2} \tfrac{\d}{\d t} E + c_0 D \le C_0 (E + 1) + C_0 (\eps + \eps E) D.
	\end{align}

Let $T_1 = \sup_{t \in [0,\ T_\eps]} \left\{ E(t) \le E_T \right\}$, then choose
	\begin{align}
	  \eps_0 = \min \left\{ \frac{1}{6 C_0},\ \frac{1}{6 C_0 E_T} \right\}.
	\end{align}
If $T_\eps < T$, the Gr\"onwall inequality enables us to get from \eqref{inequ-1} that, for any $\eps \in (0,\ \eps_0)$ and $t \in [0,\ T_1]$,
	\begin{align}
		E(t) \le e^{2 C_0 t} (E(0) + 2 C_0 t) \le \tfrac{1}{2} E_T.
	\end{align}
This implies that the solution can be continued beyond the time $T_\eps$, which is a contradiction with the maximal property of $T_\eps$. So, it follows that $T_\eps = T$ and hence the proof of Theorem \ref{thm:limit} are completed.  \qed

We will use the following weighted Poincar\'e inequality.
\begin{lemma}[Weighted Poincar\'e Inequality] \label{lemm:Poincare-inequ}
	We have the following weighted Poincar\'e inequality, for $f \in \hx{k} L^2_\omega$, $g \in \hx{k} H^l_{\omega} $,
		\begin{align}\label{eq:Poincare-inequ-x}
			\nm{\tfrac{\nabla_x^k f}{\M}}_\M \le\ & \Lambda \nm{\nabla_\omega (\tfrac{\nabla_x^k f}{\M})}_\M, \\[3pt] \label{eq:Poincare-inequ-xw}
			\nm{\tfrac{\nabla_x^k \nabla_\omega^l g}{\M}}_\M \le\ & \widetilde\Lambda \sum_{l' \le l-1} \nm{\nabla_\omega (\tfrac{\nabla_x^k \nabla_\omega^{l'} g}{\M})}_\M,
		\end{align}
		where $\Lambda,\ \widetilde\Lambda$ are the Poincar\'e constants independent of $\Omega_0$.
\end{lemma}

\subsection{Higher-Order Derivatives Estimates on Velocity Field} 
\label{sub:2-1}


For simplicity, we set the coefficients $Re$ and $b$ to $1$ in the following texts. Apply the derivatives operator $\nabla_x^s$ to equation \eqref{eq:remainder}$_2$, then we can write that
	\begin{multline}\label{eq:derivative-v}
		\nabla_x^s (\p_t v_R^\eps + v_0 \cdot \nabla_x v_R^\eps + v_R^\eps \cdot \nabla_x v_0 + \sqrt\eps v_R^\eps \cdot \nabla_x v_R^\eps) \\
	      + \nabla_x \cdot (\nabla_x^s Q_{f_R^\eps}) = - \nabla_x^{s+1} p_R^\eps + \Delta_x \nabla_x^s v_R^\eps.
	\end{multline}
Taking $L^2_x$ inner product with $\nabla_x^s v_R^\eps$, it follows directly that
	\begin{align}
		\skp{\p_t \nabla_x^s v_R^\eps}{\nabla_x^s v_R^\eps} = \frac{1}{2} \frac{\d}{\d t} \nm{\nabla_x^s v_R^\eps}_{L^2_x}^2.
	\end{align}

Noticing the fact $\nabla_x \cdot v_R^\eps=0$, we can infer from performing integrations by part that
	\begin{align}
			& \skp{\nabla_x^s (v_R^\eps \cdot \nabla_x v_R^\eps)}{\nabla_x^s v_R^\eps} \\\no
		= & \skp{v_R^\eps \cdot \nabla_x^{s+1} v_R^\eps}{\nabla_x^s v_R^\eps}
			+ \skp{\nabla_x^s v_R^\eps \cdot \nabla_x v_R^\eps}{\nabla_x^s v_R^\eps}
			+ \sum_{\scrpt{s=s_1 + s_2}{1 \le s_1 \le s-1}} \skp{\nabla_x^{s_1} v_R^\eps \cdot \nabla_x^{s_2 +1} v_R^\eps}{\nabla_x^s v_R^\eps}
		\\\no
		\ls & 0 + \nm{\nabla_x v_R^\eps}_{L^\infty_x} \nm{\nabla_x^s v_R^\eps}_{L^2_x}^2
			+ \sum_{\scrpt{s=s_1 + s_2}{1 \le s_1,s_2 \le s-1}} \nm{\nabla_x^{s_1} v_R^\eps}_{L^\infty_x} \nm{\nabla_x^{s_2+1} v_R^\eps}_{L^2_x} \nm{\nabla_x^s v_R^\eps}_{L^2_x}
		\\\no
		\ls & \nm{\nabla_x v_R^\eps}_{H^s_x} \nm{\nabla_x^s v_R^\eps}_{L^2_x}^2,
	\end{align}
where we have used the H{\"o}lder inequality, the Sobolev embedding inequalities, and the assumption $s \ge 2$.

Since $v_0$ is also divergence free, by similar and easier arguments, we get
	\begin{align}
		  \skp{\nabla_x^s (v_0 \cdot \nabla_x v_R^\eps)}{\nabla_x^s v_R^\eps}
		= & \skp{v_0 \cdot \nabla_x^{s+1} v_R^\eps}{\nabla_x^s v_R^\eps}
			+ \sum_{\scrpt{s=s_1 + s_2}{1 \le s_1 \le s}} \skp{\nabla_x^{s_1} v_0 \cdot \nabla_x^{s_2 +1} v_R^\eps}{\nabla_x^s v_R^\eps}
		\\\no
		\ls & \nm{v_0}_{H^{s+2}_x} \nm{\nabla_x^s v_R^\eps}_{L^2_x}^2,
	\end{align}
and
	\begin{align}
			\skp{\nabla_x^s (v_R^\eps \cdot \nabla_x v_0)}{\nabla_x^s v_R^\eps}
		= \sum_{s=s_1 + s_2} \skp{\nabla_x^{s_1} v_R^\eps \cdot \nabla_x^{s_2 +1} v_0}{\nabla_x^s v_R^\eps}
		\ls \nm{\nabla_x v_0}_{H^{s+2}_x} \nm{\nabla_x^s v_R^\eps}_{L^2_x}^2.
	\end{align}

Recalling the definition of $Q_f$ and the weighted Poincar\'e inequality \eqref{eq:Poincare-inequ-x}, we have
	\begin{align}
		\abs{\skp{\nabla_x \cdot (\nabla_x^s Q_{f_R^\eps})}{\nabla_x^s v_R^\eps}}
		= & \abs{\iint (\omega \otimes \omega -\tfrac{1}{3} {\rm Id}) \tfrac{\nabla_x^s f_R^\eps}{M_0} \ \nabla_x^{s+1} v_R^\eps M_0 \d \omega \d x} \\\no
		\ls & \nm{\nabla_x^{s+1} v_R^\eps}_{L^2_x} \nm{\tfrac{\nabla_x^s f_R^\eps}{M_0}}_\M \\\no
		\ls & \nm{\nabla_x^{s+1} v_R^\eps}_{L^2_x} \nm{\nabla_{\omega} (\tfrac{\nabla_x^s f_R^\eps}{M_0})}_\M.
	\end{align}

Considering the right-hand side of equation \eqref{eq:derivative-v}, it follows that
	\begin{align}
		\skp{\Delta_x \nabla_x^s v_R^\eps}{\nabla_x^s v_R^\eps} = - \nm{\nabla_x^{s+1} v_R^\eps}_{L^2_x}^2,
	\end{align}
which, together with the above inequalities, enables us to get the estimate for the velocity field, as follows
	\begin{align}\label{esm:derivatives-v}
		& \tfrac{1}{2} \tfrac{\d}{\d t} \nm{\nabla_x^s v_R^\eps}_{L^2_x}^2 + \nm{\nabla_x^{s+1} v_R^\eps}_{L^2_x}^2
	\\\no
		\ls_{_{C_0}} & \nm{\nabla_x^s v_R^\eps}_{L^2_x}^2 + \sqrt\eps \nm{\nabla_x v_R^\eps}_{H^s_x} \nm{\nabla_x^s v_R^\eps}_{L^2_x}^2
			+ \nm{\nabla_x^{s+1} v_R^\eps}_{L^2_x} \nm{\nabla_{\omega} (\tfrac{\nabla_x^s f_R^\eps}{M_0})}_\M,
	\end{align}
where the notation $A \ls_{C_0} B$ means that there exists some constant $C_0$ such that $A \le C_0 B$, and the constant $C_0$ here can be chosen such that $C_0 \ge C |v_0|_{H^{s+3}_x}$.


\subsection{Higher-Order Derivatives Estimates on Pure Spatial Variables} 
\label{sub:2-2}



Applying the derivatives operator $\nabla_x^s$ to equation \eqref{eq:remainder}$_1$, we have
	\begin{align}\label{eq:derivative-f-x}
	 		\p_t \nabla_x^s f_R^\eps + \nabla_x^s (u_0 \cdot \nabla_x f_R^\eps) + \nabla_x^s [\nabla_{\omega} \cdot (\mathcal{F}_0 f_R^\eps)]
    = \frac{1}{\eps} \nabla_x^s \mathcal{L}_{\Omega_0} f_R^\eps - \frac{1}{\sqrt\eps} \nabla_x^s h_0 - \nabla_x^s h_1.
	\end{align}
Multiplying by $\tfrac{\nabla_x^s f_R^\eps}{M_0}$ on both sides of the above equation, and integrating with respect to variables $x$ and $\omega$, it is straightforward to get
	\begin{align}
		\sskp{\p_t \nabla_x^s f_R^\eps}{\tfrac{\nabla_x^s f_R^\eps}{M_0}}
		= \frac{1}{2} \frac{\d}{\d t} \nm{\tfrac{\nabla_x^s f_R^\eps}{M_0}}_\M^2
			+ \frac{1}{2} \iint \kappa \omega \p_t \Omega_0 \abs{\tfrac{\nabla_x^s f_R^\eps}{M_0}}^2 M_0 \d\omega \d x,
	\end{align}
and
	\begin{align}
		& \sskp{\nabla_x^s (u_0 \cdot \nabla_x f_R^\eps)}{\tfrac{\nabla_x^s f_R^\eps}{M_0}} \\\no
		=\ & \sskp{u_0 \cdot \nabla_x^{s+1} f_R^\eps}{\tfrac{\nabla_x^s f_R^\eps}{M_0}}
			+ \sum_{\scriptscriptstyle s= s_1 + s_2, 1\le s_1 \le s} \sskp{\nabla_x^{s_1} v_0 \cdot \nabla_x^{s_2+1} f_R^\eps}{\tfrac{\nabla_x^s f_R^\eps}{M_0}} \\\no
		=\ & \iint (v_0+a \omega) \kappa \omega \nabla_x \Omega_0 \abs{\tfrac{\nabla_x^s f_R^\eps}{M_0}}^2 M_0 \d\omega \d x
			+ \sum_{\scriptscriptstyle s= s_1 + s_2, 1\le s_1 \le s} \sskp{\nabla_x^{s_1} v_0 \cdot \nabla_x^{s_2+1} f_R^\eps}{\tfrac{\nabla_x^s f_R^\eps}{M_0}} \\\no
		\ls \ & \|v_0\|_{L^\infty_x} \|\nabla_x \Omega_0\|_{L^\infty_x} \nm{\tfrac{\nabla_x^s f_R^\eps}{M_0}}_\M^2
				+ \sum_{\scriptscriptstyle\stackrel{s= s_1 + s_2}{1\le s_1 \le s, 0\le s_2 \le s-1}}
					\|\nabla_x^{s_1} v_0\|_{L^\infty_x} \nm{\tfrac{\nabla_x^{s_2 +1} f_R^\eps}{M_0}}_\M \nm{\tfrac{\nabla_x^s f_R^\eps}{M_0}}_\M,
	\end{align}
where we have used the fact that $v_0$ is divergence free and the H\"{o}lder inequality.

For the third term, it follows from $\mathcal{F}_0 = \mathcal{P}_{\omega^\perp} (F_0 + B(v_0)\omega)$ that
	\begin{align}
		\sskp{\nabla_x^s [\nabla_{\omega} \cdot (\mathcal{F}_0 f_R^\eps)]}{\tfrac{\nabla_x^s f_R^\eps}{M_0}}
		=\ & \sum_{s= s_1 + s_2} \sskp{\mathcal{P}_{\omega^\perp} \nabla_x^{s_1} (F_0 + B(v_0)\omega) \nabla_x^{s_2} f_R^\eps}{\nabla_{\omega} (\tfrac{\nabla_x^s f_R^\eps}{M_0})} \\\no
		\ls \ & \sum_{s= s_1 + s_2} (|\nabla_x^{s_1} F_0|_{L^\infty_x} + |\nabla_x^{s_1} B(v_0)|_{L^\infty_x}) \nm{\tfrac{\nabla_x^{s_2} f_R^\eps}{M_0}}_\M \nm{\nabla_{\omega} (\tfrac{\nabla_x^s f_R^\eps}{M_0})}_\M.
	\end{align}

We next deal with the terms on the right-hand side. Notice the commutation formula
	\begin{align*}
	  \nabla_x^s \mathcal{L}_{\Omega_0} f
	  = \nabla_x^s [\nabla_{\omega} \cdot (\nabla_{\omega} f - \kappa \mathcal{P}_{\omega^\perp} \Omega_0 f)]
	  = \mathcal{L}_{\Omega_0} \nabla_x^s f - \sum_{\scrpt{s= s_1 + s_2}{1\le s_1 \le s}} \nabla_{\omega} \cdot (\kappa \mathcal{P}_{\omega^\perp} \nabla_x^{s_1} \Omega_0 \nabla_x^{s_2} f),
	\end{align*}
which, combining with the H\"{o}lder inequality, enables us to infer that
	\begin{align}
		& \sskp{\tfrac{1}{\eps} \nabla_x^s \mathcal{L}_{\Omega_0} f_R^\eps}{\tfrac{\nabla_x^s f_R^\eps}{M_0}} \\\no
		= & \frac{1}{\eps} \sskp{\mathcal{L}_{\Omega_0} \nabla_x^s f_R^\eps}{\tfrac{\nabla_x^s f_R^\eps}{M_0}}
			 - \frac{1}{\eps} \sum_{\scrpt{s= s_1 + s_2}{1\le s_1 \le s}} \sskp{\nabla_{\omega} \cdot (\kappa \mathcal{P}_{\omega^\perp} \nabla_x^{s_1} \Omega_0 \nabla_x^{s_2} f_R^\eps)}{\tfrac{\nabla_x^s f_R^\eps}{M_0}} \\\no
		= & - \frac{1}{\eps} \nm{\nabla_{\omega} (\tfrac{\nabla_x^s f_R^\eps}{M_0})}_\M^2
				 + \frac{1}{\eps} \sum_{\scrpt{s= s_1 + s_2}{1\le s_1 \le s}} \iint \kappa \mathcal{P}_{\omega^\perp} \nabla_x^{s_1} \Omega_0 \tfrac{\nabla_x^{s_2} f_R^\eps}{M_0} \nabla_{\omega} (\tfrac{\nabla_x^s f_R^\eps}{M_0}) M_0 \d \omega \d x \\\no
		\le & - \frac{1}{\eps} \nm{\nabla_{\omega} (\tfrac{\nabla_x^s f_R^\eps}{M_0})}_\M^2
		    + \frac{1}{\eps} C \sum_{\scrpt{s= s_1 + s_2}{1\le s_1 \le s}} \nm{\nabla_x^{s_1} \Omega_0}_{L^\infty} \nm{\tfrac{\nabla_x^{s_2} f_R^\eps}{M_0}}_\M \nm{\nabla_{\omega} (\tfrac{\nabla_x^s f_R^\eps}{M_0})}_\M \\\no
		\le & - \frac{1}{\eps} \nm{\nabla_{\omega} (\tfrac{\nabla_x^s f_R^\eps}{M_0})}_\M^2
		    + \frac{1}{\eps} C \sum_{\scrpt{s= s_1 + s_2}{1\le s_1 \le s}}
		    								\nm{\nabla_x^{s_1} \Omega_0}_{L^\infty}
		    								\nm{\nabla_{\omega} (\tfrac{\nabla_x^{s_2} f_R^\eps}{M_0})}_\M
		    								\nm{\nabla_{\omega} (\tfrac{\nabla_x^s f_R^\eps}{M_0})}_\M,
	\end{align}
where in the last line we have used the weighted Poincar\'e inequality \eqref{eq:Poincare-inequ-x} in Lemma \ref{lemm:Poincare-inequ} and the equivalence between the two norms $\nm{\cdot}_{L^2_{x,\omega}}$ and $\nm{\cdot}_\M$.

At the same time, the weighted Poincar\'e inequality \eqref{eq:Poincare-inequ-x} also yields that
	\begin{align}
	  \sskp{\tfrac{1}{\sqrt\eps} \nabla_x^s h_0}{\tfrac{\nabla_x^s f_R^\eps}{M_0}}
	  \ls \nm{\tfrac{\nabla_x^s h_0}{M_0}}_\M \cdot
	  		\tfrac{1}{\sqrt\eps} \nm{\tfrac{\nabla_x^s f_R^\eps}{M_0}}_\M
	  \ls \nm{\tfrac{\nabla_x^s h_0}{M_0}}_\M \cdot
	  		\tfrac{1}{\sqrt\eps} \nm{\nabla_{\omega} (\tfrac{\nabla_x^s f_R^\eps}{M_0})}_\M.
	\end{align}

We are now left to estimate the contributions of $$\nabla_x^s h_1 = \nabla_x^s \left\{ v_R^\eps \cdot \nabla_x (f_0 + \sqrt\eps f_R^\eps) + \nabla_\omega \cdot [\mathcal{P}_{\omega^\perp}B(v_R^\eps) \omega (f_0 + \sqrt\eps f_R^\eps)] \right\}.$$ Some straightforward calculations imply that
	\begin{align}
	  \sskp{\nabla_x^s (v_R^\eps \cdot \nabla_x f_0)}{\tfrac{\nabla_x^s f_R^\eps}{M_0}}
	  \ls & \sum_{\scrpt{s= s_1 + s_2}{0\le s_1 \le s}} \nm{\nabla_x^{s_1} v_R^\eps}_{L^2_x}
		    								\nm{\tfrac{\nabla_x^{s_2 +1} f_0}{M_0}}_{L^\infty_x L^2_{\M}}
		    								\nm{\tfrac{\nabla_x^s f_R^\eps}{M_0}}_\M, \\\no
		\sskp{\nabla_x^s [\nabla_\omega \cdot (\mathcal{P}_{\omega^\perp}B(v_R^\eps) \omega f_0)]}{\tfrac{\nabla_x^s f_R^\eps}{M_0}}
		\ls & \sum_{\scrpt{s= s_1 + s_2}{0\le s_1 \le s}} \nm{\nabla_x^{s_1 +1} v_R^\eps}_{L^2_x}
		    								\nm{\tfrac{\nabla_x^{s_2} f_0}{M_0}}_{L^\infty_x L^2_{\M}}
		    								\nm{\nabla_{\omega} (\tfrac{\nabla_x^s f_R^\eps}{M_0})}_\M.
	\end{align}

Considering the last two terms of order $\sqrt\eps$, it follows from the fact $v_R^\eps$ is divergence free, and some similar but more delicate process that,
	\begin{align}
		& \sskp{\nabla_x^s (v_R^\eps \cdot \nabla_x f_R^\eps)}{\tfrac{\nabla_x^s f_R^\eps}{M_0}} \\\no
		=\ & \sskp{v_R^\eps \cdot \nabla_x^{s+1} f_R^\eps}{\tfrac{\nabla_x^s f_R^\eps}{M_0}}
			+ \sskp{\nabla_x^s v_R^\eps \cdot \nabla_x f_R^\eps}{\tfrac{\nabla_x^s f_R^\eps}{M_0}}
			+ \sum_{\scrpt{s= s_1 + s_2}{1\le s_1 \le s-1}} \sskp{\nabla_x^{s_1} v_R^\eps \cdot \nabla_x^{s_2+1} f_R^\eps}{\tfrac{\nabla_x^s f_R^\eps}{M_0}} \\\no
		=\ & \frac{1}{2} \iint v_R^\eps (\kappa \omega \nabla_x \Omega_0) \abs{\tfrac{\nabla_x^s f_R^\eps}{M_0}}^2 M_0 \d\omega \d x
			+ \nm{\nabla_x^s v_R^\eps}_{L^4_x} \nm{\nabla_x f_R^\eps}_{L^4_x L^2_{\omega}} \nm{\tfrac{\nabla_x^s f_R^\eps}{M_0}}_{L^2_{x,\omega}} \\\no
			& + \sum_{\scrpt{s= s_1 + s_2}{1\le s_1 \le s-1}} \nm{\nabla_x^{s_1} v_R^\eps}_{L^\infty_x}  \nm{\tfrac{\nabla_x^{s_2+1} f_R^\eps}{M_0}}_\M \nm{\tfrac{\nabla_x^s f_R^\eps}{M_0}} \\\no
		\ls \ & \|\nabla_x \Omega_0\|_{L^\infty_x} \|v_R^\eps\|_{H^2_x} \nm{\tfrac{\nabla_x^s f_R^\eps}{M_0}}_\M^2
				+ \nm{\nabla_x^{s+1} v_R^\eps}_{L^2_x} \nm{\tfrac{\nabla_x^2 f_R^\eps}{M_0}}_\M \nm{\tfrac{\nabla_x^s f_R^\eps}{M_0}}_\M \\\no
				& + \sum_{\scriptscriptstyle\stackrel{s= s_1 + s_2}{1\le s_1,s_2 \le s-1}}
					|\nabla_x^{s_1} v_R^\eps|_{H^2_x} \nm{\tfrac{\nabla_x^{s_2 +1} f_R^\eps}{M_0}}_\M \nm{\tfrac{\nabla_x^s f_R^\eps}{M_0}}_\M,
	\end{align}	
where we have used the Sobolev embedding inequalities and the equivalence between the two norms $\nm{\cdot}_{L^2_{x,\omega}}$ and $\nm{\cdot}_\M$. That also entails that the last term can be controlled as follows,
	\begin{align}
	  	& \sskp{\nabla_x^s [\nabla_\omega \cdot (\mathcal{P}_{\omega^\perp}B(v_R^\eps) \omega f_R^\eps)]}{\tfrac{\nabla_x^s f_R^\eps}{M_0}} \\\no
		= & \sskp{\mathcal{P}_{\omega^\perp}B(v_R^\eps) \omega \nabla_x^s f_R^\eps}{\nabla_\omega (\tfrac{\nabla_x^s f_R^\eps}{M_0})}
			+ \sskp{\mathcal{P}_{\omega^\perp} \nabla_x^s B(v_R^\eps) \omega f_R^\eps}{\nabla_\omega (\tfrac{\nabla_x^s f_R^\eps}{M_0})} \\\no
			& + \sum_{\scrpt{s= s_1 + s_2}{1\le s_1 \le s-1}}
					\sskp{\mathcal{P}_{\omega^\perp} \nabla_x^{s_1} B(v_R^\eps) \omega \nabla_x^{s_2} f_R^\eps}{\nabla_\omega (\tfrac{\nabla_x^s f_R^\eps}{M_0})} \\\no
		\ls & \nm{\nabla_x v_R^\eps}_{H^2_x} \nm{\tfrac{\nabla_x^s f_R^\eps}{M_0}}_\M \nm{\nabla_\omega (\tfrac{\nabla_x^s f_R^\eps}{M_0})}_\M
			+ \nm{\nabla_x^{s+1} v_R^\eps}_{L^2_x} \nm{\tfrac{\nabla_x^2 f_R^\eps}{M_0}}_\M \nm{\nabla_\omega (\tfrac{\nabla_x^s f_R^\eps}{M_0})}_\M \\\no
			& + \sum_{\scrpt{s= s_1 + s_2}{1\le s_1 \le s-1}} \nm{\nabla_x^{s_1+2} v_R^\eps}_{L^2_x}
				\nm{\tfrac{\nabla_x^{s_2 +1} f_R^\eps}{M_0}}_\M \nm{\nabla_\omega (\tfrac{\nabla_x^s f_R^\eps}{M_0})}_\M.
	\end{align}

Choose
	\begin{align*}
	  C_0
	  \ge & C \big( 1 +\nm{\p_t \Omega_0}_{L^\infty_x} + \nm{v_0}_{L^\infty_x} \nm{\nabla_x \Omega_0}_{L^\infty_x} + \nm{\nabla_x^s v_0}_{L^\infty_x} + \nm{\nabla_x^s \Omega_0}_{L^\infty_x} \\
	  	& + \nm{\nabla_x^s F_0}_{L^\infty_{x,\omega}} + \nm{\nabla_x^s B(v_0)}_{L^\infty_x}
	  		+ \nm{\frac{\nabla_x^s h_0}{M_0}}_\M + \nm{\frac{\nabla_x^s f_0}{M_0}}_{L^\infty_x L^2_\M} \big),
	\end{align*}
then we have derived the estimates on the pure spatial variables that
	\begin{align}\label{esm:derivatives-f-x}
	  & \tfrac{1}{2} \tfrac{\d}{\d t} \nm{\tfrac{\nabla_x^s f_R^\eps}{M_0}}_\M^2 + \tfrac{1}{\eps} \nm{\nabla_\omega (\tfrac{\nabla_x^s f_R^\eps}{M_0})}_\M^2 \\\no
	\ls_{_{C_0}} & \nm{\tfrac{\nabla_x^s f_R^\eps}{M_0}}_\M^2
	  	+ \sum_{0 \le s_2 \le s-1} \nm{\tfrac{\nabla_x^{s_2 +1} f_R^\eps}{M_0}}_\M \nm{\tfrac{\nabla_x^s f_R^\eps}{M_0}}_\M
	  	+ \sum_{0 \le s_2 \le s} \nm{\tfrac{\nabla_x^{s_2} f_R^\eps}{M_0}}_\M \nm{\nabla_\omega (\tfrac{\nabla_x^s f_R^\eps}{M_0})}_\M
	  \\\no
	  & + \sqrt\eps |v_R^\eps|_{H^2_x} \nm{\tfrac{\nabla_x^s f_R^\eps}{M_0}}_\M \nm{\nabla_\omega (\tfrac{\nabla_x^s f_R^\eps}{M_0})}_\M
	   + \sqrt\eps \nm{\nabla_x^{s+1} v_R^\eps}_{L^2_x} \nm{\tfrac{\nabla_x^2 f_R^\eps}{M_0}}_\M \nm{\tfrac{\nabla_x^s f_R^\eps}{M_0}}_\M
	  \\\no
	  & + \sqrt\eps \nm{\nabla_x^{s+1} v_R^\eps}_{L^2_x} \sum_{1\le s_2 \le s-1}
					\nm{\tfrac{\nabla_x^{s_2 +1} f_R^\eps}{M_0}}_\M \nm{\tfrac{\nabla_x^s f_R^\eps}{M_0}}_\M
    \\\no
		& + \sqrt\eps \nm{\nabla_x v_R^\eps}_{H^2_x} \nm{\tfrac{\nabla_x^s f_R^\eps}{M_0}}_\M \nm{\nabla_\omega (\tfrac{\nabla_x^s f_R^\eps}{M_0})}_\M
		  + \sqrt\eps \nm{\nabla_x^{s+1} v_R^\eps}_{L^2_x} \nm{\tfrac{\nabla_x^2 f_R^\eps}{M_0}}_\M \nm{\nabla_\omega (\tfrac{\nabla_x^s f_R^\eps}{M_0})}_\M
    \\\no
	  & + \sqrt\eps \nm{\nabla_x^{s+1} v_R^\eps}_{L^2_x} \sum_{1\le s_2 \le s-1}
	   			\nm{\tfrac{\nabla_x^{s_2 +1} f_R^\eps}{M_0}}_\M \nm{\nabla_\omega (\tfrac{\nabla_x^s f_R^\eps}{M_0})}_\M
	  \\\no
	  & + \frac{1}{\eps} \sum_{0 \le s_2 \le s-1} \nm{\nabla_{\omega} (\tfrac{\nabla_x^{s_2} f_R^\eps}{M_0})}_\M \nm{\nabla_{\omega} (\tfrac{\nabla_x^s f_R^\eps}{M_0})}_\M
	  	+ \tfrac{1}{\sqrt\eps} \nm{\nabla_{\omega} (\tfrac{\nabla_x^s f_R^\eps}{M_0})}_\M
	  \\\no
	  &	+ \nm{\nabla_x^s v_R^\eps}_{L^2_x} \nm{\tfrac{\nabla_x^s f_R^\eps}{M_0}}_\M
	  	+ \nm{\nabla_x^{s +1} v_R^\eps}_{L^2_x} \nm{\nabla_{\omega} (\tfrac{\nabla_x^s f_R^\eps}{M_0})}_\M.
 	\end{align}


\subsection{Higher-Order Derivatives Estimates on Mixed Variables} 
\label{sub:2-3}



Apply the mixed derivatives operator $\nabla_x^k \nabla_{\omega}^l$ with $k+l=s,\ l \ge 1$ to equation \eqref{eq:remainder}$_1$, then we get
	\begin{multline}\label{eq:derivative-f-xw}
	 		\p_t \nabla_x^k \nabla_{\omega}^l f_R^\eps + \nabla_x^k \nabla_{\omega}^l (u_0 \cdot \nabla_x f_R^\eps) + \nabla_x^k \nabla_{\omega}^l [\nabla_{\omega} \cdot (\mathcal{F}_0 f_R^\eps)] \\
    = \frac{1}{\eps} \nabla_x^k \nabla_{\omega}^l \mathcal{L}_{\Omega_0} f_R^\eps - \frac{1}{\sqrt\eps} \nabla_x^k \nabla_{\omega}^l h_0 - \nabla_x^k \nabla_{\omega}^l h_1.
	\end{multline}

Taking $L^2_{x,\omega}$ inner product with the quantity $\tfrac{\nabla_x^k \nabla_{\omega}^l f_R^\eps}{M_0}$ yields that
	\begin{align}
		\sskp{\p_t \nabla_x^k \nabla_{\omega}^l f_R^\eps}{\tfrac{\nabla_x^k \nabla_{\omega}^l f_R^\eps}{M_0}}
		= \frac{1}{2} \frac{\d}{\d t} \nm{\tfrac{\nabla_x^k \nabla_{\omega}^l f_R^\eps}{M_0}}_\M^2
			+ \frac{1}{2} \iint \kappa \omega \p_t \Omega_0 \abs{\tfrac{\nabla_x^k \nabla_{\omega}^l f_R^\eps}{M_0}}^2 M_0 \d\omega \d x,
	\end{align}
and
	\begin{align}
			& \sskp{\nabla_x^k \nabla_{\omega}^l (u_0 \cdot \nabla_x f_R^\eps)}{\tfrac{\nabla_x^k \nabla_{\omega}^l f_R^\eps}{M_0}} \\\no
		= & \sskp{(v_0 + a \omega) \cdot \nabla_x^{k+1} \nabla_{\omega}^l f_R^\eps}{\tfrac{\nabla_x^k \nabla_{\omega}^l f_R^\eps}{M_0}}
			+ \sum_{\scrpt{k= k_1 + k_2}{1\le k_1 \le k}} \sskp{\nabla_x^{k_1} v_0 \cdot \nabla_x^{k_2+1} \nabla_{\omega}^l f_R^\eps}{\tfrac{\nabla_x^k \nabla_{\omega}^l f_R^\eps}{M_0}} \\\no
			& + \sskp{a \nabla_x^{k+1} \nabla_{\omega}^{l-1} f_R^\eps}{\tfrac{\nabla_x^k \nabla_{\omega}^l f_R^\eps}{M_0}}
		\\\no
		\ls & \|v_0\|_{L^\infty_x} \|\nabla_x \Omega_0\|_{L^\infty_x} \nm{\tfrac{\nabla_x^k \nabla_{\omega}^l f_R^\eps}{M_0}}_\M^2
				+ \sum_{0\le k_2 \le k-1} \|\nabla_x^k v_0\|_{L^\infty_x} \nm{\tfrac{\nabla_x^{k_2 +1} \nabla_{\omega}^l f_R^\eps}{M_0}}_\M \nm{\tfrac{\nabla_x^k \nabla_{\omega}^l f_R^\eps}{M_0}}_\M \\\no
				& + \nm{\tfrac{\nabla_x^{k+1} \nabla_{\omega}^{l-1} f_R^\eps}{M_0}}_\M \nm{\tfrac{\nabla_x^k \nabla_{\omega}^l f_R^\eps}{M_0}}_\M
	\end{align}
where we have used the fact that $v_0$ is divergence free and the H\"{o}lder inequality. Note that by the Sobolev embedding inequality we have $|\nabla_x^k v_0|_{L^\infty_x} \le C |v_0|_{H^{s+1}_x}$ since $k + 2\le s -l +2\le s+1$.

For the third term, it follows that
	\begin{align}
			& \sskp{\nabla_x^k \nabla_{\omega}^l [\nabla_{\omega} \cdot (\mathcal{F}_0 f_R^\eps)]}{\tfrac{\nabla_x^k \nabla_{\omega}^l f_R^\eps}{M_0}} \\\no
		= & \sskp{\nabla_x^k \nabla_{\omega}^l [(F_0 + B(v_0)\omega) f_R^\eps]}{\nabla_{\omega} (\tfrac{\nabla_x^k \nabla_{\omega}^l f_R^\eps}{M_0})}
	\\\no
		= & \sum_{k = k_1 + k_2} \sskp{\nabla_x^{k_1} [F_0 + B(v_0)\omega] \nabla_x^{k_2} \nabla_{\omega}^l f_R^\eps + \nabla_x^{k_1} B(v_0) \nabla_x^{k_2} \nabla_{\omega}^{l-1} f_R^\eps}{\nabla_{\omega} (\tfrac{\nabla_x^k \nabla_{\omega}^l f_R^\eps}{M_0})} \\\no
		\ls & \sum_{k = k_1 + k_2} (\|\nabla_x^{k_1} F_0\|_{L^\infty_x} + \|\nabla_x^{k_1} B(v_0)\|_{L^\infty_x}) \nm{\tfrac{\nabla_x^{k_2} \nabla_{\omega}^l f_R^\eps}{M_0}}_\M \nm{\nabla_{\omega} (\tfrac{\nabla_x^k \nabla_{\omega}^l f_R^\eps}{M_0})}_\M \\\no
				& + \sum_{k = k_1 + k_2} \nm{\nabla_x^{k_1} B(v_0)}_{L^\infty_x} \nm{\tfrac{\nabla_x^{k_2} \nabla_{\omega}^{l-1} f_R^\eps}{M_0}}_\M \nm{\nabla_{\omega} (\tfrac{\nabla_x^k \nabla_{\omega}^l f_R^\eps}{M_0})}_\M.
	\end{align}

We next turn to consider the right-hand side terms. Firstly we can derive the following commutation formula
	\begin{align*}
			[\nabla_x^k \nabla_{\omega}^l,\ \mathcal{L}_{\Omega_0}] f
	  = - \sum_{\scrpt{k = k_1 + k_2}{0 \le k_1 \le k}} 2\kappa \nabla_x^{k_1} \Omega_0 \nabla_x^{k_2} \nabla_{\omega}^{l-1} f
	  	- \sum_{\scrpt{k = k_1 + k_2}{1 \le k_1 \le k}} \nabla_{\omega} \cdot (\kappa \mathcal{P}_{\omega^\perp} \nabla_x^{k_1} \Omega_0 \nabla_x^{k_2} \nabla_{\omega}^l f),
	\end{align*}
which implies immediately that
	\begin{align}
		& \sskp{\tfrac{1}{\eps} \nabla_x^k \nabla_{\omega}^l \mathcal{L}_{\Omega_0} f_R^\eps}{\tfrac{\nabla_x^k \nabla_{\omega}^l f_R^\eps}{M_0}} \\\no
		= & \frac{1}{\eps} \sskp{\mathcal{L}_{\Omega_0} \nabla_x^k \nabla_{\omega}^l f_R^\eps}{\tfrac{\nabla_x^k \nabla_{\omega}^l f_R^\eps}{M_0}}
			 - \frac{1}{\eps} \sum_{\scrpt{k = k_1 + k_2}{0 \le k_1 \le k}} \sskp{2\kappa \nabla_x^{k_1} \Omega_0 \nabla_x^{k_2} \nabla_{\omega}^{l-1} f_R^\eps)}{\tfrac{\nabla_x^k \nabla_{\omega}^l f_R^\eps}{M_0}} \\\no
			& + \frac{1}{\eps} \sum_{\scrpt{k = k_1 + k_2}{1\le k_1 \le k}} \sskp{\kappa \mathcal{P}_{\omega^\perp} \nabla_x^{k_1} \Omega_0 \nabla_x^{k_2} \nabla_{\omega}^l f_R^\eps}{\nabla_{\omega} (\tfrac{\nabla_x^k \nabla_{\omega}^l f_R^\eps}{M_0})} \\\no
		\le & - \frac{1}{\eps} \nm{\nabla_{\omega} (\tfrac{\nabla_x^k \nabla_{\omega}^l f_R^\eps}{M_0})}_\M^2
					+ \frac{1}{\eps} C \sum_{\scrpt{k = k_1 + k_2}{0 \le k_1 \le k}} \nm{\nabla_x^{k_1} \Omega_0}_{L^\infty} \nm{\tfrac{\nabla_x^{k_2} \nabla_{\omega}^{l-1} f_R^\eps}{M_0}}_\M \nm{\tfrac{\nabla_x^k \nabla_{\omega}^l f_R^\eps}{M_0}}_\M \\\no
				& + \frac{1}{\eps} C \sum_{\scrpt{k= k_1 + k_2}{1\le k_1 \le k}} \nm{\nabla_x^{k_1} \Omega_0}_{L^\infty} \nm{\tfrac{\nabla_x^{k_2} \nabla_{\omega}^l f_R^\eps}{M_0}}_\M \nm{\nabla_{\omega} (\tfrac{\nabla_x^k \nabla_{\omega}^l f_R^\eps}{M_0})}_\M \\\no
		\le & - \frac{1}{\eps} \nm{\nabla_{\omega} (\tfrac{\nabla_x^k \nabla_{\omega}^l f_R^\eps}{M_0})}_\M^2
					+ \frac{1}{\eps} C \nm{\nabla_x^k \Omega_0}_{L^\infty} \sum_{\scrpt{0\le k_2 \le k}{l_1\le l-1, l_2 \le l-2}} \nm{\nabla_{\omega} (\tfrac{\nabla_x^{k_2} \nabla_{\omega}^{l_2} f_R^\eps}{M_0})}_\M \nm{\nabla_{\omega} (\tfrac{\nabla_x^k \nabla_{\omega}^{l_1} f_R^\eps}{M_0})}_\M \\\no
		    & + \frac{1}{\eps} C \nm{\nabla_x^k \Omega_0}_{L^\infty} \sum_{\scrpt{0\le k_2 \le k-1}{l_2 \le l-1}} \nm{\nabla_{\omega} (\tfrac{\nabla_x^{k_2} \nabla_{\omega}^{l_2} f_R^\eps}{M_0})}_\M \nm{\nabla_{\omega} (\tfrac{\nabla_x^k \nabla_{\omega}^l f_R^\eps}{M_0})}_\M,
	\end{align}
where in the last line we have used the weighted Poincar\'e inequality \eqref{eq:Poincare-inequ-xw} in Lemma \ref{lemm:Poincare-inequ}. And we also have
	\begin{align}
	  \sskp{\tfrac{1}{\sqrt\eps} \nabla_x^k \nabla_{\omega}^l h_0}{\tfrac{\nabla_x^k \nabla_{\omega}^l f_R^\eps}{M_0}}
	  \ls & \nm{\tfrac{\nabla_x^k \nabla_{\omega}^l h_0}{M_0}}_\M \cdot
	  		\tfrac{1}{\sqrt\eps} \nm{\tfrac{\nabla_x^k \nabla_{\omega}^l f_R^\eps}{M_0}}_\M \\\no
	  \ls & \nm{\tfrac{\nabla_x^k \nabla_{\omega}^l h_0}{M_0}}_\M \cdot
	  		\tfrac{1}{\sqrt\eps} \sum_{l'\le l-1} \nm{\nabla_{\omega} (\tfrac{\nabla_x^k \nabla_{\omega}^{l'} f_R^\eps}{M_0})}_\M.
	\end{align}	

We then need to treat the term of $$\nabla_x^k \nabla_{\omega}^l h_1 = \nabla_x^s \nabla_{\omega}^l \left\{ v_R^\eps \cdot \nabla_x (f_0 + \sqrt\eps f_R^\eps) + \nabla_\omega \cdot [\mathcal{P}_{\omega^\perp}B(v_R^\eps) \omega (f_0 + \sqrt\eps f_R^\eps)] \right\}.$$

Firstly performing some direct calculations yields that
	\begin{align}
	  	\sskp{\nabla_x^k \nabla_{\omega}^l (v_R^\eps \cdot \nabla_x f_0)}{\tfrac{\nabla_x^k \nabla_{\omega}^l f_R^\eps}{M_0}}
	  \ls \sum_{\scrpt{k = k_1 + k_2}{0\le k_1 \le k}} \nm{\nabla_x^{k_1} v_R^\eps}_{L^2_x}
		    								\nm{\tfrac{\nabla_x^{k_2 +1} \nabla_{\omega}^l f_0}{M_0}}_{L^\infty_x L^2_{\M}}
		    								\nm{\tfrac{\nabla_x^k \nabla_{\omega}^l f_R^\eps}{M_0}}_\M,
	\end{align}
and
	\begin{align}
			& \sskp{\nabla_x^k \nabla_{\omega}^l [\nabla_\omega \cdot (\mathcal{P}_{\omega^\perp}B(v_R^\eps) \omega f_0)]}{\tfrac{\nabla_x^k \nabla_{\omega}^l f_R^\eps}{M_0}} \\\no
		\ls & \sum_{\scrpt{k = k_1 + k_2}{0\le k_1 \le k}} \nm{\nabla_x^{k_1 +1} v_R^\eps}_{L^2_x}
		    								\left( \nm{\tfrac{\nabla_x^{k_2} \nabla_{\omega}^l f_0}{M_0}}_{L^\infty_x L^2_{\M}} + \nm{\tfrac{\nabla_x^{k_2} \nabla_{\omega}^{l-1} f_0}{M_0}}_{L^\infty_x L^2_{\M}} \right)
		    								\nm{\nabla_{\omega} (\tfrac{\nabla_x^k \nabla_{\omega}^l f_R^\eps}{M_0})}_\M.
	\end{align}

As for the last two terms of order $\sqrt\eps$, some more delicate arguments imply that,
	\begin{align}
		& \sskp{\nabla_x^k \nabla_{\omega}^l (v_R^\eps \cdot \nabla_x f_R^\eps)}{\tfrac{\nabla_x^k \nabla_{\omega}^l f_R^\eps}{M_0}} \\\no
		=\ & \sskp{v_R^\eps \cdot \nabla_x^{k+1} \nabla_{\omega}^l f_R^\eps}{\tfrac{\nabla_x^k \nabla_{\omega}^l f_R^\eps}{M_0}}
			+ \sum_{\scrpt{k = k_1 + k_2}{1\le k_1 \le k}} \sskp{\nabla_x^{k_1} v_R^\eps \cdot \nabla_x^{k_2+1} \nabla_{\omega}^l f_R^\eps}{\tfrac{\nabla_x^k \nabla_{\omega}^l f_R^\eps}{M_0}} \\\no
		\ls \ & \nm{\nabla_x \Omega_0}_{L^\infty_x} \nm{v_R^\eps}_{H^2_x} \nm{\tfrac{\nabla_x^k \nabla_{\omega}^l f_R^\eps}{M_0}}_\M^2
				+ \nm{\nabla_x^k v_R^\eps}_{H^2_x} \sum_{0 \le k_2 \le k-1}
					\nm{\tfrac{\nabla_x^{k_2 +1} \nabla_{\omega}^l f_R^\eps}{M_0}}_\M \nm{\tfrac{\nabla_x^k \nabla_{\omega}^l f_R^\eps}{M_0}}_\M,
	\end{align}	
where we have used the Sobolev embedding inequalities (Note that again $k+2 \le s+1$).

The last term can be controlled as follows,
	\begin{align}
	  	& \sskp{\nabla_x^k \nabla_{\omega}^l [\nabla_\omega \cdot (\mathcal{P}_{\omega^\perp}B(v_R^\eps) \omega f_R^\eps)]}{\tfrac{\nabla_x^k \nabla_{\omega}^l f_R^\eps}{M_0}} \\\no
		= & \sskp{B(v_R^\eps) \omega \nabla_x^k \nabla_{\omega}^l f_R^\eps}{\nabla_\omega (\tfrac{\nabla_x^k \nabla_{\omega}^l f_R^\eps}{M_0})}
				+ \sum_{\scrpt{k = k_1 + k_2}{1\le k_1 \le k}}
					\sskp{\nabla_x^{k_1} B(v_R^\eps) \omega \nabla_x^{k_2} \nabla_{\omega}^l f_R^\eps}{\nabla_\omega (\tfrac{\nabla_x^k \nabla_{\omega}^l f_R^\eps}{M_0})} \\\no
		  & + \sskp{B(v_R^\eps) \nabla_x^k \nabla_{\omega}^{l-1} f_R^\eps}{\nabla_\omega (\tfrac{\nabla_x^k \nabla_{\omega}^l f_R^\eps}{M_0})}
				+ \sum_{\scrpt{k = k_1 + k_2}{1\le k_1 \le k}}
					\sskp{\nabla_x^{k_1} B(v_R^\eps) \nabla_x^{k_2} \nabla_{\omega}^{l-1} f_R^\eps}{\nabla_\omega (\tfrac{\nabla_x^k \nabla_{\omega}^l f_R^\eps}{M_0})} \\\no
		\ls & \nm{\nabla_x v_R^\eps}_{L^\infty_x} \nm{\tfrac{\nabla_x^k \nabla_{\omega}^l f_R^\eps}{M_0}}_\M \nm{\nabla_\omega (\tfrac{\nabla_x^k \nabla_{\omega}^l f_R^\eps}{M_0})}_\M
				+ \nm{\nabla_x v_R^\eps}_{L^\infty_x} \nm{\tfrac{\nabla_x^k \nabla_{\omega}^{l-1} f_R^\eps}{M_0}}_\M \nm{\nabla_\omega (\tfrac{\nabla_x^k \nabla_{\omega}^l f_R^\eps}{M_0})}_\M\\\no
			  & + \sum_{\scrpt{k = k_1 + k_2}{1 \le k_1 \le k}} \nm{\nabla_x^{k_1} v_R^\eps}_{L^4_x}
				\nm{\nabla_x^{k_2} \nabla_{\omega}^l f_R^\eps}_{L^4_x L^2_{\omega}} \nm{\nabla_\omega (\tfrac{\nabla_x^s f_R^\eps}{M_0})}_\M \\\no
				& + \sum_{\scrpt{k = k_1 + k_2}{1 \le k_1 \le k}} \nm{\nabla_x^{k_1} v_R^\eps}_{L^4_x}
				\nm{\nabla_x^{k_2} \nabla_{\omega}^{l-1} f_R^\eps}_{L^4_x L^2_{\omega}} \nm{\nabla_\omega (\tfrac{\nabla_x^s f_R^\eps}{M_0})}_\M \\\no
		\ls & \nm{\nabla_x v_R^\eps}_{H^2_x}
						\left( \nm{\tfrac{\nabla_x^k \nabla_{\omega}^l f_R^\eps}{M_0}}_\M
									+ \nm{\tfrac{\nabla_x^k \nabla_{\omega}^{l-1} f_R^\eps}{M_0}}_\M \right)
					\nm{\nabla_\omega (\tfrac{\nabla_x^k \nabla_{\omega}^l f_R^\eps}{M_0})}_\M\\\no
			  & + \nm{\nabla_x^{k+1} v_R^\eps}_{L^2_x}
			  			\sum_{0 \le k_2 \le k-1}
			  			\left( \nm{\tfrac{\nabla_x^{k_2+1} \nabla_{\omega}^l f_R^\eps}{M_0}}_\M
			  						+ \nm{\tfrac{\nabla_x^{k_2+1} \nabla_{\omega}^{l-1} f_R^\eps}{M_0}}_\M \right)
			   		\nm{\nabla_\omega (\tfrac{\nabla_x^k \nabla_{\omega}^l f_R^\eps}{M_0})}_\M,
	\end{align}
where we have used the Sobolev embedding inequalities again.

At last, by choosing $C_0$ such that
	\begin{align*}
	  C_0
	  \ge & C \big( 1 +\nm{\p_t \Omega_0}_{L^\infty_x} + \nm{v_0}_{L^\infty_x} \nm{\nabla_x \Omega_0}_{L^\infty_x} + \nm{\nabla_x^s v_0}_{L^\infty_x} + \nm{\nabla_x^s \Omega_0}_{L^\infty_x} \\
	  	& + \nm{\nabla_x^s F_0}_{L^\infty_{x,\omega}} + \nm{\nabla_x^s B(v_0)}_{L^\infty_x}
	  		+ \nm{\frac{\nabla_{x,\omega}^s h_0}{M_0}}_\M + \nm{\frac{\nabla_{x,\omega}^s f_0}{M_0}}_{L^\infty_x L^2_\M} \big),
	\end{align*}
we thus obtain the estimates for mixed derivatives that
	\begin{align}\label{esm:derivatives-f-xw}
	  & \tfrac{1}{2} \tfrac{\d}{\d t} \nm{\tfrac{\nabla_x^k \nabla_{\omega}^l f_R^\eps}{M_0}}_\M^2 + \tfrac{1}{\eps} \nm{\nabla_\omega (\tfrac{\nabla_x^k \nabla_{\omega}^l f_R^\eps}{M_0})}_\M^2 \\\no
	\ls_{_{C_0}} & \nm{\tfrac{\nabla_x^k \nabla_{\omega}^l f_R^\eps}{M_0}}_\M^2
	  	+ \nm{\tfrac{\nabla_x^{k+1} \nabla_{\omega}^{l-1} f_R^\eps}{M_0}}_\M \nm{\tfrac{\nabla_x^k \nabla_{\omega}^l f_R^\eps}{M_0}}_\M
	  	+ \sum_{0 \le k_2 \le k-1} \nm{\tfrac{\nabla_x^{k_2 +1} \nabla_{\omega}^l f_R^\eps}{M_0}}_\M \nm{\tfrac{\nabla_x^k \nabla_{\omega}^l f_R^\eps}{M_0}}_\M
	  \\\no
	  &	+ \sum_{0 \le k_2 \le k} \left( \nm{\tfrac{\nabla_x^{k_2} \nabla_{\omega}^l f_R^\eps}{M_0}}_\M + \nm{\tfrac{\nabla_x^{k_2} \nabla_{\omega}^{l-1} f_R^\eps}{M_0}}_\M \right) \nm{\nabla_\omega (\tfrac{\nabla_x^k \nabla_{\omega}^l f_R^\eps}{M_0})}_\M
	  \\\no
	  & + \sqrt\eps \nm{v_R^\eps}_{H^2_x} \nm{\tfrac{\nabla_x^k \nabla_{\omega}^l f_R^\eps}{M_0}}_\M  \sum_{l'\le l-1} \nm{\nabla_\omega (\tfrac{\nabla_x^k \nabla_{\omega}^{l'} f_R^\eps}{M_0})}_\M
	  \\\no
	  & + \sqrt\eps \nm{\nabla_x^{s+1} v_R^\eps}_{L^2_x} \sum_{0 \le k_2 \le k-1}
					\nm{\tfrac{\nabla_x^{k_2 +1} \nabla_{\omega}^l f_R^\eps}{M_0}}_\M \nm{\tfrac{\nabla_x^k \nabla_{\omega}^l f_R^\eps}{M_0}}_\M
	  \\\no
	  & + \sqrt\eps \nm{\nabla_x v_R^\eps}_{H^2_x}
						\left( \nm{\tfrac{\nabla_x^k \nabla_{\omega}^l f_R^\eps}{M_0}}_\M
									+ \nm{\tfrac{\nabla_x^k \nabla_{\omega}^{l-1} f_R^\eps}{M_0}}_\M \right)
					\nm{\nabla_\omega (\tfrac{\nabla_x^k \nabla_{\omega}^l f_R^\eps}{M_0})}_\M
		\\\no
		& + \sqrt\eps \nm{\nabla_x^s v_R^\eps}_{L^2_x}
			  			\sum_{0 \le k_2 \le k-1}
			  			\left( \nm{\tfrac{\nabla_x^{k_2+1} \nabla_{\omega}^l f_R^\eps}{M_0}}_\M
			  						+ \nm{\tfrac{\nabla_x^{k_2+1} \nabla_{\omega}^{l-1} f_R^\eps}{M_0}}_\M \right)
			   		\nm{\nabla_\omega (\tfrac{\nabla_x^k \nabla_{\omega}^l f_R^\eps}{M_0})}_\M
	  \\\no
	  & + \frac{1}{\eps} \sum_{\scrpt{0\le k_2 \le k}{l_1\le l-1, l_2 \le l-2}} \nm{\nabla_{\omega} (\tfrac{\nabla_x^{k_2} \nabla_{\omega}^{l_2} f_R^\eps}{M_0})}_\M \nm{\nabla_{\omega} (\tfrac{\nabla_x^k \nabla_{\omega}^{l_1} f_R^\eps}{M_0})}_\M
	  \\\no
		& + \frac{1}{\eps} \sum_{\scrpt{0\le k_2 \le k-1}{l_2 \le l-1}} \nm{\nabla_{\omega} (\tfrac{\nabla_x^{k_2} \nabla_{\omega}^{l_2} f_R^\eps}{M_0})}_\M \nm{\nabla_{\omega} (\tfrac{\nabla_x^k \nabla_{\omega}^l f_R^\eps}{M_0})}_\M
		+ \tfrac{1}{\sqrt\eps} \sum_{l'\le l-1} \nm{\nabla_{\omega} (\tfrac{\nabla_x^k \nabla_{\omega}^{l'} f_R^\eps}{M_0})}_\M
	  \\\no
	  &	+ \nm{\nabla_x^k v_R^\eps}_{L^2_x} \nm{\tfrac{\nabla_x^k \nabla_{\omega}^l f_R^\eps}{M_0}}_\M
	  	+ \nm{\nabla_x^{k +1} v_R^\eps}_{L^2_x} \nm{\nabla_{\omega} (\tfrac{\nabla_x^k \nabla_{\omega}^l f_R^\eps}{M_0})}_\M.
 	\end{align}


\subsection{Closing the Uniform A Priori Estimates} 
\label{sub:2-4}



Define the energy functionals with respect to some small parameter $\eta_0$ satisfying $C_0 \eta_0^{1/2} \le 1/6$,
	\begin{align}
	  E_{\eta_0} = & \nm{v_R^\eps}_{H^s_x}^2 + \sum_{0\le k+l =m \le s} \eta_0^m \nm{\tfrac{\nabla_x^k \nabla_{\omega}^l f_R^\eps}{M_0}}_\M^2, \\\no
	  D_{\eta_0} = & \nm{\nabla_x v_R^\eps}_{H^s_x}^2 + \frac{1}{\eps} \sum_{0\le k+l =m \le s} \eta_0^m \nm{\nabla_{\omega} (\tfrac{\nabla_x^k \nabla_{\omega}^l f_R^\eps)}{M_0}}_\M^2.
	\end{align}
Note that $E_{\eta_0},\ D_{\eta_0}$ are equivalent to energy functionals $E,\ D$ defined in Lemma \ref{lemm:apriori-uniform}, respectively, when $\eta_0$ had been fixed.

Choose $\eps_0 \le \eta_0^{s+1}$, then for any $\eps < \eps_0$, it follows $\eps^{1/2} < \eta_0^{(s+1)/2}$. Thus the previous higher-order derivatives estimate \eqref{esm:derivatives-v} can be recast as
	\begin{align}\label{esm:derivatives-v-re}
		\tfrac{1}{2} \tfrac{\d}{\d t} \nm{\nabla_x^s v_R^\eps}_{L^2_x}^2 + \nm{\nabla_x^{s+1} v_R^\eps}_{L^2_x}^2
		\le & C_0 (E_{\eta_0} + \sqrt\eps E_{\eta_0} D_{\eta_0}^{1/2} + \eta_0^{1/2} D_{\eta_0}).
		\\\no
		\le & C_0 (E_{\eta_0} + 1) + C_0 (\eta_0^{1/2} + \eps E_{\eta_0}) D_{\eta_0},
	\end{align}
where we have used the H\"older inequality $\sqrt\eps E_{\eta_0} D_{\eta_0}^{1/2} \le \tfrac{1}{2} (E_{\eta_0} + \eps E_{\eta_0} D_{\eta_0})$.

By the basic fact
	\begin{align*}
		\eta_0^{s/2} \sum_{0 \le s_2 \le s-1} \nm{\nabla_{\omega} (\tfrac{\nabla_x^{s_2} f_R^\eps}{M_0})}_\M
		\le \eta_0^{1/2} \sum_{0 \le s_2 \le s-1} \eta_0^{s_2/2}\nm{\nabla_{\omega} (\tfrac{\nabla_x^{s_2} f_R^\eps}{M_0})}_\M,
	\end{align*}
we can rewrite the estimate on the pure spatial variables \eqref{esm:derivatives-f-x} as
	\begin{align}\label{esm:derivatives-f-x-re}
	  & \tfrac{1}{2} \tfrac{\d}{\d t} \nm{\tfrac{\nabla_x^s f_R^\eps}{M_0}}_\M^2 + \tfrac{1}{\eps} \nm{\nabla_\omega (\tfrac{\nabla_x^s f_R^\eps}{M_0})}_\M^2
	\\\no
	  \le & C_0 \Big\{ E_{\eta_0} + \sqrt\eps E_{\eta_0}^{1/2} D_{\eta_0}^{1/2}
	  		+ \sqrt\eps E_{\eta_0} D_{\eta_0}^{1/2}
	  		+ \eps E_{\eta_0}^{1/2} D_{\eta_0} + \eta_0^{1/2} D_{\eta_0}
	  		+ \eta_0^{s/2} D_{\eta_0}^{1/2} 
	  		+ \sqrt\eps \eta_0^{s/2} D_{\eta_0} \Big\}
	\\\no
	 	\le & C_0 (E_{\eta_0} + 1) + C_0 (\eps + \eta_0^{1/2} + \eps E_{\eta_0}) D_{\eta_0},
 	\end{align}
where we have used the H\"older inequality again.

Performing a similar argument yields that the estimates for mixed derivatives \eqref{esm:derivatives-f-xw} can be rewritten as
	\begin{align}\label{esm:derivatives-f-xw-re}
	  & \tfrac{1}{2} \tfrac{\d}{\d t} \nm{\tfrac{\nabla_x^k \nabla_{\omega}^l f_R^\eps}{M_0}}_\M^2 + \tfrac{1}{\eps} \nm{\nabla_\omega (\tfrac{\nabla_x^k \nabla_{\omega}^l f_R^\eps}{M_0})}_\M^2
	  \\\no
	\le & C_0 \Big\{ E_{\eta_0} + \sqrt\eps E_{\eta_0}^{1/2} D_{\eta_0}^{1/2}
			+ \sqrt\eps E_{\eta_0} D_{\eta_0}^{1/2} + \eps E_{\eta_0}^{1/2} D_{\eta_0}
			+ \eta_0 D_{\eta_0} + \eta_0^{(s+1)/2} D_{\eta_0}^{1/2}
			+ \sqrt\eps \eta_0^{s/2} D_{\eta_0} \Big\}
		\\\no
	\le & C_0 (E_{\eta_0} + 1) + C_0 (\eps + \eta_0 + \eps E_{\eta_0}) D_{\eta_0}.
 	\end{align}

Combining the above higher-order derivatives estimates \eqref{esm:derivatives-v-re}--\eqref{esm:derivatives-f-xw-re} together, we are led to the \emph{a priori} estimate uniformly-in-$\eps$, namely, 
	\begin{align}\label{esm:apriori-uniform-eta}
		\tfrac{1}{2} \tfrac{\d}{\d t} E_{\eta_0} + D_{\eta_0} \le C_0 (E_{\eta_0} + 1) + C_0 (\eps + \eta_0^{1/2} + \eps E_{\eta_0}) D_{\eta_0},
	\end{align}
which is equivalent to the formulation stated in \eqref{esm:apriori-uniform}, provided that the small parameter $\eta_0$ had been fixed.

At the end, recalling the definitions of $F_0$, $B(v_0)$ and $f_0$, we mention that $C_0$ can be chosen as some constant depending only on the values of $\|(\rho_0, \Omega_0, v_0)\|_{L^\infty(0,T;\ H^{s+4}_x)}$, and hence depending on the initial values of $\|(\rho_0^{in}, \Omega_0^{in}, v_0^{in})\|_{H^{s+4}_x}$. This completes the whole proof of Lemma \ref{lemm:apriori-uniform}. \qed


\section{Local Existence of Remainder Equation} 
\label{sec:local_existence_of_remainder_equation}



In this section we study the local existence of remainder equation \eqref{eq:remainder}, i.e.,
  \begin{align}
    \begin{cases}
      \p_t f + u_0 \nabla_x f + \nabla_{\omega}\cdot(\mathcal{F}_0 f)
      = \frac{1}{\eps}\mathcal{L}_{\Omega_0} f - \frac{1}{\sqrt\eps} h_0 - h_1 ,
    \\[3pt]
      Re (\p_t v + v_0 \cdot \nabla_x v + v \cdot \nabla_x v_0 + \sqrt\eps v \cdot \nabla_x v)
      + b \nabla_x\cdot Q_{f} = -\nabla_x p + \Delta_x v,
    \\[3pt]
      \nabla_x \cdot v =0,
    \end{cases}
  \end{align}
where we have dropped the sub- and superscripts $\eps,\ R$ for brevity. We firstly state the main result of this section.
	\begin{lemma}[Local Existence] \label{lemm:local-Re}
		There exists some $T_* > 0$, such that the Cauchy problem of the remainder equation \eqref{eq:remainder} admits a unique solution $(v, f)$ satisfying that
			\begin{align*}
				v(t,x) \in\ & L^\infty(0, T_*;\ H^s_x) \cap L^2(0, T_*;\ H^{s+1}_x), \\
				f(t,x,\omega) \in\ & L^\infty(0, T_*;\ H^s_{x,\omega}) \cap L^2(0, T_*;\ H^{s+1}_{x,\omega}).
			\end{align*}

	\end{lemma}

The proof proceeds following an iteration scheme. Note that the first equation is linear with respect to the distribution function $f$, so given a $v_n$, we can determine a $f_n$ by using the first equation; Next combining the given $v_n$ and the determined $f_n$, we can determine $v_{n+1}$ by solving the equation of velocity field. More precisely, the $(n+1)$-th step iteration scheme can be constructed as follows,
  \begin{align} \label{eq:iteration} \tag*{(Re)$_{n+1}$}
    \begin{cases}
      \p_t f_n + u_0 \nabla_x f_n + \nabla_{\omega}\cdot(\mathcal{F}_0 f_n)
      = \frac{1}{\eps}\mathcal{L}_{\Omega_0} f_n - \frac{1}{\sqrt\eps} h_0 - h_1(v_n, f_n),
    \\[3pt]
      \p_t v_{n+1} + v_0 \cdot \nabla_x v_{n+1} + v_{n+1} \cdot \nabla_x v_0 + \sqrt\eps v_n \cdot \nabla_x v_{n+1} \\[3pt]
      \hspace*{5cm} + \nabla_x\cdot Q_{f_n} = -\nabla_x p_{n+1} + \Delta_x v_{n+1},
    \\[3pt]
      \nabla_x \cdot v_{n+1} =0,
    \end{cases}
  \end{align}
where $h_1(v_n, f_n) = v_n \cdot \nabla_x (f_0 + \sqrt\eps f_n) + \nabla_\omega \cdot [\mathcal{P}_{\omega^\perp}B(v_n) \omega (f_0 + \sqrt\eps f_n)]$. The initial data are given by $v_{n+1}^{in} = v_R^{\eps, in},\ f_{n+1}^{in} = f_R^{\eps, in}$, and the iteration begins with a given $(v_R^\eps)_0$.

Our task is to prove that the sequence $\{v_n\}_{n=0}^{\infty}$ is a Cauchy sequence, then the limit point $v$ will be the solution we seek, which also yields a unique distribution function $f$, correspondingly. For that, we define the solution set $\mathcal{S}_T(A)$ by
	\begin{align}
		\mathcal{S}_T(A)
		= \begin{Bmatrix}
				v \in L^\infty(0, T;\ H^s_x) \cap L^2(0, T;\ H^{s+1}_x) \\[3pt]
				\|v\|_{L^\infty(0, T;\ H^s_x)}^2 + \|v\|_{L^2(0, T;\ H^{s+1}_x)}^2 \le A
			\end{Bmatrix},
	\end{align}
where the constant $A$ will be determined later.

\subsection{Uniform bound in a large norm} 
\label{sub:uniform_bound_in_large_norm}


We will prove the uniform-in-$n$ boundedness of $v_n$, namely, assume $v_n \in \mathcal{S}_{T_*}(A)$, it requires to prove $v_{n+1} \in \mathcal{S}_{T_*}(A)$ as well.

Introduce the following notations:
	\begin{align}
		& E_{v,n} = \|v_n\|_{H^s_x}^2, \quad D_{v,n} = \|\nabla_x v_n\|_{H^s_x}^2, \\
		E_{f,n} = \sum_{0\le k+l =m \le s} & \eta_0^m \nm{\tfrac{\nabla_x^k \nabla_{\omega}^l f_n}{M_0}}_\M^2, \quad
	  D_{f,n} = \frac{1}{\eps} \sum_{0\le k+l =m \le s} \eta_0^m \nm{\nabla_{\omega} (\tfrac{\nabla_x^k \nabla_{\omega}^l f_n}{M_0})}_\M^2,
	\end{align}	
and
	\begin{align}
	  \mathcal{E}_{v,n}(t) = \sup_{s \in (0,t]} E_{v,n}(s) + \int_0^t D_{v,n}(s) \d s, \quad
	  \mathcal{E}_{f,n}(t) = \sup_{s \in (0,t]} E_{f,n}(s) + \int_0^t D_{f,n}(s) \d s.
	\end{align}

Performing an almost exactly the same process as that in the last section, we can infer that
	\begin{align} \label{esm:derivatives-f-x-n}
	  & \tfrac{1}{2} \tfrac{\d}{\d t} \nm{\tfrac{\nabla_x^s f_n}{M_0}}_\M^2 + \tfrac{1}{\eps} \nm{\nabla_\omega (\tfrac{\nabla_x^s f_n}{M_0})}_\M^2
	\\\no
	  \le & C_0 \Big\{ E_{f,n} + \sqrt\eps E_{f,n}^{1/2} D_{f,n}^{1/2}
	  		+ \eps E_{v,n}^{1/2} D_{f,n}^{1/2} E_{f,n}^{1/2}
	  		+ \sqrt\eps E_{f,n} D_{v,n}^{1/2}
	  		+ \eps E_{f,n}^{1/2} D_{f,n}^{1/2} D_{v,n}^{1/2}
	 \\\no
	  		& \qquad + \eta_0^{1/2} D_{f,n}
	  		+ \eta_0^{s/2} D_{f,n}^{1/2} + \eta_0^{s/2} E_{f,n}^{1/2} E_{v,n}^{1/2}
	  		+ \sqrt\eps \eta_0^{s/2} D_{f,n}^{1/2} D_{v,n}^{1/2} \Big\}
	\\\no
	 	\le & C_0 (1 + \eps E_{v,n} + \eta_0^{s/2} E_{v,n} + \eps D_{v,n}) E_{f,n}
	 			+ C_0 (\eta_0^{s/2} + \eta_0^s D_{v,n})
	 			+ C_0 (\eps + \eta_0^{1/2}) D_{f,n},
 	\end{align}
and
	\begin{align} \label{esm:derivatives-f-xw-n}
	  & \tfrac{1}{2} \tfrac{\d}{\d t} \nm{\tfrac{\nabla_x^k \nabla_{\omega}^l f_n}{M_0}}_\M^2 + \tfrac{1}{\eps} \nm{\nabla_\omega (\tfrac{\nabla_x^k \nabla_{\omega}^l f_n}{M_0})}_\M^2
	\\\no
	\le & C_0 \Big\{ E_{f,n} + \sqrt\eps E_{f,n}^{1/2} D_{f,n}^{1/2}
			+ \eps E_{v,n}^{1/2} E_{f,n}^{1/2} D_{f,n}^{1/2}
			+ \sqrt\eps E_{f,n} D_{v,n}^{1/2} + \eps E_{f,n}^{1/2} D_{f,n}^{1/2} D_{v,n}^{1/2}
		\\\no
			& \qquad + \eta_0 D_{f,n} + \eta_0^{(s+1)/2} D_{f,n}^{1/2}
			+ \eta_0^{s/2} E_{f,n}^{1/2} E_{v,n}^{1/2}
			+ \sqrt\eps \eta_0^{s/2} D_{f,n}^{1/2} D_{v,n}^{1/2} \Big\}
	\\\no
	\le & C_0 (1 + \eps E_{v,n} + \eta_0^{s/2} E_{v,n} + \eps D_{v,n}) E_{f,n}
	 			+ C_0 (\eta_0^{s/2} + \eta_0^s D_{v,n})
	 			+ C_0 (\eps^{1/2} + \eta_0^{1/2}) D_{f,n},
 	\end{align}
which together entail that
	\begin{align} \label{esm:derivatives-f-n}
			\tfrac{1}{2} \tfrac{\d}{\d t} E_{f,n} + D_{f,n} 	
		\le & C_0 (1 + \eps E_{v,n} + \eta_0^{s/2} E_{v,n} + \eps D_{v,n}) E_{f,n} \\\no
	 		& + C_0 (\eta_0^{s/2} + \eta_0^s D_{v,n})
	 			+ C_0 (\eps^{1/2} + \eta_0^{1/2}) D_{f,n}.
	\end{align}
Thus by recalling the assumption of iteration scheme $v_n \in \mathcal{S}_{T_*}(A)$, and choosing some small $\eta_0$ as before, we can infer from the Gr\"onwall inequality that
	\begin{align}
	  \mathcal{E}_{f,n}(t)
	  \le & \exp \left\{ 2C_0 \int_0^t [1 + \eps E_{v,n}(s) + \eta_0^{s/2} E_{v,n}(s) + \eps D_{v,n}(s)] \d s \right\} \\\no
	  & \times \left\{ E_{f,n}(0) + 2C_0 \eta_0^{s/2} \int_0^t (1 + \eta_0^{s/2} D_{v,n}(s)) \d s \right\}
	\\\no
	  \le & \exp \left\{ 2C_0 t [1 + (\eps + \eta_0^{1/2}) A] + 2C_0 \eps A \right\}
	  			\left\{ E_{f,n}(0) + 2C_0 \eta_0^{1/2} [t + \eta_0^{1/2} A] \right\}.
	\end{align}

Taking some $T_0 > 0$ satisfying
	\begin{align*}
	  T_0 \le \tfrac{\ln 2}{4C_0 [1 + A]},
	\end{align*}
then $\exp \{2C_0 T_0 [1 + (\eps + \eta_0^{1/2}) A] \} \le \tfrac{\ln 2}{2} $. As a consequence, some sufficiently small $\eps_0$ satisfying $ 2 C_0 A \eps_0 \le \tfrac{\ln 2}{2}$ yields for any $\eps < \eps_0$ and $ t < T_0$, that
	\begin{align}\label{esm:energy-f-n}
	  \mathcal{E}_{f,n}(t)
	  \le 2 \left\{ E_{f,n}(0) + \ln 2 \right\}.
	\end{align}

On the other hand, a similar argument as that in \eqref{esm:derivatives-v} gives that
	\begin{align}
		& \tfrac{1}{2} \tfrac{\d}{\d t} \nm{\nabla_x^s v_{n+1}}_{L^2_x}^2 + \nm{\nabla_x^{s+1} v_{n+1}}_{L^2_x}^2
	\\\no
		\ls_{_{C_0}} & \nm{\nabla_x^s v_{n+1}}_{L^2_x}^2
			 + \eps^{1/2} \nm{\nabla_x^s v_n}_{L^2_x} \nm{\nabla_x v_{n+1}}_{H^s_x} \nm{\nabla_x^s v_{n+1}}_{L^2_x}
			 + \nm{\nabla_x^{s+1} v_{n+1}}_{L^2_x} \nm{\nabla_{\omega}(\tfrac{\nabla_x^s f_n}{M_0})}_\M,
	\end{align}
which yields
	\begin{align}\label{esm:derivatives-v-n}
			& \tfrac{1}{2} \tfrac{\d}{\d t} E_{v,n+1} + D_{v,n+1} \\\no
		\le & C_0 \left\{ E_{v,n+1} + \eps^{1/2} E_{v,n}^{1/2} E_{v,n+1}^{1/2} D_{v,n+1}^{1/2}
					 + \eta_0^{1/2} D_{v,n+1}^{1/2} D_{f,n}^{1/2} \right\}  \\\no
		\le &	C_0 E_{v,n+1} (1 + \eps^{1/2} E_{v,n}) + C_0 \eta_0^{1/2} D_{f,n}
					+ C_0 (\eps^{1/2} + \eta_0^{1/2}) D_{v,n+1}.
	\end{align}	
Applying the Gr\"onwall inequality implies that
	\begin{align}
	  	\mathcal{E}_{v,n+1}(t)
	  \le & \exp(2C_0 \int_0^t (1 + \eps^{1/2} E_{v,n}(s) ) \d s) \left\{ E_{v,n+1}(0) + 2C_0 \eta_0^{1/2} \int_0^t D_{f,n}(s) \d s \right\}
	\\\no
		\le & 2 \left\{ E_{v,n+1}(0) + 2C_0 \eta_0^{1/2} \mathcal{E}_{f,n}(t) \right\} \\\no
		\le & 2 \left\{ E_{v,n+1}(0) + E_{f,n}(0) + \ln 2 \right\}.
	\end{align}

Therefore, choose
	\begin{align}
	  A \ge 2 \left( \|v_R^{\eps, in}\|^2_{H^s_x} + \|f_R^{\eps, in}\|^2_{H^s_{x,\omega}} + \ln 2 \right),  \quad
	  T_0 \le \tfrac{\ln 2}{4C_0 [1 + A]},
	\end{align}
then $\mathcal{E}_{v,n+1} \le A$, consequently we have $v_{n+1} \in \mathcal{S}_{T_*}(A)$ for $T_* \le T_0$.

\subsection{Contraction in a low norm} 
\label{sub:contraction_in_a_low_norm}



We are left to prove in this subsection that the contraction of the sequence $\{v_n\}_{n=0}^{\infty}$ in $L^2_x$ spaces. Setting $\phi_{n+1} = v_{n+1} - v_n$ and $\psi_n = f_n - f_{n-1}$, then equation \ref{eq:iteration} yields that
  \begin{align} \label{eq:itera-convg}
    \begin{cases}
      \p_t \psi_n + u_0 \nabla_x \psi_n + \nabla_{\omega}\cdot(\mathcal{F}_0 \psi_n)
      = \frac{1}{\eps}\mathcal{L}_{\Omega_0} \psi_n - h_1(\phi_n, f_n) \\[3pt]
      	\hspace*{4.2cm} - \sqrt\eps v_{n-1} \cdot \nabla_x \psi_n - \sqrt\eps \nabla_\omega \cdot (\mathcal{P}_{\omega^\perp}B(v_{n-1}) \omega \psi_n),
    \\[3pt]
      \p_t \phi_{n+1} + v_0 \cdot \nabla_x \phi_{n+1} + \phi_{n+1} \cdot \nabla_x v_0 + \sqrt\eps v_n \cdot \nabla_x \phi_{n+1} + \sqrt\eps \phi_n \cdot \nabla_x v_n \\[3pt]
      \hspace*{5cm} + \nabla_x\cdot Q_{\psi_n} = -\nabla_x (p_{n+1}-p_n) + \Delta_x \phi_{n+1},
    \\[3pt]
      \nabla_x \cdot \phi_{n+1} =0.
    \end{cases}
  \end{align}

Denote
	\begin{align}
		& E_{\phi,n} = \|\phi_n\|_{L^2_x}^2, \qquad D_{v,n} = \|\nabla_x \phi_n\|_{L^2_x}^2, \\
		& E_{\psi,n} = \nm{\tfrac{\psi_n}{M_0}}_\M^2, \quad
	  D_{f,n} = \frac{1}{\eps} \nm{\nabla_{\omega} (\tfrac{\psi_n}{M_0})}_\M^2,
	\end{align}	
and
	\begin{align}
	  \mathcal{E}_{n+1}(t) = \sup_{s \in (0,t]} [E_{\phi,n+1}(s) + E_{\psi,n}(s)] + \int_0^t [D_{\phi,n+1}(s) + D_{\psi,n}(s)] \d s.
	\end{align}

The similar argument as before will lead to
	\begin{align} \label{esm:f-convg}
			& \tfrac{1}{2} \tfrac{\d}{\d t} E_{\psi,n} + D_{\psi,n} \\\no
	  \le & C_0 \Big\{ E_{\psi,n} + \sqrt\eps E_{\psi,n}^{1/2} D_{\psi,n}^{1/2}
	  		+ E_{\phi,n}^{1/2} E_{\psi,n}^{1/2}
	  		+ \sqrt\eps D_{\phi,n}^{1/2} D_{\psi,n}^{1/2} \\\no
	  		& \qquad + \eps E_{f,n}^{1/2} D_{\phi,n}^{1/2} D_{\psi,n}^{1/2}
	  		+ \sqrt\eps E_{v,n-1}^{1/2} E_{\psi,n}
	  		+ \eps D_{v-1,n}^{1/2} E_{\psi,n}^{1/2} D_{\psi,n}^{1/2} \Big\},
	\end{align}
and
	\begin{align}\label{esm:v-convg}
			\tfrac{1}{2} \tfrac{\d}{\d t} E_{\phi,n+1} + D_{\phi,n+1}
		\le C_0 \left\{ E_{\phi,n+1} + \eps^{1/2} D_{v,n}^{1/2} E_{\phi,n}^{1/2} E_{\phi,n+1}^{1/2} + \eps^{1/2} D_{\phi,n+1}^{1/2} D_{\psi,n}^{1/2} \right\}.
	\end{align}	
Integrating the above two inequalities, and noticing the uniform bounds $\mathcal{E}_{v,n},\ \mathcal{E}_{f,n} \le A$, we can finally obtain, for any $t \in [0, T_*]$, that
	\begin{align}
	  \left\{ 1 - C_0 (T_* + \eps^{1/2} +  \eps^{1/2} A^{1/2}) \right\} \mathcal{E}_{n+1}(t)
	  \le C_0 (T_* + \eps^{1/2} +  \eps^{1/2} A^{1/2}) \mathcal{E}_n(t),
	\end{align}
where we have used the H\"older inequalities
	\begin{align*}
	  & \eps \int_0^t E_{f,n}^{1/2} D_{\phi,n}^{1/2} D_{\psi,n}^{1/2} \d s
	  	\ls \eps A^{1/2} (\mathcal{E}_{n+1}(t) + \mathcal{E}_n(t)),
	\\
	  & \eps \int_0^t D_{v-1,n}^{1/2} E_{\psi,n}^{1/2} D_{\psi,n}^{1/2} \d s
	  	\ls \eps \mathcal{E}_{n+1}^{1/2}(t) \left\{ \int_0^t D_{v-1,n}\d s \right\}^{1/2} \left\{ \int_0^t D_{\psi,n}\d s \right\}^{1/2}
	  \ls \eps A^{1/2} \mathcal{E}_{n+1}(t),
	\\
		& \eps^{1/2} \!\! \int_0^t \! D_{v,n}^{1/2} E_{\phi,n}^{1/2} E_{\phi,n+1}^{1/2} \d s
			\ls \eps^{1/2} \!\left\{ \int_0^t D_{v,n}^{1/2} \d s \right\} \mathcal{E}_{n+1}^{1/2} (t)
			\mathcal{E}_{n}^{1/2} (t)
			\ls \eps^{1/2} A^{1/2} t^{1/2} (\mathcal{E}_{n+1}(t) + \mathcal{E}_n(t)).
	\end{align*}
Recalling the same choose of $A,\ \eps_0$ and $T_*$ as before yields that
	\begin{align*}
		C_0 (T_* + \eps^{1/2} +  \eps^{1/2} A^{1/2}) \le \tfrac{\ln 2}{4(1+A)} + \tfrac{1}{6} + \tfrac{\sqrt{\ln 2}}{2C_0} \le \tfrac{1}{3},
	\end{align*}
which leads us to the contraction
	\begin{align}
		\mathcal{E}_{n+1}(t) \le \tfrac{1}{2} \mathcal{E}_n(t),
	\end{align}
Thus combining with a standard compactness argument, we can get a unique local solution of remainder equation \eqref{eq:remainder} as stated in Lemma \ref{lemm:local-Re}.

\section{Conclusions}

In this paper, we first prove the existence of the local classical solution of the macroscopic SOH-NS model, then use Hilbert expansion in term of $\sqrt{\eps}$ to justify the hydrodynamic limit from SOK-NS to SOH-NS. This provides a first analytically rigorous justification of the modeling and asymptotic analysis in \cite{DMVY-2017-arXiv}. Many analytic questions are left to be answered yet. The first one is about the global in time well-posedness of SOH-NS, at least near equilibrium. Obvious stationary solutions of SOH-NS model consists  of uniform (space independent) swimmer density and mean orientation fields as well as fluid velocity and pressure. However, linear stability analysis done in \cite{DMVY-2017-arXiv} illustrates that these states are mostly unstable to perturbations of an aligned state. This is also consistent with previous studies in \cite{SS-2008-PF}, \cite{SS-2008-PRL} for passive suspensions. This basically implies that the global in time existence around these equilibriums are of little hope. Instead, we should look for the stability around other states, say the isotropic states. This will be left for our future work.

\appendix

\section{The Stereographic Projection for the SOH-NS System}
\label{sec:SPT-3D}

In this section, we prove that by utilizing the stereographic projection transform to deal with the geometric constraint $|\Omega| = 1$, the SOH-NS system \eqref{SOH-NS} in $\R^3$ can be rewritten as the system \eqref{SOH-NS-3D-SPT-Smp}. The following chain rules of the derivatives are frequently used.

\begin{lemma}
  For $n=3$, we let $\widehat{\rho} = \ln \rho$, and $\Omega = \big{(} \tfrac{2 \phi}{W} , \tfrac{2 \psi}{W} , \tfrac{\phi^2 + \psi^2 - 1}{W} \big{)}$ be the stereographic projection transform, where $W = 1 + \phi^2 + \psi^2$ and $(\phi, \psi) \in \R^2$. Then the SOH-NS system \eqref{SOH-NS} can be rewritten as the system \eqref{SOH-NS-3D-SPT-Smp}.
\end{lemma}

\begin{proof}
  By direct calculations,
  \begin{equation*}
    \begin{aligned}
      \Omega_{\phi \phi} = & \big{(} \tfrac{4 \phi ( \phi^2 - 3 \psi^2 - 3 )}{W^3}, \tfrac{ 4 \psi ( 3 \phi^2 - \psi^2 - 1 ) }{W^3} , - \tfrac{4 ( 3 \phi^2 - \psi^2 - 1 )}{W^3} \big{)} \,, \\
      \Omega_{\phi \psi} = & \big{(} \tfrac{4 \psi ( 3 \phi^3 - \psi^2 - 1 )}{W^3} , \tfrac{4 \phi ( 3 \psi^2 - \phi^2 - 1 )}{W^3} , - \tfrac{16 \phi \psi}{W^3} \big{)} \,, \\
      \Omega_{\psi \psi} = & \big{(} \tfrac{4 \phi ( 3 \psi^2 - \phi^2 - 1 )}{W^3} , \tfrac{4 \psi ( \psi^2 - 3 \phi^2 - 3  )}{W^3} , - \tfrac{4 ( 3 \psi^2 - \phi^2 - 1 )}{W^3} \big{)} \,.
    \end{aligned}
  \end{equation*}
  Thus we can derive the following relations:
  \begin{equation}\label{Codnt-Rel-SPT}
    \begin{aligned}
      & |\Omega_\phi| = |\Omega_\psi| = \tfrac{2}{W}\,, \ \Omega_\phi \cdot \Omega_\psi = \Omega_\phi \cdot \Omega = \Omega_\psi \cdot \Omega = 0 \,, \\
      & \Omega_\phi \cdot \Omega_{\phi \phi} = - \tfrac{8 \phi }{W^3} \,, \ \Omega_\phi \cdot \Omega_{\phi \psi} = - \tfrac{8 \psi}{W^3}\,, \ \Omega_\phi \cdot \Omega_{\psi \psi} = \tfrac{16 \phi}{W^3} \,, \\
      & \Omega_\psi \cdot \Omega_{\phi \phi} = \tfrac{8 \psi}{W^3}\,, \ \Omega_\psi \cdot \Omega_{\phi \psi} = - \tfrac{8 \phi}{W^3} \,, \ \Omega_\psi \cdot \Omega_{\psi \psi} = - \tfrac{8 \psi}{W^3} \,, \\
      & P_{\Omega^\bot} a = \tfrac{1}{4} W^2 ( \Omega_\phi \cdot a ) \Omega_\phi + \tfrac{1}{4} W^2 ( \Omega_\psi \cdot a ) \Omega_\psi
    \end{aligned}
  \end{equation}
  for all $a \in \R^3$. Based on the above relations \eqref{Codnt-Rel-SPT}, we now can derive the system \eqref{SOH-NS-3D-SPT-Smp}.

  {\em Step 1. Derivation of $\rho$-equation.} For the second term $\nabla_x \cdot ( \rho U )$, we derive that
  \begin{equation*}
    \begin{aligned}
      \nabla_x \cdot ( \rho U ) = & U \cdot \nabla_x \rho + \rho \nabla_x \cdot ( a c_1 \Omega + v ) \\
      = & \rho U \cdot \nabla_x \widehat{\rho} + a c_1 \rho ( \Omega_\phi \cdot \nabla_x \phi + \Omega_\psi \cdot \nabla_x \psi ) \,,
    \end{aligned}
  \end{equation*}
  where we make use of the facts $\nabla_x \cdot v = 0$ and $\widehat{\rho} = \ln \rho$. Then, the first equation of the system \eqref{SOH-NS-3D-SPT-Smp} is immediately derived from the first $\rho$-equation of the SOH-NS system \eqref{SOH-NS} and the above equality.

  {\em Step 2. Derivation of $\Omega$-equation.} For the terms in the left hand side of the second $\Omega$-equation, we calculate that by the last relation of \eqref{Codnt-Rel-SPT} and the chain rules of differentiation
  \begin{equation*}
    \begin{aligned}
      \rho ( \partial_t \Omega + V \cdot \nabla_x \Omega ) + \tfrac{a}{\kappa} P_{\Omega^\bot} \nabla_x \rho = & \rho ( \partial_t \phi + V \cdot \nabla_x \phi +  \tfrac{a}{4 \kappa} W^2 \Omega_\phi \cdot \nabla_x \widehat{\rho} ) \Omega_\phi \\
      & + \rho ( \partial_t \psi + V \cdot \nabla_x \psi + \tfrac{a}{4 \kappa} W^2 \Omega_\psi \cdot \nabla_x \widehat{\rho} ) \Omega_\psi \,,
    \end{aligned}
  \end{equation*}
  and
  \begin{equation*}
    \begin{aligned}
      \rho P_{\Omega^\bot}  ( \widetilde{\lambda} S(v) + A(v)) \Omega = & \tfrac{\rho}{4} W^2 [ ( \widetilde{\lambda} S(v) + A(v) ) : \Omega_\phi \otimes \Omega ] \Omega_\phi \\
      & + \tfrac{\rho}{4} W^2 [ ( \widetilde{\lambda} S(v) + A(v) ) : \Omega_\psi \otimes \Omega ] \Omega_\psi \,,
    \end{aligned}
  \end{equation*}
  and
  \begin{equation*}
    \gamma P_{\Omega^\bot} \Delta_x ( \rho \Omega ) = \tfrac{\gamma}{4} W^2 [ \Omega_\phi \cdot \Delta_x (\rho \Omega) ] \Omega_\phi + \tfrac{\gamma}{4} W^2 [ \Omega_\psi \cdot \Delta_x (\rho \Omega) ] \Omega_\psi \,.
  \end{equation*}
  For the term $\Delta_x (\rho \Omega)$,  we derive from the chain rules of differentiation that
  \begin{equation*}
    \begin{aligned}
      \Delta_x ( \rho \Omega ) = & \Delta_x \rho \Omega + 2 \nabla_x \rho \cdot \nabla_x \phi \Omega_\phi + 2 \nabla_x \rho \cdot \nabla_x \psi \Omega_\psi + \rho \Delta_x \phi \Omega_\phi \\
      & + \rho \Delta_x \psi \Omega_\psi + \rho |\nabla_x \phi|^2 \Omega_{\phi \phi} + 2 \rho \nabla_x \phi \cdot \nabla_x \psi \Omega_{\phi \psi} + \rho |\nabla_x \psi|^2 \Omega_{\psi \psi} \,.
    \end{aligned}
  \end{equation*}
  Thus it is easy to be implied that
  \begin{equation*}
    \begin{aligned}
      \rho ( \partial_t \phi + V \cdot \nabla_x \phi + \mathcal{H}_\phi ( \widehat{\rho}, \phi, \psi, v ) ) \Omega_\phi +  \rho ( \partial_t \psi + V \cdot \nabla_x \psi + \mathcal{H}_\psi ( \widehat{\rho}, \phi, \psi, v ) ) \Omega_\psi = 0\,.
    \end{aligned}
  \end{equation*}
  Then by the relation $ \Omega_\phi \cdot \Omega_\psi = 0 $, we derive from dot product with $\Omega_\phi$ and $\Omega_\psi$ in the above equation that the second $\phi$-equation and the third $\psi$-equation of the system \eqref{SOH-NS-3D-SPT-Smp} hold, respectively.

  {\em Step 3. Derivation of $v$-equation.} We merely need to compute the term  $ - b \nabla_x \cdot ( \rho \mathcal{Q} (\Omega) ) $. In fact, by the chain rules of differentiation, we directly calculate that
  \begin{equation*}
    \begin{aligned}
      - b \nabla_x \cdot ( \rho \mathcal{Q} (\Omega) ) = & - b e^{\widehat{\rho}} \Big{[} \nabla_x \widehat{\rho} + c_4 ( \Omega_\phi \cdot \nabla_x \phi + \Omega_\psi \cdot \nabla_x \psi ) \Omega \\
      & \qquad \qquad + c_4 \big{(} ( \Omega \cdot \nabla_x \phi ) \Omega_\phi + ( \Omega \cdot \nabla_x \psi ) \Omega_\psi \big{)} \Big{]} \\
      \equiv & \ - b e^{\widehat{\rho}} \mathcal{G} (\widehat{\rho}, \phi, \psi) \,.
    \end{aligned}
  \end{equation*}
  Consequently, we gain the forth $v$-equation of the system \eqref{SOH-NS-3D-SPT-Smp} and we finish the proof of this lemma.

\end{proof}





\end{document}